\documentclass[final, 10pt]{amsart}
\usepackage[foot]{amsaddr}
\usepackage{mathrsfs}
\usepackage{mathtools}
\usepackage{amsmath}
\usepackage{amsthm}
\usepackage{amssymb}
\usepackage{amsfonts}
\usepackage{graphicx}
\usepackage{enumitem}
\usepackage{tikz}
\usepackage{tkz-tab}
\usepackage{nicematrix}
\NiceMatrixOptions{%
    cell-space-limits=1pt,
    columns-width=10pt
}
\usepackage[final]{hyperref}
\usepackage{xcolor}
\usepackage{dsfont}
\usepackage{cases}

\definecolor{jade}{rgb}{0.0, 0.66, 0.42}
\definecolor{cerulean}{rgb}{0.0, 0.48, 0.65}
\definecolor{ruddy}{rgb}{1.0, 0.0, 0.16}
\definecolor{PSprune}{RGB}{99, 0, 60}
\definecolor{PSgrisfonce}{RGB}{49, 62, 72}
\definecolor{PSgris}{RGB}{124, 135, 143}
\definecolor{PSgrisclair}{RGB}{213, 218, 223}
\definecolor{PSrouge}{RGB}{198, 11, 70}
\definecolor{PSrose}{RGB}{237, 20, 91}
\definecolor{PSorange}{RGB}{238, 52, 35}
\definecolor{PSorangeclair}{RGB}{243, 115, 32}
\definecolor{PSviolet}{RGB}{124, 42, 144}
\definecolor{PSmauve}{RGB}{125, 106, 175}
\definecolor{PSmarron}{RGB}{198, 103, 29}
\definecolor{PSjaune}{RGB}{254, 188, 24}
\definecolor{PSbleufonce}{RGB}{0, 78, 125}
\definecolor{PSbleu}{RGB}{14, 135, 201}
\definecolor{PSturquoise}{RGB}{0, 148, 181}
\definecolor{PSlagon}{RGB}{70, 195, 210}
\definecolor{PSvertfonce}{RGB}{0, 128, 122}
\definecolor{PSvert}{RGB}{64, 183, 105}
\definecolor{PSfeuille}{RGB}{140, 198, 62}
\definecolor{PSvertjaune}{RGB}{213, 223, 61}

\theoremstyle{plain}
\newtheorem{theorem}{Theorem}
\newtheorem{proposition}[theorem]{Proposition}
\newtheorem{property}[theorem]{Property}
\newtheorem{corollary}[theorem]{Corollary}
\newtheorem{lemma}[theorem]{Lemma}
\theoremstyle{definition}
\newtheorem{definition}[theorem]{Definition}
\newtheorem{notation}[theorem]{Notation}
\theoremstyle{remark}
\newtheorem{remark}[theorem]{Remark}
\usepackage{BCDG}

\hypersetup{  
  pdftitle={Numerical analysis of a Lax--Wendroff type lattice Boltzmann scheme for the advection equation},  
  pdfsubject={Lattice Boltzmann method},  
  pdfauthor={Bellotti, Caetano, Dubois, Graille},
  breaklinks=true,
  colorlinks=true,
  urlcolor=PSbleu,
  linkcolor=PSbleu,
  citecolor=PSbleu,
}

\begin{document}

\title[%
    Numerical analysis of a lattice Boltzmann scheme%
]{%
    Numerical analysis of a Lax--Wendroff type lattice Boltzmann scheme for the advection equation
}

\author{T.~Bellotti$^1$}
\author{F.~Caetano$^2$}
\author{F.~Dubois$^{3, 2}$}
\author{B.~Graille$^2$}
\email{benjamin.graille@universite-paris-saclay.fr}

\address{$^1$ Universit\'e Paris-Saclay, CNRS, CentraleSup\'elec, Laboratoire EM2C \& F\'ed\'eration de Math\'ematiques de CentraleSup\'elec,
91190, Gif-sur-Yvette, France}
\address{$^2$ Universit\'e Paris-Saclay, CNRS, Laboratoire de math\'ematiques d'Orsay, 91405, Orsay, France}
\address{$^3$ Conservatoire National des Arts et M\'etiers, MACS laboratory, Paris, France}

\date{\today}

\begin{abstract}
A linear monodimensional lattice Boltzmann scheme with three velocities tackling the advection equation is investigated.
We prove that it is second-order accurate and provide necessary and sufficient stability conditions through an explicit formula involving the two relaxation parameters.  
Relying on existing results for finite difference schemes, we prove convergence for smooth initial data with initialization of the non-conserved moments at equilibrium.
Several numerical simulations illustrate these theoretical findings.
\end{abstract}
    
\maketitle


AMS2000: 65M12

%%%%%%%%%%%%%%%%%%%%%%%%%%%%%%%%%%%%%%%%%%%%%%%%%%
% SECTION INTRODUCTION
%%%%%%%%%%%%%%%%%%%%%%%%%%%%%%%%%%%%%%%%%%%%%%%%%%
\section{Introduction}
%%%%%%%%%%%%%%%%%%%%%%%%%%%%%%%%%%%%%%%%%%%%%%%%%%

We investigate a three-velocity lattice Boltzmann scheme, see \cite{succi2001lattice, guo2013lattice} for a general introduction, designed to approximate the solution of the linear advection equation with second-order accuracy.
The scheme under investigation fully aligns with the Lax-Wendroff method \cite{lax1960systems} for a specific choice of free parameters, and retains its centered character in full generality.

Finite-velocity models with three velocities have a long history in the literature; the earliest reference we found is by Broadwell \cite{broadwell1964shock}, where a finite-velocity approximation of the Boltzmann equation (see \cite{Carleman1957ProblmesMD} for an earlier example with two velocities)---reduced to a set of three equations under symmetry assumptions---is introduced.
Almost three decades later, Qian, d'Humi{\`e}res, and Lallemand \cite{qian1992lattice} proposed several lattice Boltzmann schemes, including a three-velocity model, to approximate the Navier-Stokes system. We note that the target equation we address here---the linear advection equation---is indeed simpler than those in the aforementioned references, which focused on either a discrete-velocity Boltzmann equation or the Navier-Stokes system.

In the case of a scalar equation and a three-velocity model considered here, the numerical scheme features four free parameters: two linear equilibria and two relaxation parameters.
In this contribution, the first two are fixed to achieve second-order accuracy for the advection equation. 
The relaxation parameters remain free and solely influence stability, whereas accuracy is uniform in them.
When they both equal one, we obtain a relaxation scheme which is the Lax-Wendroff method.

An essential question for any numerical scheme addressing evolutionary partial differential equations is its \emph{stability}, and lattice Boltzmann methods are no exception.
Most of the present work focuses on this aspect.
Let us therefore review the main facets of and contributions to this topic.

A first approach relies on von Neumann stability analysis, and boils down to considering linear or linearized schemes, applying a spatial Fourier transform, and verifying whether the spectral radius of the resulting matrix is less than one for all frequencies.
This method was adopted in \cite{sterling1996stability, PhysRevE.61.6546}. Given that the matrix size corresponds to the number of discrete velocities, the scheme may include numerous free parameters, and since all frequencies must be considered, this stability criterion is typically verified numerically.
This approach has been widely used in subsequent works. Notably, Ginzburg and collaborators \cite{ginzburg2010optimal} derived partially explicit stability conditions---thus not solely numerically-checked---for schemes with two relaxation parameters whose sum equals two, using the Cohn-Schur algorithm, see \cite{miller_location_1971, strikwerda_finite_2004}.

A second major research stream on the stability of lattice Boltzmann algorithms focuses on weighted  $L^2$ stability and the so-called ``stability structure'' \cite{banda2006stability, junk2009weighted, rheinlander2010stability}.
Although this concept is primarily designed for linearized schemes, it can also yield valuable insights in certain non-linear settings, as demonstrated in \cite{junk2009convergence}.
This approach involves constructing a weighted vector norm to evaluate the discrete distribution functions, ensuring that both the transport and---more critically---the relaxation steps of the algorithm become $L^2$-contractive under this norm.
In this way, the elementary bound $\lVert A^n \rVert\leq \lVert A \rVert^n$, where $A$ represents the numerical scheme, can be effectively employed, as $\lVert A \rVert\leq 1$.

A third set of contributions focuses on $L^{\infty}$-stability and maximum principles, which are particularly effective for scalar target equations and generally more demanding than mere $L^2$ stability.
This stringency enables their direct applicability to non-linear problems.
This approach was pioneered by Dellacherie \cite{dellacherie2014construction} in a linear framework and later extended to non-linear problems through various notions of monotonicity, as seen in \cite{caetano, bellotti2023monotonicity, aregba2024convergence, aregba2026monotonicity}.

We note that the previously discussed stability notions sometimes lead to convergence proofs of the numerical solution to the target equation, as demonstrated in \cite{junk2009convergence, dellacherie2014construction, caetano, bellotti2023monotonicity, aregba2024convergence, aregba2026monotonicity}.
This will also be one of the goals of the present contribution.

Our work conducts a comprehensive $L^2$ stability analysis of the considered algorithm by supplementing the von Neumann analysis with a discussion of eigenvalue multiplicities, in addition to the mere analysis of spectral radii (i.e., their moduli).
We derive explicit necessary and sufficient stability conditions, obtained using the Cohn-Schur algorithm \cite{miller_location_1971, strikwerda_finite_2004}, in terms of the Courant number and the two relaxation parameters of the scheme.

The article is structured as follows.
In Section~\ref{sec:descriptionScheme}, we introduce the three-velocities lattice Boltzmann scheme, whose equilibria are fully determined in Section~\ref{sec:consistency} to achieve second-order consistency with the target equation.
Section~\ref{sec:stability} defines stability and recapitulates the necessary and sufficient stability conditions.
Proofs, which constitute the bulk of the paper, are provided in Section~\ref{sec:proofs}.
In Section~\ref{sec:convergence}, we use known results on finite difference methods to prove convergence of the conserved moment of the lattice Boltzmann scheme for smooth initial data and initialization of the non-conserved moments ``at equilibrium''.
Additionally, we study discrete wave propagation by the scheme for smooth data.
Finally, Section~\ref{sec:numericalExperiments} presents several numerical experiments that showcase the expected wave propagation patterns, convergence for smooth and non-smooth initial data, and empirical second-order convergence of an adaptation of the scheme to the (non-linear) inviscid Burgers' equation.

%%%%%%%%%%%%%%%%%%%%%%%%%%%%%%%%%%%%%%%%%%%%%%%%%%
% SECTION DESCRIPTION
%%%%%%%%%%%%%%%%%%%%%%%%%%%%%%%%%%%%%%%%%%%%%%%%%%
%%%%%%%%%%%%%%%%%%%%%%%%%%%%%%%%%%%%%%%%%%%%%%%%%%
% SECTION DESCRIPTION
%%%%%%%%%%%%%%%%%%%%%%%%%%%%%%%%%%%%%%%%%%%%%%%%%%
\section{Description of the scheme}\label{sec:descriptionScheme}

In this section, the investigated lattice Boltzmann scheme is presented. 
We tackle the approximation of the following Cauchy problem
\begin{align}
    &\partial_{t}\mz(t, x)+V \partial_{x} \mz(t, x) = 0
    &&
    \text{for} \quad t>0, \ x\in\mathbb{R},\label{eq:pde} \\
    &\mz(0, x) = \mz^{\circ}(x)
    &&
    \text{for} \quad x\in\mathbb{R},\label{eq:initialcondition} 
\end{align}
a linear scalar advection equation with constant velocity $V\in\mathbb{R}$.
Space is discretized using a uniform grid $(j \spacestep)_{j\in\mathbb{Z}}$ with $\spacestep>0$, and time relying on $(n \timestep)_{n\in\mathbb{N}}$ with $\timestep>0$.
We introduce the ratio $\lambda := \spacestep/\timestep>0$, which is assumed given and fixed whatever the value of $\spacestep$.

% In this paper, we consider a linear three-velocities monodimensional lattice Boltzmann scheme specified by the following:
% \begin{itemize}
%     \item three discrete velocities $0$, $\lambda$, and $-\lambda$ and the associated particle distribution functions $\fz$, $\fp$, and $\fm$;
%     \item three polynomials defining the three moments denoted by $u$, $v$, and $w$,
%     \begin{equation*}
%         \begin{aligned}
%         P_0 &= 1, & P_1 &= X, & P_2 &= \tfrac{1}{2}X^2,\\
%         \mz &= \fz+\fp+\fm, & \mun &= \lambda \fp - \lambda \fm, & \mdeux &= \tfrac{1}{2}\lambda^2 \fp + \tfrac{1}{2}\lambda^2 \fm;
%         \end{aligned}
%     \end{equation*}
%     \item three equilibrium value of each moment:
%     \begin{equation*}
%         \mzeq(\mz) = \mz, \quad
%         \muneq(\mz) = c \lambda \mz, \quad
%         \mdeuxeq(\mz) = \tfrac{1}{2}d \lambda^2\mz,
%     \end{equation*}
%     where $c$ and $d$ are real values.
%     For the moment, $c$ and $d$ are kept free: their value to ensure consistency with \eqref{eq:pde} is discussed in Section \ref{sec:consistency}.
%     \item Two relaxation parameters for the non-conserved moments $\sun$ and $\sdeux$; and the centered relaxation parameters $\tun=\sun-1$ and $\tdeux=\sdeux-1$.
% \end{itemize}
In this paper, we consider a linear three-velocities monodimensional lattice Boltzmann scheme specified by
three discrete velocities $0$, $\lambda$, and $-\lambda$,
the associated particle distribution functions $\fz$, $\fp$, and $\fm$,
and three polynomials defining the three moments denoted by $u$, $v$, and $w$,
\begin{equation*}
    \begin{aligned}
    P_0 &= 1, & P_1 &= X, & P_2 &= \tfrac{1}{2}X^2,\\
    \mz &= \fz+\fp+\fm, & \mun &= \lambda \fp - \lambda \fm, & \mdeux &= \tfrac{1}{2}\lambda^2 \fp + \tfrac{1}{2}\lambda^2 \fm.
    \end{aligned}
\end{equation*}
The three equilibrium values of the moments read
\begin{equation*}
    \mzeq(\mz) = \mz, \quad
    \muneq(\mz) = c \lambda \mz, \quad
    \mdeuxeq(\mz) = \tfrac{1}{2}d \lambda^2\mz,
\end{equation*}
where $c$ and $d$ are real values.
For the moment, $c$ and $d$ are kept free: their value to ensure consistency with \eqref{eq:pde} is discussed in Section \ref{sec:consistency}.
Finally, we introduce two relaxation parameters for the non-conserved moments $\sun$ and $\sdeux$; and the centered relaxation parameters $\tun=\sun-1$ and $\tdeux=\sdeux-1$.
We use vectorial notations by introducing the vector of the particle distribution functions  $\vectf = (\fz,\fp,\fm)^T$ and the vector of the moments $\vectm = (\mz, \mun, \mdeux)^T$ linked by the matrix relation $\vectm = \matM\vectf$ with
\begin{equation*}
\matM = \begin{pNiceArray}{ccc}
1&1&1\\ 0&\lambda&-\lambda\\ 0&\frac{\lambda^2}{2}&\frac{\lambda^2}{2}
\end{pNiceArray},
\qquad
\matM^{-1} = \begin{pNiceArray}{ccc}
1 & 0 & -\frac{2}{\lambda^2} \\
0 & \frac{1}{2\lambda} & \frac{1}{\lambda^2} \\
0 & -\frac{1}{2\lambda} & \frac{1}{\lambda^2}
\end{pNiceArray}.
\end{equation*}
Then, the algorithm proceeds as follows.
Given $(\vectm_{j}^0)_{j\in\mathbb{Z}}$ arbitrary (which however must be eventually linked to the initial datum $\mz^{\circ}$ to establish a connection with the Cauchy problem at hand), for $n\in\mathbb{N}$, compute
\begin{itemize}
    \item Relaxation. For $j\in\mathbb{Z}$, $n\in\mathbb{N}$,
    \begin{equation}\label{eq:relaxation}
        \vectm_{j}^{n, \star}
        =
        \mathbf{diag}(1, 1-\sun, 1-\sdeux) \, \vectm_{j}^{n}
        +
        \mathbf{diag}(0, \sun, \sdeux) \,
        \boldsymbol{m}^{\textrm{eq}}(\mz_j^n).
    \end{equation}
    \item Transport. For $j\in\mathbb{Z}$, $n\in\mathbb{N}$,
    \begin{equation}\label{eq:transport}
        f_{0, j}^{n+1} = f_{0, j}^{n, \star}, 
        \qquad 
        f_{\pm, j}^{n+1} = f_{\pm, j\mp\lambda \timestep/\spacestep}^{n, \star} = f_{\pm, j\mp1}^{n, \star}.
    \end{equation}
\end{itemize}

As the problem is linear, it is convenient to introduce a discrete Fourier transform in space, see \cite[Chapter 2]{strikwerda_finite_2004} for the details.
We denote any space-grid function having undergone a Fourier transformation by a hat, so that the scheme on the distribution functions reads 
\begin{equation*}
    \hat{\vectf}^{n+1}
    (\xi)
    =
    \matG(\xi)\hat{\vectf}^{n}, 
    \qquad \forall\xi\in[-\pi, \pi],
\end{equation*}
where $\matG$ is called ``amplification matrix'', and given by
\begin{equation*}
\matG(\xi) = \matT(\xi) \matM^{-1} \matR \matM,
\end{equation*}
with
\begin{equation*}
\matT(\xi) = \begin{pNiceArray}{ccc}
    1&0&0 \\ 0&e^{-\i\xi}&0 \\ 0&0&e^{\i\xi}
\end{pNiceArray},
\qquad
\matR = \begin{pNiceArray}{ccc}
    1&0&0 \\ 
    \sun c\lambda & 1-\sun&0 \\ 
    \frac{1}{2} \sdeux d \lambda^2 & 0 & 1-\sdeux
\end{pNiceArray}.
\end{equation*}
After algebra, we obtain
\begin{equation}\label{eq:matGs}
\matG = \begin{pNiceArray}{ccc}
1 {-} \sdeux d
&
\sdeux(1{-}d)
&
\sdeux(1{-}d) \\
\frac{e^{{-}\i\xi}}{2} (\sdeux d {+} \sun c)
&
\frac{e^{{-}\i\xi}}{2}
(2 {-} \sdeux (1{-}d) {-} \sun (1{-}c))
&
\frac{e^{{-}\i\xi}}{2}
({-} \sdeux(1{-}d){+}\sun(1{+}c)) \\
\frac{e^{\i\xi}}{2} (\sdeux d {-} \sun c)
&
\frac{e^{\i\xi}}{2}
({-} \sdeux(1{-}d){+}\sun(1{-}c) )
&
\frac{e^{\i\xi}}{2}
(2 {-} \sdeux (1{-}d) {-} \sun (1{+}c))
\end{pNiceArray}.
\end{equation}
The matrix $\matG$ can be rewritten in terms of the centered relaxation parameters $\tun$ and $\tdeux$ as
{\footnotesize\begin{equation}\label{eq:matG}
    \matG = \begin{pNiceArray}{ccc}
    1 {-} d {-} \tdeux d
    &
    (1{+}\tdeux)(1{-}d)
    &
    (1{+}\tdeux)(1{-}d) \\
    \frac{e^{{-}\i\xi}}{2} (
        d{+}c{+} \tdeux d {+} \tun c
    )
    &
    \frac{e^{{-}\i\xi}}{2}(
        d{+}c
        {-} \tdeux (1{-}d) {-} \tun (1{-}c)
    )
    &
    \frac{e^{{-}\i\xi}}{2}(
        d{+}c
        {-} \tdeux (1{-}d) {+} \tun (1{+}c)
    ) \\
    \frac{e^{\i\xi}}{2} (
        d{-}c {+} \tdeux d {-} \tun c
    )
    &
    \frac{e^{\i\xi}}{2}(
        d{-}c
        {-} \tdeux(1{-}d)
        {+}\tun(1{-}c) 
    )
    &
    \frac{e^{\i\xi}}{2}(
        d{-}c {-} \tdeux (1{-}d) {-} \tun (1{+}c)
    )
    \end{pNiceArray}.
\end{equation}}

\begin{remark}
    The amplification matrix does not depend on the choice of the second-order moment $P_2$ as long as two conditions are met.
    First, $P_2$ reads $\alpha X^2+\beta$, for any real values $\alpha\neq0$ and $\beta$.
    Second, the equilibrium values of the particle distribution functions are unchanged \cite{dubois_notion_2020}.
    In particular, a classical choice for the second-order moment is $\tilde P_2=3X^2-2\lambda^2$ for which the corresponding matrix of the moments has orthogonal rows
    \begin{equation*}
        \tilde \matM = \begin{pNiceArray}{ccc}
            1&1&1\\
            0&\lambda&-\lambda\\
            -2\lambda^2&\lambda^2&\lambda^2
        \end{pNiceArray},
        \qquad
        \tilde \matM^{-1} = \begin{pNiceArray}{ccc}
            \frac{1}{3} & 0 & -\frac{2}{6\lambda^2} \\
            \frac{1}{3} & \frac{1}{2\lambda} & \frac{1}{6\lambda^2} \\
            \frac{1}{3} & -\frac{1}{2\lambda} & \frac{1}{6\lambda^2}
        \end{pNiceArray}.    
    \end{equation*}
    The amplification matrix $\matG$ is then unchanged as long as ${\tilde \mdeux}^{\eq} = (3d-2)\lambda^2 \mz$.
\end{remark}

%%%%%%%%%%%%%%%%%%%%%%%%%%%%%%%%%%%%%%%%%%%%%%%%%%
 
%%%%%%%%%%%%%%%%%%%%%%%%%%%%%%%%%%%%%%%%%%%%%%%%%%
% SECTION CONSISTENCY
%%%%%%%%%%%%%%%%%%%%%%%%%%%%%%%%%%%%%%%%%%%%%%%%%%
%%%%%%%%%%%%%%%%%%%%%%%%%%%%%%%%%%%%%%%%%%%%%%%%%%
% SECTION CONSISTENCY
%%%%%%%%%%%%%%%%%%%%%%%%%%%%%%%%%%%%%%%%%%%%%%%%%%
\section{Consistency}\label{sec:consistency}

We now apply the algebraic framework developed by Bellotti et al. \cite{bellotti_finite_2022}, which establishes an exact correspondence between a lattice Boltzmann scheme and an equivalent multistep finite-difference scheme. This correspondence enables the systematic derivation of the associated macroscopic equations and provides access to the classical finite-difference analysis framework.

\begin{proposition}[Consistency]
    Assume that $\sun, \sdeux\neq 0$.
    Then, the modified equation of the scheme truncated at second order, cf. \cite[Theorem 3.7]{bellotti2023truncation}, is given by 
    \begin{equation}\label{eq:modEq2nd}
        \partial_t \mz +\lambda c \partial_x \mz 
        - \lambda^2 \Delta t \bigl( \tfrac{1}{2} {-} \tfrac{1}{\sun}\bigr) (d{-}c^2)\partial_{xx}\mz
        = \mathcal{O}(\Delta t^2).
    \end{equation}
    Therefore, the scheme is first-order accurate with respect to \eqref{eq:pde}  upon selecting $\lambda c = V$.
    Moreover, the scheme becomes second-order accurate regardless of $\sun$ by taking $d = c^2$.
    Under this choice, the modified equation truncated at third order is 
    \begin{equation}\label{eq:modEq3rd}
        \partial_t \mz + V \partial_x \mz 
        = -\lambda^3 \Delta t^2 c(1-c^2) 
        \bigl (\tfrac{1}{12}-\bigl( \tfrac{1}{\sun}{-}\tfrac{1}{2}\bigr ) \bigl( \tfrac{1}{\sdeux}{-}\tfrac{1}{2}\bigr )\bigr ) \partial_{xxx}\mz + \mathcal{O}(\Delta t^3).
    \end{equation}
\end{proposition}

\begin{proof}
    Equation \eqref{eq:modEq2nd} is given by \cite[Theorem 3.7]{bellotti2023truncation}.
    Equation \eqref{eq:modEq3rd} can be obtained by analogous computations at the next order.
\end{proof}

From now on in the paper, we adopt the choice $d = c^2$, so that $d$ is no longer a free parameter.
The product $( \frac{1}{\sun}{-}\tfrac{1}{2} ) ( \frac{1}{\sdeux}{-}\tfrac{1}{2} )$ is the so-called ``magic parameter'' \cite{dhumieres_viscosity_2009}, and the fact of achieving third-order accuracy by taking it equal to $1/12$ is well-known within the lattice Boltzmann community, see \cite{ginzburg2008consistent, kuzmin2011role}.

For the previous consistency analysis, we have assumed that $\sun, \sdeux\neq 0$.
When one or both vanish, the study must be adapted accordingly. When $\sun = 0$, the scheme is not consistent with the target equation.
Assuming $\sun\neq 0$, when $\sdeux = 0$, it is only first-order accurate despite having taken $d = c^2$, unless $\sun = 2$, in which case the scheme becomes second-order accurate.
This fact can be seen by following the procedure by \cite{bellotti2023truncation} for the specific scheme at hand, or by performing the ``equivalent equation'' analysis following the guidelines of \cite{dubois2008equivalent, dubois2022nonlinear}.

We shall come back to consistency from the standpoint of Fourier analysis in Sec.~\ref{sec:convergence}, dedicated to convergence.
Indeed, upon building a suitable link between $\vectm^0$ and the initial datum $\mz^{\circ}$ of the Cauchy problem, see \eqref{eq:initialcondition}, we expect and hope that $\mz_{j}^{n}$ be a ``good approximation'' of the solution $\mz(n\timestep, j\spacestep)$.
%%%%%%%%%%%%%%%%%%%%%%%%%%%%%%%%%%%%%%%%%%%%%%%%%%

%%%%%%%%%%%%%%%%%%%%%%%%%%%%%%%%%%%%%%%%%%%%%%%%%%
% SECTION STABILITY
%%%%%%%%%%%%%%%%%%%%%%%%%%%%%%%%%%%%%%%%%%%%%%%%%%
%%%%%%%%%%%%%%%%%%%%%%%%%%%%%%%%%%%%%%%%%%%%%%%%%%
% SECTION STABILITY
%%%%%%%%%%%%%%%%%%%%%%%%%%%%%%%%%%%%%%%%%%%%%%%%%%
\section{Stability}\label{sec:stability}

% This section is devoted to give an explicit relation between the relaxation parameters that guaranties the stability of the scheme. This condition is a necessary and sufficient condition.
% We first remind some definitions and properties on complex polynomials.

This section aims to provide an explicit relation between the relaxation parameters ensuring stability of the scheme. This condition is both necessary and sufficient. We first recall some definitions and properties of polynomials with complex coefficients.

%%%%%%%%%%%%%%%%%%%%%%%%%%%%%%%%%%%%%%%%%%%%%%%%%%
% SUBSECTION REMINDER
%%%%%%%%%%%%%%%%%%%%%%%%%%%%%%%%%%%%%%%%%%%%%%%%%%
\subsection{Reminder on complex polynomials}

% We begin with some definitions and notations of the various classes of polynomials, about which the stability conditions will be concerned. These classes are all distinguished by the location of their zeros relative to the unit circle. This section is based on \cite{miller_location_1971,marden_geometry_1966, strikwerda_finite_2004}.
% We emphasize that the use of these tools to analyze stability of lattice Boltzmann scheme is not new, see for instance \cite{ginzburg2010optimal}.

We begin by introducing definitions and notation for several classes of polynomials relevant to the stability analysis, which are characterized by the location of their zeros with respect to the unit circle. This section is based on \cite{miller_location_1971, marden_geometry_1966, strikwerda_finite_2004}. We emphasize that these tools have already been used in the stability analysis of lattice Boltzmann schemes; see, for instance, \cite{ginzburg2010optimal}, where the stability condition reduces to locating the roots of a second-degree polynomial.

Let $\varphi\in\mathbb{C}[z]$ be a polynomial with complex coefficients of degree $d$,
\begin{equation*}
\varphi = a_d z^d + \ldots + a_0 = \sum_{\ell=0}^d a_\ell z^\ell.
\end{equation*}
We say that $\varphi$ is of exact degree $d$ if $a_d$ is not zero.

\begin{definition}[Schur polynomial]
    The polynomial $\varphi$ is a \emph{Schur polynomial} if all its roots $r$ satisfy $\vert r\vert < 1$. We denote by $\Sp$ the set of the Schur polynomials.
\end{definition}

\begin{definition}[von Neumann polynomial]
    The polynomial $\varphi$ is a \emph{von Neumann polynomial} if all its roots $r$ satisfy $\vert r\vert \leq 1$. We denote by $\vNp$ the set of the von Neumann polynomials.
\end{definition}
    
\begin{definition}[simple von Neumann polynomial]
    The polynomial $\varphi$ is a \emph{simple von Neumann polynomial} if $\varphi$ is a von Neumann polynomial and its roots on the unit circle are simple roots. We denote by $\svNp$ the set of the simple von Neumann polynomials.
\end{definition}

We finally define the Cohn-Schur operator $\cohnschur:\mathbb{C}[z]\to\mathbb{C}[z]$ an application that decreases the degree.
For any polynomial $\varphi$ of exact degree $d$ we first define the polynomial $\varphi^\star$ by
\begin{equation*}
\varphi^\star = \bar a_0 z^d + \ldots + \bar a_d = \sum_{\ell=0}^d \bar a_{d-\ell} \; z^\ell,
\end{equation*}
where the bar on the coefficients of $\varphi$ denotes the complex conjugate and then
\begin{equation*}    
    \cohnschur(\varphi) = \frac{
        \varphi^\star(0)\, \varphi
        -
        \varphi(0)\, \varphi^\star
        }{z}.
\end{equation*}
Remark that $\cohnschur(\varphi)$ is indeed a polynomial as the constant term of $\varphi^\star(0) \varphi-\varphi(0) \varphi^\star$ vanishes.
It is easy to see that the degree of $\cohnschur(\varphi)$ is less than that of $\varphi$.
We also denote by a prime $\varphi^\prime$ the derivative of the polynomial $\varphi$ with respect to its determinate.

The next three theorems give recursive tests for Schur, von Neumann and simple von Neumann polynomials \cite{miller_location_1971}. 

\begin{theorem}[Miller 6.2 page 404]\label{thm:strikwerda431}
    The polynomial $\varphi$ is a Schur polynomial of exact degree $d$ if, and only if, $\vert\varphi(0)\vert < \vert\varphi^\star(0)\vert$ and $\cohnschur(\varphi)$ is a Schur polynomial of exact degree $d{-}1$.
\end{theorem}

\begin{theorem}[Miller 6.1 page 403]\label{thm:strikwerda437}
    The polynomial $\varphi$ is a von Neumann polynomial of exact degree $d$ if, and only if, either
    \begin{enumerate}[label=\emph{(\alph*)}]
        \item $\vert\varphi(0)\vert < \vert\varphi^\star(0)\vert$ and $\cohnschur(\varphi)$ is a von Neumann polynomial of exact degree $d{-}1$ or
        \item $\cohnschur(\varphi)$ is identically zero and $\varphi^\prime$ is a von Neumann polynomial of exact degree $d{-}1$.
    \end{enumerate}
\end{theorem}

\begin{theorem}[Miller 6.3 page 404]\label{thm:strikwerda432}
    The polynomial $\varphi$ is a simple von Neumann polynomial of exact degree $d$ if, and only if, either
    \begin{enumerate}[label=\emph{(\alph*)}]
        \item $\vert\varphi(0)\vert < \vert\varphi^\star(0)\vert$ and $\cohnschur(\varphi)$ is a simple von Neumann polynomial of exact degree $d{-}1$ or
        \item $\cohnschur(\varphi)$ is identically zero and $\varphi^\prime$ is a Schur polynomial of exact degree $d{-}1$.
    \end{enumerate}
\end{theorem}

The proofs of these three theorems are lying on Rouché's theorem from complex analysis \cite{rouche_memoire_1862}.

\begin{definition}
    We denote
    \begin{equation*}
        \ensinegalite = \Bigl\lbrace
        \varphi \in\mathbb{C}[z] \text{ s.t. } 
        \vert\varphi(0)\vert < \vert\varphi^\star(0)\vert
        \Bigr\rbrace
    % \end{equation*}
    % and
    % \begin{equation*}
    \quad\textrm{and}\quad
        \ensegalite = \Bigl\lbrace
        \varphi \in\mathbb{C}[z] \text{ s.t. } 
        \cohnschur(\varphi) = 0
        \Bigr\rbrace .%= \operatorname{Ker}(\cohnschur).
    \end{equation*}
\end{definition}

With these notations, theorems \ref{thm:strikwerda431}, \ref{thm:strikwerda437}, and \ref{thm:strikwerda432} read
\begin{align}
    \varphi\in\Sp &\iff \phantom{\Big(}\varphi\in\ensinegalite \text{ and } \cohnschur(\varphi)\in\Sp,\\
    \varphi\in\vNp &\iff \Bigl(
        \varphi\in\ensinegalite \text{ and } \cohnschur(\varphi) \in\vNp
    \Bigr) \text{ or } \Bigl(
        \varphi\in\ensegalite \text{ and } \varphi'     \in\vNp
    \Bigr),\\
    \varphi\in\svNp &\iff \Bigl(
        \varphi\in\ensinegalite \text{ and } \cohnschur(\varphi) \in\svNp
    \Bigr) \text{ or } \Bigl(
        \varphi\in\ensegalite \text{ and } \varphi' \in\Sp
    \Bigr).
\end{align}

Note that we cannot have simultaneously $\varphi\in\ensinegalite\cap\ensegalite$. Indeed, if $\varphi$ is of exact degree $d$ and belongs to $\ensinegalite$, the polynomial $\cohnschur(\varphi)$ is of exact degree $d{-}1$ as its term of degree $d{-}1$ is equal to $\vert\varphi^\star(0)\vert^2-\vert\varphi(0)\vert^2>0$.

% Cas d'un polynôme de degré 1 :
% varphi = alpha X + beta
% varphi* = bar beta X + bar alpha
% C(varphi) = (|alpha|^2 X + bar alpha beta 
%               - |beta|^2 X - beta bar alpha)/X
%           = |alpha|^2-|beta|^2 
%
% varphi in Sp iff |beta| < |alpha|
% 
% varphi in svNp iff varphi in vNp iff
%   |beta| <= |alpha|
%%%%%%%%%%%%%%%%%%%%%%%%%%%%%%%%%%%%%%%%%%%%%%%%%%
% SUBSECTION DEFINITION STABILITY
%%%%%%%%%%%%%%%%%%%%%%%%%%%%%%%%%%%%%%%%%%%%%%%%%%
\subsection{Definition of the stability}

We consider in this contribution the $\ell^2$-stability. As the amplification matrix $\matG$ given in \eqref{eq:matG} does not depend on the time step $\Delta t$ or the space step $\Delta x$, the definition for the stability reads as follow.

\begin{definition}[Stability] \label{thm:defstab}
    The lattice Boltzmann scheme characterized by the amplification matrix $\matG$ is said stable if there exists one constant $C>0$ such that 
    \begin{equation*}
        % \exists C > 0 \quad s.t. \quad
        \forall n \in\mathbb{N}, \quad \forall \xi\in[-\pi,\pi]\qquad |\matG(\xi)^n|\leq C.
    \end{equation*}
\end{definition}

%The fact that this definition be fulfilled is equivalent, see \cite[Proposition 1]{coulombel:cel-00616497} and thanks to the Parseval's relation, to the $\ell^2$ control $\sum_{j\in\mathbb{N}}\spacestep |\vectf_j^n|^2\leq C^2 \sum_{j\in\mathbb{N}}\spacestep |\vectf_j^0|^2$ for all $n \in\mathbb{N}$ and for all $\vectf_j^0$.
The fact that this definition be fulfilled is equivalent to the $\ell^2$ control
$$
\spacestep \sum_{j\in\mathbb{Z}} |\vectf_j^n|^2
\leq C^2 
\spacestep \sum_{j\in\mathbb{Z}} |\vectf_j^0|^2
$$
for all $n \in\mathbb{N}$ and for all $\vectf_j^0$
thanks to the Parseval's relation.
% \cite[Proposition 1]{coulombel:cel-00616497}.
% \cite{coulombel:cel-00616497}.
The following result gives a necessary condition for stability.%, and is similar to \cite[Lemma 3]{coulombel:cel-00616497}.
Its sole necessary character comes from the fact that is does not ensure a bound uniformly in $\xi$.
\begin{proposition}
    If the lattice Boltzmann scheme characterized by the amplification matrix $\matG$ is stable, then the minimal polynomial of $\matG$ is a simple von Neumann polynomial for any $\xi\in[-\pi,\pi]$.
\end{proposition}

\begin{proof}
    The proof can be found in \cite[Lemma 3]{coulombel:cel-00616497}. For readibility reasons, we present it here briefly. 
    The key point is the Dunford-Jordan decomposition of the complex square matrices. The matrix $\matG$ is equivalent to a block diagonal matrix that reads
    \begin{equation*}
        \begin{pNiceMatrix}[nullify-dots]
            J_{\lambda_1} & 0 & \Cdots & 0\\ 
            0 & \Ddots & \Ddots & \Vdots\\
            \Vdots & \Ddots & \Ddots & 0\\
            0 & \Cdots & 0 & J_{\lambda_b}
        \end{pNiceMatrix}, 
        \text{ where }
        J_{\lambda_k} = \begin{pNiceMatrix}[
            first-row,last-col,nullify-dots
        ]
            & \Ldots[line-style={solid,<->},shorten=0pt]^{\mu_k \text{ columns}} \\
            \lambda_k & 1      & 0      & \Cdots & 0      & \\
            0         & \Ddots & \Ddots & \Ddots & \Vdots & \\ 
            \Vdots    & \Ddots & \Ddots & \Ddots & 0      &
            \phantom{\hskip10pt}\Vdots[line-style={solid,<->},shorten=0pt]^{\mu_k \text{ rows}} \\
            \Vdots    &        & \Ddots & \Ddots & 1      & \\
            0         & \Cdots & \Cdots & 0      & \lambda_k & 
        \end{pNiceMatrix},
        \ 1\leq k\leq b.
    \end{equation*}
    The complex values $\lambda_1,\ldots,\lambda_b$ are the eigenvalues of $\matG$ which are not necessarily distinct. The sequel of matrices $(\matG^n)_{n\in\mathbb{N}}$ is then bounded if, and only if, we have first $\vert\lambda_k\vert\leq1$, for $1\leq k\leq b$, and second $\mu_k=1$ if $\vert\lambda_k\vert=1$. This condition is therefore exactly that the minimal polynomial of $\matG$ is a simple von Neumann polynomial.
\end{proof}

A sufficient condition, that is indeed useful in most of the configurations considered in the paper, comes from \cite[Proposition 3 and 4]{bellotti2025stability}.

\begin{proposition}\label{thm:sufficientcond}
    If the characteristic polynomial of $\matG$ is a simple von Neumann polynomial for any $\xi\in[-\pi,\pi]$, then the lattice Boltzmann scheme characterized by the amplification matrix $\matG$ is stable.
\end{proposition}

We hence see that gaps between necessary and sufficient conditions come from those between minimal and characteristic polynomials.

\begin{notation}[minimal and characteristic polynomials]
    The characteristic polynomial of the matrix $\matG$ defined in \eqref{eq:matG} is denoted by $\polyPt(X;\xi,c,\tun,\tdeux)$.
    Moreover, its minimal polynomial is denoted by $\polymini(X;\xi,c,\tun,\tdeux)$.
\end{notation}

As the minimal polynomial of a complex matrix is not easy to determine and its degree may vary with $\xi$, we propose to compute the characteristic polynomial $\polyPt$ of the amplification matrix $\matG$. Then, if $\polyPt\in\svNp$, 
the scheme is stable thanks to Prop.~\ref{thm:sufficientcond};
%the minimal polynomial $\polymini$ is also a simple von Neumann polynomial; 
if $\polyPt\notin\vNp$, the minimal polynomial $\polymini$ is not a von Neumann polynomial and the scheme is unstable; and if $\polyPt\in\vNp\backslash\svNp$, the case has to be investigated more precisely. Fortunately, the first two cases are generic and the last one is not.

The characteristic polynomial of the amplification matrix $\matG$ reads
\begin{equation}\label{eq:defP3}
\polyPt(X;\xi,c,\tun,\tdeux) = \alphatz + \alphatu  X + \alphatd  X^2 + \alphatt  X^3 ,
\end{equation}
with
\begin{align*}
    \alphatz &= -\tun \tdeux,\\
    \Re(\alphatu ) &=
    \tun \tdeux -(1-\cos\xi)\,\tun\,(1+\tdeux)\,c^2 - \cos\xi\,(\tun+\tdeux),
    \\
    \Im(\alphatu ) &= c\,\tdeux\,(1+\tun) \sin\xi,\\
    \Re(\alphatd ) &= 
    -1+c^2(1+\tdeux)(1 - \cos\xi) + \cos\xi\,(\tun+\tdeux),
    \\
    \Im(\alphatd ) &= c(1+\tun) \sin\xi,\\
    \alphatt &= 1,
\end{align*}
where $\Re(\cdot)$ and $\Im(\cdot)$ refer to the real and imaginary parts of a complex number.

\begin{remark}[Corresponding finite difference scheme]\label{rem:correspondingFD}
    As emphasized at the end of Sec. \ref{sec:consistency}, we look at $\mz_{j}^n$ as an approximation of the solution of the target PDE.
    One can therefore aim at recasting the discrete dynamics of $\mz_{j}^n$ only in terms of itself.
    This can be done in full generality using the characteristic polynomial $\polyPt$ from \eqref{eq:defP3}.
    This is a consequence of the Cayley-Hamilton theorem.
    More precisely, 
    \begin{multline}\label{eq:FD}
        \text{if} \quad (\mz_{j}^n)_{j\in\mathbb{Z}}^{n\in\mathbb{N}}
        \quad 
        \text{fulfills \eqref{eq:relaxation} and \eqref{eq:transport}, it also fulfills, for }\xi\in[-\pi, \pi]\\
        \begin{cases}
            \hat{\mz}^{1}(\xi) = (\matM\matT(\xi) \matM^{-1}\matR \, \hat{\vectm}^{0}(\xi))_1, \\
            \hat{\mz}^{2}(\xi) = ((\matM\matT(\xi) \matM^{-1}\matR)^2\, \hat{\vectm}^{0}(\xi))_1,\\
            \hat{\mz}^{n+1}(\xi) =  - \alphatd(\xi) \hat{\mz}^{n}(\xi) - \alphatu(\xi) \hat{\mz}^{n-1}(\xi) -\alphatz(\xi) \hat{\mz}^{n-2}(\xi), 
            \quad n\geq 2.
        \end{cases}            
    \end{multline}
    This means that $\mz_{j}^n$ also fulfills a multi-step finite difference scheme, whose ``amplification polynomial'' is $\polyPt$, with specific initialization schemes, linked to the choice of $\vectm^0$, that provide $\mz_{j}^1$ and $\mz_{j}^2$.
    The stability of the multi-step finite difference scheme, namely the fact that 
    \begin{equation*}
    \sum_{j\in\mathbb{Z}}\spacestep |z_j^n|^2\leq C^2 ( \sum_{j\in\mathbb{Z}}\spacestep |z_j^0|^2 + \sum_{j\in\mathbb{Z}}\spacestep |z_j^1|^2 + \sum_{j\in\mathbb{Z}}\spacestep |z_j^2|^2),
    % for all $n \geq 3$,
    \quad n\geq3,
    \end{equation*}
    where 
    \begin{equation*}
        \hat{z}^{n+1}(\xi) =  - \alphatd(\xi) \hat{z}^{n}(\xi) - \alphatu(\xi) \hat{z}^{n-1}(\xi) -\alphatz(\xi) \hat{z}^{n-2}(\xi),
        \quad n\geq 2,
    \end{equation*}
    is easy to ascertain.
    % By the results of \cite[Chapter 4]{strikwerda_finite_2004}

    Contrarily to the lattice Boltzmann setting given in Prop.~\ref{thm:sufficientcond}, the stability of the multi-step finite difference scheme is equivalent to $\polyPt(X;\xi,c,\tun,\tdeux)\in\svNp$ for all $\xi\in[-\pi, \pi]$ \cite[Chapter 4]{strikwerda_finite_2004}.
    Note that the initial data $z_j^0$, $z_j^1$, and $z_j^2$ are arbitrary.
    % We also warn the reader that the fact of having $\polyPt(X;\xi,c,\tun,\tdeux)\in\svNp$ for all $\xi\in[-\pi, \pi]$ is---contrarily to the setting of finite differences above---only sufficient to have a stable lattice Boltzmann scheme according to Definition \ref{thm:defstab}, cf. Proposition \ref{thm:sufficientcond}.
    In particular, a counter-example is provided in \cite{bellotti2024influence}, where the lattice Boltzmann scheme is stable, whereas the multi-step finite difference scheme that is obtained using the characteristic polynomial is not.
\end{remark}

%%%%%%%%%%%%%%%%%%%%%%%%%%%%%%%%%%%%%%%%%%%%%%%%%%
% SUBSECTION STABILITY CONDITION
%%%%%%%%%%%%%%%%%%%%%%%%%%%%%%%%%%%%%%%%%%%%%%%%%%
\subsection{Main result: the stability condition}

In this section, the stability condition is investigated and a necessary and sufficient condition is given.

The first result is useful to reduce the complexity of the proof. The parameters that ensure the scheme stability have to be searched such that conditions $-1\leq\tun,\tdeux\leq1$ and $\vert c\vert \leq1$ are satisfied (the case $\tun=\tdeux=-1$ being fully treated).

\begin{proposition}\label{thm:cond_sc}
    Assuming that the scheme is stable in the sense of Def.~\ref{thm:defstab}, first the centered relaxation parameters verify
    \begin{equation*}%\tag{$\textsf{C}_s$}\label{eq:cond_square}
        \vert\tun\vert\leq1, 
        \qquad\text{and}\qquad
        \vert\tdeux\vert\leq1 ;
    \end{equation*}
    second the equilibrium value $c$ satisfies
    \begin{equation*}%\tag{$\textsf{C}_c$}\label{eq:cond_c}
        \vert c\vert \leq 1 
        \qquad\text{if}\qquad 
        (\tun,\tdeux) \neq (-1,-1).
    \end{equation*}
    Moreover, if $(\tun,\tdeux)=(-1,-1)$, the scheme is stable.
\end{proposition}

\begin{proof}
    First, for $\xi=0$, we have
    \begin{equation*}
        % P3 = -tvtw + X(tvtw-tv-tw) + X^2(-1+tv+tw) + X^3
        %    = (X-1)(X+tv)(X+tw)
        %\restriction{\polyPt}{\xi=0}(X) = (X-1)(X+\tun)(X+\tdeux).
        \polyPt(X;\xi=0,c,\tun,\tdeux) = (X-1)(X+\tun)(X+\tdeux).
    \end{equation*}
    If the scheme is stable, the polynomial $\polyPt$ has to be a von Neumann polynomial in particular for $\xi=0$: the condition 
    %\eqref{eq:cond_square} 
    $-1\leq\tun,\tdeux\leq1$
    is then immediately satisfied.

    To prove that the second condition is also a necessary condition for the stability, the characteristic polynomial $\polyPt$ is studied for $\xi=\pi$ and $\xi=\pi/2$. First we assume that $\tdeux\in(-1,1]$:
\begin{equation*}
    % P3 = -tvtw + X(tvtw-2tv(1+tw)c^2+tv+tw)
    %      + X^2(-1+2c^2(1+tw)-(tv+tw)) + X^3
    % P3(tv) = -tvtw + tv^2tw + tv^2 + tvtw - 2tv^2(1+tw)c^2
    %          - tv^2 - tv^3 - tv^2tw + 2tv^2(1+tw)c^2 + tv^3
    %        = 0
    % P3 = (X-tv)(X^2+aX+tw)
    % avec -tv+a = -1+2c^2(1+tw)-(tv+tw)
    % => a = (1+tw)(2c^2-1)
%\restriction{\polyPt}{\xi=\pi}(X) 
\polyPt(X;\xi=\pi,c,\tun,\tdeux)
= (X-\tun) (X^2 + X(2c^2-1)(\tdeux+1)+\tdeux).
\end{equation*}
%First, as $\tun$ is a root, we obtain directly that $\vert\tun\vert\leq1$ is a necessary condition for $\polyPt$ to be a von Neumann polynomial.
%Second, as $\tdeux$ is the product of the two other roots, we also obtain that $\vert\tdeux\vert\leq1$ is a necessary condition for $\polyPt$ to be a von Neumann polynomial.
Assuming then that $\vert c\vert > 1$, we note $Q_3^\pi=X^2 + X(2c^2-1)(\tdeux+1)+\tdeux$.
This polynomial has two real roots as its discriminant $\Delta$ is positive:
\begin{equation*}
    \Delta = (2c^2-1)^2(\tdeux+1)^2-4\tdeux
    = 4c^2(c^2-1)(1+\tdeux)^2 + (1-\tdeux)^2 > 0.
\end{equation*}
As $Q_3^\pi$ is convex and $Q_3^\pi(-1)=2(1-c^2)(1+\tdeux) <0$, one of its roots is less than $-1$.
Second, we assume that $\tdeux=-1$ and $\tun\in(-1,1]$:
\begin{equation*}
    % P3 = tv + X (-tv - ic(1+tv)) 
    %         + X^2(-1+ic(1+tv)) + X^3
    % P3(1) = tv - tv - ic(1+tv) - 1 +ic(1+tv) + 1
    % P3 = (X-1)(X^2-aX-tv)
    % avec a = tv - tv - ic(1+tv) = -ic (1+tv)
    %\restriction{\polyPt}{\xi=\pi/2}(X) 
    \polyPt(X;\xi=\pi/2,c,\tun,-1)
    = 
    (X-1) (X^2 + ic(1+\tun)X -\tun).
\end{equation*}
% Delta = -c^2(1+tv)^2 + 4 tv
%       = -(c^2-1)(1+tv)^2 - (1-tv)^2 < 0
% les racines sont donc purement imaginaires
% P3(iY) = -Y^2-c(1+tv)Y-tv
% P3(1) = -1-c(1+tv)-tv < 0
% P3(-1) = -1+c(1+tv)-tv = (1+tv)(c-1) > 0
% et donc une racine est strictement inférieure à -1...
Assuming then that $\vert c\vert > 1$, we note $Q_3^{\pi/2}=X^2+ic(1+\tun)X-\tun$.
This polynomial has two purely imaginary roots as its discriminant $\Delta$ is negative:
\begin{equation*}
    \Delta = -c^2(1+\tun)^2+4\tun = -(c^2-1)(1+\tun)^2 - (1-\tun)^2 < 0.
\end{equation*}
We define $\tilde Q_3^{\pi/2}(X) = Q_3^{\pi/2}(iX) = -X^2-c(1+\tun)X-\tun$, $\tilde Q_3^{\pi/2}$ is concave. 
If $c>1$, $\tilde Q_3^{\pi/2}(-1)=(1+\tun)(c-1)>0$, one of its roots is less than $-1$.
If $c<-1$, $\tilde Q_3^{\pi/2}(1)=(1+\tun)(-c-1)>0$, one of its roots is greater than $1$.
That ends the proof.

The case $\tun=-1$ and $\tdeux=-1$ is fully degenerated as the matrix $\matG$ reads
\begin{equation*}
    \matG = \begin{pNiceMatrix}
        1&0&0\\
        0&e^{-i\xi}&0\\
        0&0&e^{i\xi}
    \end{pNiceMatrix}
\end{equation*}
and its minimal polynomial is clearly a simple von Neumann polynomial for any $\xi\in[-\pi,\pi]$ and for any $c\in\mathbb{R}$.
\end{proof}

% tv=tw=-1
% P3 = -1 + X(1+2cosxi) + X^2(-1-2cosxi) + X^3
%    = (X-1)(X^2+X+1- (1+2cosxi)X)
%    = (X-1) (X^2 - 2cosxiX + 1)
% Delta' =  cosxi^2-1 < 0
% Xpm = cosxi pm i sinxi de module 1 !
%
%     ( 1       0          0     )
% G = ( 0  exp(-i xi)      0     )
%     ( 0       0      exp(i xi) )
% donc le polynôme minimal de G est simple von Neumann quel que soit xi !

The following theorem is the main result of the contribution and gives the necessary and sufficient condition for the D1Q3 scheme to be stable. In order to be more readble for regular users of lattice Boltzmann scheme written in the MRT formalism, the condition is written according to the relaxation parameters $\sun,\sdeux$ and not the centered ones $\tun,\tdeux$. 

\begin{theorem}\label{thm:main}
    %{\afaire ADD THE CASE $\vert c\vert > 1$}
    %\averifier
The $D1Q3$ scheme is stable if, and only if the relaxation parameters verify the following properties according to the value of the velocity $c$:
\begin{itemize}
    \item $0\leq\sun,\sdeux\leq2$ and $(\sun,\sdeux)\notin\lbrace(0,2),(2,2)\rbrace$ if $c=0$;
    \item $0\leq\sun<2$ and $0\leq\sdeux\leq\sdeux^\star(\sun)$ if $0<\vert c\vert<1$;
    \item $0\leq\sun,\sdeux\leq2$ and $(\sun,\sdeux)\notin\lbrace(0,2),(2,0),(2,2)\rbrace$ if $\vert c\vert=1$;
    \item $\sun=\sdeux=0$ if $\vert c\vert > 1$.
\end{itemize}
The limit value $\sdeux^\star:[0,2]\to[0,2]$ reads as a bijective decreasing function given by
\begin{equation}
    \label{eq:defsstar}
    \sdeux^\star(s) = 
    1 + \frac{s^2-8s+8}{
        \sqrt{(s^2-8s+8)^2+4s^2(2-s)^2} + 2s(2-s)
    }.
\end{equation}
\end{theorem}

\begin{figure}
    \includegraphics[width=\textwidth]{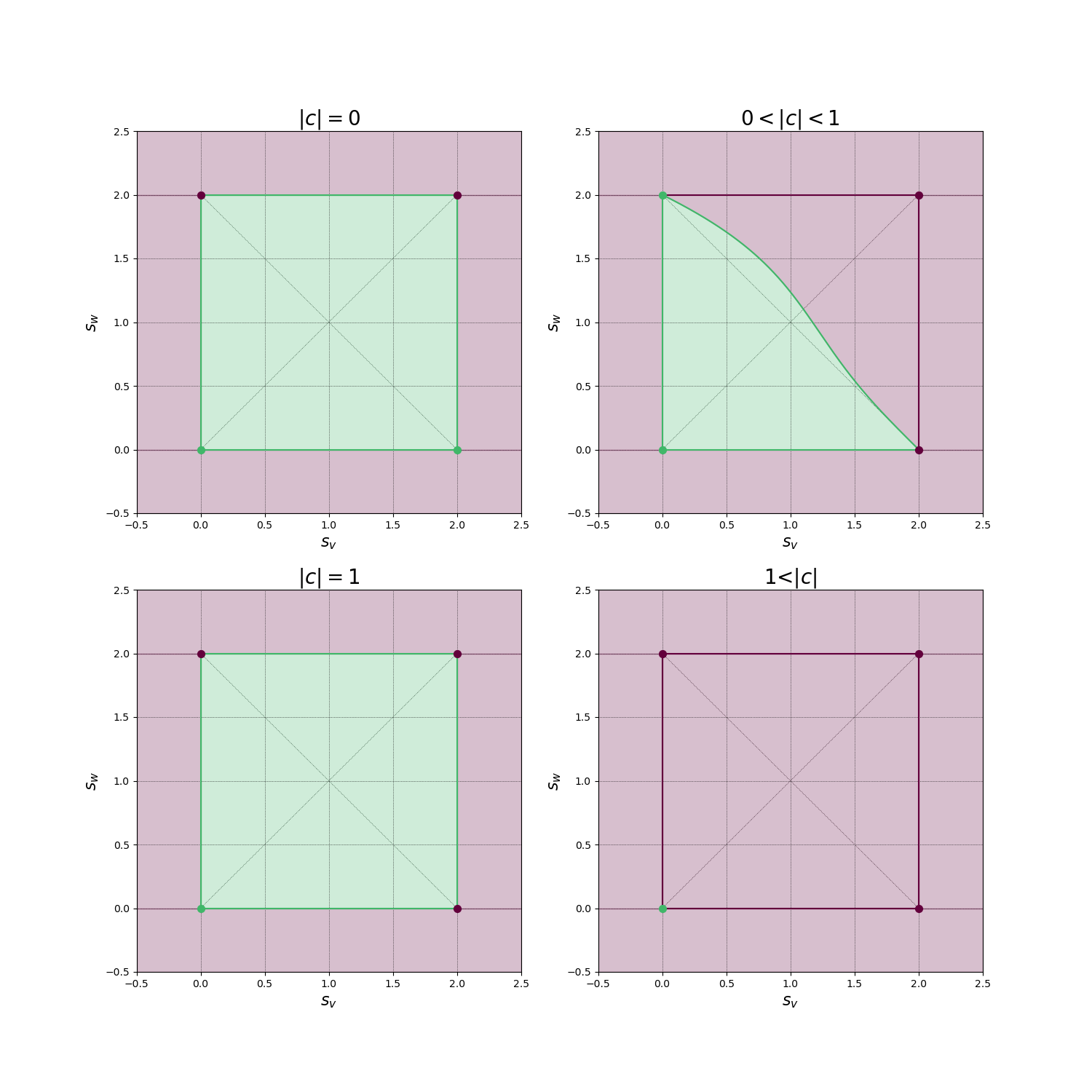}
    \caption{Representation of the stability condition of Thm.~\ref{thm:main}: green means stable while purple means unstable}
    \label{fig:condtot}
\end{figure}

Fig.~\ref{fig:condtot} shows the stable area according to the relaxation parameters $\sun$ and $\sdeux$. This area does not depend on the velocity $c$ as long as it satisfies $0<\vert c\vert<1$. 

\begin{proof}[Proof of Thm.~\ref{thm:main}]
    The proof simply consists of applying,
    \begin{itemize}
        \item Prop.~\ref{thm:stabvertexlb} for $\sun=0$ and $\sdeux=0$, Prop.~\ref{thm:stabvertexrb} for $\sun=2$ and $\sdeux=0$, Prop.~\ref{thm:stabvertexlt} for $\sun=0$ and $\sdeux=2$, Prop.~\ref{thm:stabvertexrt} for $\sun=2$ and $\sdeux=2$;
        \item Prop.~\ref{thm:stabedgeb} for $0<\sun<2$ and $\sdeux=0$, Prop.~\ref{thm:stabedgel} for $\sun=0$ and $0<\sdeux<2$, Prop.~\ref{thm:stabedget} for $0<\sun<2$ and $\sdeux=2$, Prop.~\ref{thm:stabedger} for $\sun=2$ and $0<\sdeux<2$;
        \item Prop.~\ref{thm:stabsquare} for $0<\sun<2$ and $0<\sdeux<2$;
    \end{itemize}
    keeping in mind that $\sun=\tun+1$ and $\sdeux=\tdeux+1$.
\end{proof}

\begin{corollary}\label{thm:corSRT}
    The single relaxation time scheme obtained by fixing $\sdeux=\sun$ is stable if the unique relaxation parameter $s=\sun=\sdeux$ satisfies
    \begin{itemize}
        \item $0 \leq s \leq s^\star = 6-2\sqrt{6}\gtrapprox 1.1010$, if $0<\vert c\vert < 1$ ;
        \item $0 \leq s < 2$, if $\vert c\vert\in\lbrace 0, 1\rbrace$.
    \end{itemize}
\end{corollary}

\begin{proof}[Proof of Cor.~\ref{thm:corSRT}]
    The limit relaxation parameter $s^\star$ is defined as the unique solution of the fixed point problem $s=\sdeux^\star(s)$ on $[0,2]$ or equivalently $s^\star=t^\star+1$ with $t^\star$ the unique solution of $t=\tdeux^\star(t)$ on $[-1,1]$.
    As $t=3-2\sqrt{2}$ (the unique root of $t^2-6t+1$ in $[-1,1]$) is not a fixed point of $\tdeux^\star$, we have
    \begin{align*}
        t = \tdeux^\star(t) &\iff \left\lbrace
        \begin{aligned}
        &t(t^2-6t+1) = -2(1-t^2) + \sqrt{(t^2-6t+1)^2+4(1-t^2)^2}\\
        &\text{and }-1\leq t\leq 1
        \end{aligned} \right.\\
        &\iff \left\lbrace
        \begin{aligned}
        &(t(t^2-6t+1) + 2(1-t^2))^2 = (t^2-6t+1)^2+4(1-t^2)^2\\
        &\text{and } -1\leq t\leq 1 
        \text{ and }t(t^2-6t+1) + 2(1-t^2) \geq 0
        \end{aligned} \right. \\
        &\iff \left\lbrace
        \begin{aligned}
        &4(1-t^2)t(t^2-6t+1) = (1-t^2)(t^2-6t+1)^2\\
        &\text{and } -1\leq t\leq 1 
        \text{ and }t^3-8t^2+t+2 \geq 0.
        \end{aligned} \right. 
    \end{align*}
    As both $t=\pm1$ do not satisfy the inequality $t^3-8t^2+t+2 \geq 0$, we can simplify the first equality by $1-t^2$ and also by $t^2-6t+1$ (as $3-2\sqrt{2}$ is not a solution)
    that yields
    \begin{align*}
        t = \tdeux^\star(t) 
        &\iff \left(
        t^2-10t + 1 = 0 
        \text{ and }
        -1\leq t\leq 1
        \text{ and }
        t^3-8t^2+t+2 \geq 0
        \right).\\
        &\iff t = 5 - 2\sqrt{6}.
    \end{align*}
    That ends the proof.
\end{proof}

\begin{corollary}\label{thm:corIrina}
    % For $c\in[-1,1]$, the lattice Boltzmann scheme introduced in \cite{dhumieres_viscosity_2009} with the linked relaxation parameters $\sdeux = 2-\sun$ is stable if $0<\sun<2$.
    The lattice Boltzmann scheme introduced in \cite{dhumieres_viscosity_2009} with the linked relaxation parameters $\sdeux = 2-\sun$ is stable if
    \begin{itemize}
        \item $0\leq \sun < 2$ if $0<\vert c\vert <1$ ;
        \item $0<\sun<2$ if $c\in\lbrace 0,1\rbrace$.
    \end{itemize}
\end{corollary}

\begin{proof}[Proof of Cor.~\ref{thm:corIrina}]
    Thanks to Thm.~\ref{thm:main}, we only have to prove that $2-s\leq \sdeux^\star(s)$ for $0\leq s\leq 2$ or equivalently that $t+\tdeux^\star(t) \geq 0$ for $-1\leq t\leq 1$. We first consider the obvious case $t=3-2\sqrt{2}>0$ with $\tdeux^\star(t)=0$.
    For $t\neq 3-2\sqrt{2}$, we rewrite $t + \tdeux^\star(t)$ as
    \begin{equation*}
        t + \tdeux^\star(t) = \sqrt{1+x^2} + 1-x
        \quad\text{with}\quad
        x = 2(1-t^2)/(t^2-6t+1). 
    \end{equation*}
    % with $x = 2(1-t^2)/(t^2-6t+1)$. 
    The function $x\mapsto \sqrt{1+x^2} + 1-x$ is clearly decreasing and positive on $\mathbb{R}$, what ends the proof.
\end{proof}
%%%%%%%%%%%%%%%%%%%%%%%%%%%%%%%%%%%%%%%%%%%%%%%%%%

%%%%%%%%%%%%%%%%%%%%%%%%%%%%%%%%%%%%%%%%%%%%%%%%%%
% SECTION PROOF OF STABILITY
%%%%%%%%%%%%%%%%%%%%%%%%%%%%%%%%%%%%%%%%%%%%%%%%%%
%%%%%%%%%%%%%%%%%%%%%%%%%%%%%%%%%%%%%%%%%%%%%%%%%%
% SECTION PROOF OF STABILITY
%%%%%%%%%%%%%%%%%%%%%%%%%%%%%%%%%%%%%%%%%%%%%%%%%%
\section{Proof of the stability condition}\label{sec:proofs}

This section is devoted to the study of the main theorem. We first restrict the domain in which the parameters $(\tun, \tdeux, c)$ lie thanks to symmetry properties and we introduce some notations to facilitate a complete proof. The computations were carried out with the help of SageMath \cite{sagemath}. 
The computer-assisted calculations accompanying this work are available at \cite{bellotti_2026_23390935}.

We recall that we have defined \( \polyPt \) as the characteristic polynomial of the amplification matrix \( \matG \) given in \eqref{eq:matGs}. The subscript 3 indicates that the degree of this polynomial is 3.
Then we note  \(\polyPdcs=\cohnschur(\polyPt)\) the polynomial obtained after one step of the Cohn-Schur operator and \(\polyPdde=\mathrm{d}\polyPt/\mathrm{d}X\) the polynomial obtained after the derivative operator. The subscript 2 indicates that the degree of these polynomials is 2. We continue this process by defining polynomials of degree 1
$\polyPucscs=\cohnschur(\polyPdcs)$,
$\polyPucsde=\mathrm{d}\polyPdcs/\mathrm{d}X$,
$\polyPudecs=\cohnschur(\polyPdde)$, and
$\polyPudede=\mathrm{d}\polyPdde/\mathrm{d}X$.
The expression of each of these polynomials is given in the appendix~\ref{sec:exprpoly}.

%%%%%%%%%%%%%%%%%%%%%%%%%%%%%%%%%%%%%%%%%%%%%%%%%%
% DECOMPOSITION OF THE DOMAIN
%%%%%%%%%%%%%%%%%%%%%%%%%%%%%%%%%%%%%%%%%%%%%%%%%%
\subsection{Decomposition of the studied domain}

Thanks to theorem~\ref{thm:cond_sc}, we assume in this section that 
\begin{equation*}
    \vert\tun\vert \leq 1,
    \quad
    \vert\tdeux\vert \leq 1,
    \quad
    \vert c\vert \leq 1.
\end{equation*}

The following property is usefull to reduce the complexity of the proof.

\begin{property}\label{thm:symprop}
    The matrix $\matG$ defined in \eqref{eq:matG} 
    depends on the $\tun$, $\tdeux$, $\xi$, and $c$ variables and satisfies the symmetry properties:
    \begin{equation*}
        \matG(\xi,-c,\tun,\tdeux) =
        \matP \overline{\matG(\xi,c,\tun,\tdeux)} \matP,
        \quad
        \matG(-\xi,c,\tun,\tdeux) =
        \overline{\matG(\xi,c,\tun,\tdeux)}
    \end{equation*}
    with
    \begin{equation*}
        \matP = \begin{pmatrix}
        1&0&0 \\ 0&0&1 \\ 0&1&0
        \end{pmatrix}
    \end{equation*}
    for $-\pi\leq\xi\leq\pi$, $-1\leq c\leq1$, and $-1\leq\tun,\tdeux\leq1$.
    The characteristic polynomial $\polyPt$ and the minimal polynomial $\polymini$ satisfy
    \begin{align*}
        \polyPt(X;\xi,-c,\tun,\tdeux) &= \overline{\polyPt(X;\xi,c,\tun,\tdeux)},
        &
        \polyPt(X;-\xi,c,\tun,\tdeux) &= \overline{\polyPt(X;\xi,c,\tun,\tdeux)} \\
        \polymini(X;\xi,-c,\tun,\tdeux) &= \overline{\polymini(X;\xi,c,\tun,\tdeux)},
        &
        \polymini(X;-\xi,c,\tun,\tdeux) &= \overline{\polymini(X;\xi,c,\tun,\tdeux)},
    \end{align*}
    for $-\pi\leq\xi\leq\pi$, $-1\leq c\leq1$, and $-1\leq\tun,\tdeux\leq1$.
\end{property}

We then investigate the stability condition assuming that the parameters lie in the hypercube
\begin{equation}\label{eq:condinit}
    -1 \leq \tun \leq 1,
    \quad
    -1 \leq \tdeux \leq 1,
    \quad
    0 \leq c \leq 1,
    \quad
    0 \leq \xi \leq \pi.
\end{equation}

In order to split the proof, we decompose the square domain
\begin{equation}\label{eq:domaintvtw}
    \squarett = [-1,1]{\times}[-1,1]
\end{equation}
where the centered relaxation parameters $\tun$ and $\tdeux$ are taken into its interior $\squarettinterior$, its open edges $\edgestt$, and its vertices $\verticestt$ (see Fig.~\ref{fig:fig:notations_stt}):
\begin{equation}\label{eq:domaintvtwdecomp}
    \squarett = \squarettinterior \cup \edgestt \cup \verticestt
\end{equation}
where
\begin{gather*}
    \squarettinterior = (-1,1){\times}(-1,1),\\
    \begin{aligned}
        \edgestt &= \edgestt^b \cup \edgestt^t \cup \edgestt^l \cup \edgestt^r,
        &
        \verticestt &= \verticestt^{lb} \cup \verticestt^{rb} \cup \verticestt^{lt} \cup \verticestt^{rt},
        \\ 
        \edgestt^b &=(-1,1){\times}\lbrace-1\rbrace,
        &
        \verticestt^{lb} &= \lbrace-1\rbrace{\times}\lbrace-1\rbrace,
        \\
        \edgestt^t &=(-1,1){\times}\lbrace+1\rbrace,
        &
        \verticestt^{rb} &= \lbrace+1\rbrace{\times}\lbrace-1\rbrace,
        \\
        \edgestt^l &=\lbrace-1\rbrace{\times}(-1,1),
        &
        \verticestt^{lt} &= \lbrace-1\rbrace{\times}\lbrace+1\rbrace,
        \\
        \edgestt^r &=\lbrace+1\rbrace{\times}(-1,1),
        &
        \verticestt^{rt} &= \lbrace+1\rbrace{\times}\lbrace+1\rbrace.
    \end{aligned}
\end{gather*}

\begin{figure}[htb]
    \includegraphics[width=.5\textwidth]{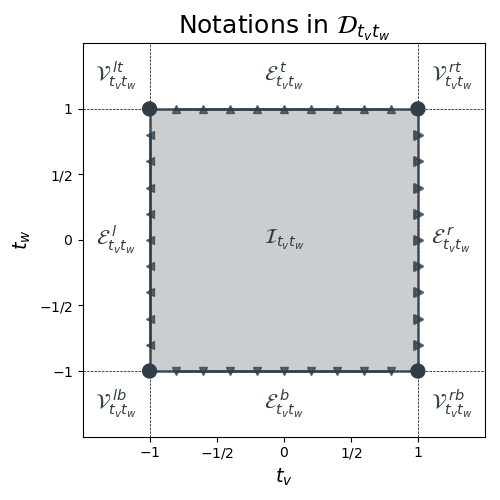}
    \caption{Decomposition of the square $\squarett = [-1,1]{\times}[-1,1]$.}
    \label{fig:fig:notations_stt}
\end{figure}

We also decompose the rectangular domain
\begin{equation}\label{eq:domaincxi}
    \squarecxi = [0,1]{\times}[0,\pi]
\end{equation}
where the velocity $c$ and the wave number $\xi$ are taken 
into its interior $\squarecxiinterior$, its open edges $\edgescxi$, and its vertices $\verticescxi$ (see Fig.~\ref{fig:notations_scxi}):
\begin{equation}\label{eq:domaincxidecomp}
    \squarecxi = \squarecxiinterior \cup \edgescxi \cup \verticescxi
\end{equation}
where
\begin{gather*}
    \squarecxiinterior = (0,1){\times}(0,\pi),\\
    \begin{aligned}
        \edgescxi &= \edgescxi^b \cup \edgescxi^t \cup \edgescxi^l \cup \edgescxi^r,
        &
        \verticescxi &= \verticescxi^{lb} \cup \verticescxi^{rb} \cup \verticescxi^{lt} \cup \verticescxi^{rt},
        \\ 
        \edgescxi^b &=(0,1){\times}\lbrace0\rbrace,
        &
        \verticescxi^{lb} &= \lbrace0\rbrace{\times}\lbrace0\rbrace,
        \\
        \edgescxi^t &=(0,1){\times}\lbrace\pi\rbrace,
        &
        \verticescxi^{rb} &= \lbrace1\rbrace{\times}\lbrace0\rbrace,
        \\
        \edgescxi^l &=\lbrace0\rbrace{\times}(0,\pi),
        &
        \verticescxi^{lt} &= \lbrace0\rbrace{\times}\lbrace\pi\rbrace,
        \\
        \edgescxi^r &=\lbrace1\rbrace{\times}(0,\pi),
        &
        \verticescxi^{rt} &= \lbrace1\rbrace{\times}\lbrace\pi\rbrace.
    \end{aligned}
\end{gather*}

\begin{figure}[htb]
    \includegraphics[width=.5\textwidth]{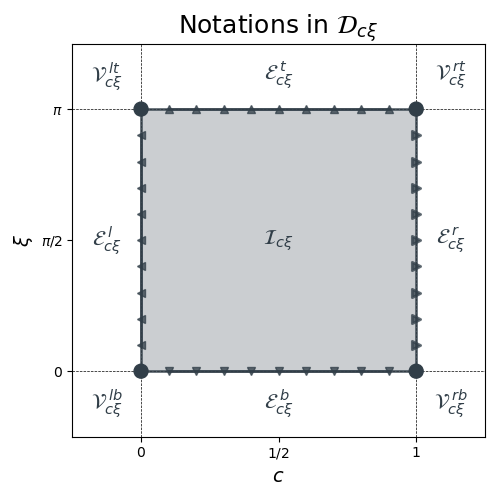}
    \caption{Decomposition of the square $\squarecxi = [0,1]{\times}[0,\pi]$.}
    \label{fig:notations_scxi}
\end{figure}

%%%%%%%%%%%%%%%%%%%%%%%%%%%%%%%%%%%%
% Technical lemmas
%%%%%%%%%%%%%%%%%%%%%%%%%%%%%%%%%%%%
\subsection{Technical lemmas}
%%%%%%%%%%%%%%%%%%%%%%%%%%%%%%%%%%%%

In this section, we prove some technical lemmas that are used in the sequel to investigate the stability area.
In the first two lemmas, we characterize the membership properties of the polynomial \(\polyPt\) in the sets \( \ensinegalite \) and \( \ensegalite \).

\begin{lemma}\label{thm:P3I}
    % For $\xi\in[-\pi,\pi]$ and $(c, \tun, \tdeux) \in[-1,1]^3$, we have
    For $(\tun, \tdeux)\in\squarett$ and $(c,\xi)\in\squarecxi$, we have
    \begin{equation*}
        \polyPt(\cdot;\xi,c,\tun,\tdeux)\in\ensinegalite
        \iff
        (\tun,\tdeux)\in\squarettinterior \cup \edgestt.
    \end{equation*}
\end{lemma}

\begin{proof}[Proof of lemma~\ref{thm:P3I}]
% The proof consists only in writing
 $\polyPt\in\ensinegalite$ if, and only if, $\vert\alphatz\vert<\vert\alphatt\vert$ that reads 
$\vert\tun\tdeux\vert<1$.
\end{proof}

\begin{lemma}\label{thm:P3E}
    For $(\tun, \tdeux)\in\squarett$ and $(c,\xi)\in\squarecxi$, we have
    \begin{equation*}
        \polyPt(\cdot;\xi,c,\tun,\tdeux)\in\ensegalite
        \iff
        (\tun,\tdeux)\in\verticestt.
    \end{equation*}
\end{lemma}

\begin{proof}[Proof of lemma~\ref{thm:P3E}]
    By lemma~\ref{thm:P3I}, and since the intersection $\ensinegalite\cap\ensegalite$ is empty, it follows that
    \begin{equation*}
        \polyPt(\cdot;\xi,c,\tun,\tdeux)\in\ensegalite
        \implies
        (\tun,\tdeux)\in\verticestt.
    \end{equation*}
    Denoting $\polyPdcs=\cohnschur(\polyPt)$,  
    we have $\polyPdcs(\cdot;\xi,c,\tun,\tdeux)\equiv0$ if, and only if, 
    \begin{equation*}
        \alphatz\,\alphatzbar = \alphatt\,\alphattbar,
        \quad \alphattbar\,\alphatu = \alphatz\,\alphatdbar
        \quad\text{and}\quad
        \alphattbar\,\alphatd = \alphatz\,\alphatubar .
    \end{equation*}
    It is obvious that two of these three relations imply the third. The previous equalities yield as $\alphatt=1$
    \begin{equation*}
        \alphatz\,\alphatzbar = 1 
        \quad\text{and}\quad
        \alphatu  = \alphatz\,\alphatdbar .
    \end{equation*}
    The condition $\vert\alphatz\vert=1$ reads $\vert\tun\tdeux\vert=1$. 
    First, we assume that $\tun\tdeux=1$. As $(\tun,\tdeux)\in[-1,1]^2$, we have $\tun=\tdeux$ and
    \begin{equation*}
        \alphatu  - \alphatz\,\alphatdbar
        = c(\tdeux-\tun) \bigl(
            c(1-\cos\xi) + i\sin\xi
        \bigr) = 0.
    \end{equation*}
    Second, we assume that $\tun\tdeux=-1$. As $(\tun,\tdeux)\in[-1,1]^2$, we have $\tun=-\tdeux$ and
    \begin{equation*}
        \alphatu  - \alphatz\,\alphatdbar
        = -(\tdeux+\tun) \bigl(
            c^2 + (2-c^2)\cos\xi- ic\sin\xi
        \bigr) = 0.
    \end{equation*}
    That ends the proof.
\end{proof}

The next two lemmas are dedicated to the characterization of membership for the polynomial \( \polyPdcs = \cohnschur(\polyPt) \) obtained after one step of the Cohn-Schur operator.

\begin{lemma}\label{thm:P2csI}
    For $(\tun,\tdeux)\in\squarettinterior\cup\edgestt$ and $(c,\xi)\in\squarecxi$, we have
\begin{equation*}
    \polyPdcs(\cdot;\xi,c,\tun,\tdeux)\in\ensinegalite \iff 
    \left\lbrace
    \begin{aligned}
        &(\tun,\tdeux) \in \squarettinterior,\\
        \text{or}\\
        &(\tun,\tdeux) \in \edgestt^l\cup\edgestt^r \text{ and }
        (c,\xi) \in \squarecxiinterior\cup\edgescxi^l\cup\edgescxi^t,\\
        \text{or}\\
        &(\tun,\tdeux) \in \edgestt^{b}\cup\edgestt^{t} \text{ and } (c,\xi) \in \squarecxiinterior\cup\edgescxi^l.
    \end{aligned}
    \right.
\end{equation*}
\end{lemma}

\begin{proof}[Proof of lemma~\ref{thm:P2csI}]
    The polynomial $\polyPdcs=\cohnschur(\polyPt)$ is given in \eqref{eq:ann_defP2cs}. Then, the condition $\polyPdcs(\cdot;\xi,c,\tun,\tdeux)\in\ensinegalite$ reads
    \begin{equation*}
        \polycdcs(\xi, c, \tun, \tdeux) = 
        \betadcsz + \betadcsu c^2 + \betadcsd c^4 > 0,
    \end{equation*}
    with
    \begin{align*}
        \betadcsz &= (1-\tun \tdeux)^2 
        \bigl( (1+\tun \tdeux)^2 - \cos^2(\xi)(\tun+\tdeux)^2 \bigr),\\
        \betadcsu &= \begin{multlined}[t]
            (\cos\xi-1) \big( 
                \tdeux^2 (1 - \tun^2)^2 + \cos\xi \big(
                \tun^4\tdeux^2 
                + 2\tun^3\tdeux^3 
                + 2\tun^2\tdeux^4 
                - 2\tun^3\tdeux 
                - 6\tun^2\tdeux^2 \\ 
                - 2\tun \tdeux^3 
                + 2\tun^2 
                + 2\tun \tdeux 
                + \tdeux^2 
                \big)
            \big),
        \end{multlined}\\
        \betadcsd &= - (1-\tdeux^2)^2 \tun^2 (1-\cos\xi)^2.
    \end{align*}

    First, we consider $(\tun,\tdeux)\in\squarettinterior$, that is $\vert\tun\vert<1$ and $\vert\tdeux\vert<1$.
    We observe that $\polycdcs$ can be seen as a polynomial of degree 2 in the indeterminate $c^2$. As $\betadcsd\leq0$, this polynomial is concave.
    We then compute its value in $c^2=0$ and $c^2=1$:
    \begin{align*}
        \polycdcs(\xi,0,\tun,\tdeux) &= \betadcsz =
        (1-\tun\tdeux)^2\bigl(
            (1-\tun^2)(1-\tdeux^2) 
            + (1-\cos^2(\xi))(\tun+\tdeux)^2
        \bigr),\\
        \polycdcs(\xi,1,\tun,\tdeux) &= \betadcsz + \betadcsu + \betadcsd\\
        &= (1-\tun^2)(1-\tdeux^2)
        \bigl(
            (1-\cos\xi\tun\tdeux)^2+(1-\cos^2(\xi))\tun^2\tdeux^2
        \bigr).
    \end{align*}
    % First $\polycdcs(\xi,0,\tun,\tdeux)\geq0$ and $\polycdcs(\xi,0,\tun,\tdeux)>0$ if, and only if, $\xi\notin\lbrace0,\pm\pi\rbrace$ or $(\vert\tun\vert<1\text{ and }\vert\tdeux\vert<1)$.
    % Second $\polycdcs(\xi,1,\tun,\tdeux)\geq0$ and $\polycdcs(\xi,1,\tun,\tdeux)>0$ if, and only if, $\vert\tun\vert<1\text{ and }\vert\tdeux\vert<1$.
    % We obtain that, if $\vert\tun\vert<1\text{ and }\vert\tdeux\vert<1$, then the condition $\polyPdcs(\cdot;\xi,c,\tun,\tdeux)\in\ensinegalite$ is verified.
    % Moreover, if $\vert\tun\vert=1$ or $\vert\tdeux\vert=1$ the condition is not verified if $c^2=1$ or if $\xi\in\lbrace0,\pm\pi\rbrace$, since we have
    % \begin{equation*}
    %     \polycdcs(0,c,\tun,\tdeux) = (1-\tun\tdeux)^2
    %         (1-\tun^2)(1-\tdeux^2) .
    % \end{equation*}
    % That ends the proof.
    % {\afaire
    % Il faut corriger ce lemme en utilisant
    % \begin{equation*}
    %     \polycdcs(\pm\pi,c,\pm1,\tdeux) = 4(1-\tdeux^2)^2 c^2(1-c^2),
    %     \quad
    %     \polycdcs(\pm\pi,c,\tun,\pm1) = 0.
    % \end{equation*}
    % }
    We have
    $\polycdcs(\xi,0,\tun,\tdeux)>0$ and $\polycdcs(\xi,1,\tun,\tdeux)>0$, the condition $\polycdcs(\xi,c,\tun,\tdeux)>0$ is then satisfied for $(c,\xi)\in\squarecxi$.

    Second, we consider $(\tun,\tdeux)\in\edgestt^l\cup\edgestt^r$, that is $\vert\tun\vert=1$ and $\vert\tdeux\vert<1$. We have
    \begin{equation*}
        \polycdcs(\xi,c,\pm1,\tdeux) = 
        (1-\cos\xi)(1-c^2)(1-\tdeux^2)^2(1+\cos\xi+c^2(1-\cos\xi)).
    \end{equation*}
    The condition $\polycdcs(\xi,c,\pm1,\tdeux)>0$ is then satisfied for $(c,\xi)\in\squarecxiinterior\cup\edgescxi^l\cup\edgescxi^t$.

    Third, we consider $(\tun,\tdeux)\in\edgestt^b\cup\edgestt^t$, that is $\vert\tun\vert<1$ and $\vert\tdeux\vert=1$. We have
    \begin{equation*}
        \polycdcs(\xi,c,\tun,\pm1) = 
        (1-\cos^2\xi)(1-c^2)(1-\tun^2)^2.
    \end{equation*}
    The condition $\polycdcs(\xi,c,\tun,\pm1)>0$ is then satisfied for $(c,\xi)\in\squarecxiinterior\cup\edgescxi^l$.
\end{proof}

\begin{lemma}\label{thm:P2csE}
%     For $\xi\in[-\pi,\pi]$ and $(c, \tun, \tdeux) \in[-1,1]^3$ such that $\vert\tun\tdeux\vert<1$, we have
% \begin{equation*}
%     \polyPdcs(\cdot;\xi,c,\tun,\tdeux)\in\ensegalite \iff \left\lbrace
%     \begin{aligned}
%         &\vert\tun\vert=1 \text{ and } \left\lbrace\begin{aligned}
%             &\xi=0,\\
%             \text{or}\\
%             &c^2=1,\\
%             \text{or}\\
%             &\xi=\pm\pi \text{ and } c=0;
%         \end{aligned} \right.\\
%         \text{or}\\
%         &\vert\tdeux\vert=1 \text{ and } \left\lbrace\begin{aligned}
%             &\xi\in\lbrace-\pi,0,\pi\rbrace,\\
%             \text{or}\\
%             &c^2=1.
%         \end{aligned} \right.\\
%     \end{aligned}
%     \right.
% \end{equation*}
    For $(\tun,\tdeux)\in\squarettinterior\cup\edgestt$ and $(c,\xi)\in\squarecxi$, we have
    \begin{equation*}
        \polyPdcs(\cdot;\xi,c,\tun,\tdeux)\in\ensegalite \iff \left\lbrace
        \begin{aligned}
            &\begin{multlined}
                (\tun,\tdeux)\in\edgestt^{l}\cup\edgestt^r \\
                \text{ and } (c,\xi)\in 
                \edgescxi^b\cup\edgescxi^r\cup\verticescxi,
            \end{multlined} \\
            \text{or}\\
            &\begin{multlined}
                (\tun,\tdeux)\in\edgestt^{b}\cup\edgestt^t \\
                \text{ and } (c,\xi)\in 
                \edgescxi^b\cup\edgescxi^r\cup\edgescxi^t\cup\verticescxi.
            \end{multlined}
        \end{aligned} \right.
    \end{equation*}
\end{lemma}

\begin{proof}[Proof of lemma~\ref{thm:P2csE}]
    We first reduce the investigation to $\vert\tun\vert=1$ or $\vert\tdeux\vert=1$. 
    % Then we prove that this condition is a necessary condition for $\polyPdcs\in\ensegalite$.
    Indeed, as the intersection $\ensinegalite\cap\ensegalite$ is empty, we necessarely have 
    \begin{equation*}
        \polyPdcs(\cdot;\xi,c,\tun,\tdeux)\in\ensegalite
        \implies (\tun,\tdeux)\in\edgestt\cup\verticestt
    \end{equation*}
    thanks to the previous lemma~\ref{thm:P2csI}.
    Note that we assume $\vert\tun\tdeux\vert<1$, then both equalities $\vert\tun\vert=1$ and $\vert\tdeux\vert=1$ are not allowed simultaneously in this proof. The expression of the polynomial $\polyPucscs$ is given in \eqref{eq:ann_defP1cscs}.
    
    Considering $\tun=-1$, the coefficients $\alphaucscsz$ and $\alphaucscsu$ admit these factorizations:
    \begin{align*}
        \alphaucscsz &= 
        (1-\cos\xi)(1-c^2)(1-\tdeux^2)^2(c^2(1-\cos\xi)-(1+\cos\xi)),\\
        \alphaucscsu &= 
        (1-\cos\xi)(1-c^2)(1-\tdeux^2)^2(c^2(1-\cos\xi)+(1+\cos\xi)).
    \end{align*}
    As $\vert\tdeux\vert<1$, the equalities $\alphaucscsz=\alphaucscsu=0$ are then equivalent to
    \begin{equation*}
        \cos\xi=1
        \text{ or }
        c^2=1
        \text{ or }
        (c=0\text{ and }\cos\xi=-1).
    \end{equation*}
    Considering $\tun=1$, the coefficients $\alphaucscsz$ and $\alphaucscsu$ admit these factorizations:
    \begin{gather*}
        % \Re(\alphaucscsz) &= 
        % -(1-\cos\xi)(1-c^2)(1-\tdeux^2)^2(1+\cos\xi+c^2(1-\cos\xi)),\\
        % \Im(\alphaucscsz) &= 
        % 2\sin\xi(1-\cos\xi)c(1-c^2)(1-\tdeux^2)^2,
        % \\
        % \alphaucscsu &= 
        % (1-\cos\xi)(1-c^2)(1-\tdeux^2)^2(1+\cos\xi+c^2(1-\cos\xi)).
        \alphaucscsu = - \Re(\alphaucscsz) = 
        (1-\cos\xi)(1-c^2)(1-\tdeux^2)^2(1+\cos\xi+c^2(1-\cos\xi)),\\
        \Im(\alphaucscsz) = 
        2\sin\xi(1-\cos\xi)c(1-c^2)(1-\tdeux^2)^2.
    \end{gather*}
    As $\vert\tdeux\vert<1$, the equalities $\alphaucscsz=\alphaucscsu=0$ are then equivalent to
    \begin{equation*}
        \cos\xi=1
        \text{ or }
        c^2=1
        \text{ or }
        (c=0\text{ and }\cos\xi=-1).
    \end{equation*}
    Considering now $\tdeux=-1$, the coefficients $\alphaucscsz$ and $\alphaucscsu$ admit these factorizations:
    \begin{equation*}
        \alphaucscsu = -\alphaucscsz = 
        (1-\cos^2\xi)(1-c^2)(1-\tun^2)^2.
    \end{equation*}
    As $\vert\tun\vert<1$, the equalities $\alphaucscsz=\alphaucscsu=0$ are then equivalent to
    \begin{equation*}
        \cos^2\xi=1
        \text{ or }
        c^2=1.
    \end{equation*}
    Considering finally $\tdeux=1$, the coefficients $\alphaucscsz$ and $\alphaucscsu$ admit these factorizations:
    \begin{gather*}
        \alphaucscsu = -\Re(\alphaucscsz) = 
        (1-\cos^2\xi)(1-c^2)(1-\tun^2)^2,\\
        \Im(\alphaucscsz) = 
        2\sin\xi(1-\cos\xi)c(1-c^2)(1-\tun^2)^2.
    \end{gather*}
    As $\vert\tun\vert<1$, the equalities $\alphaucscsz=\alphaucscsu=0$ are then equivalent to $\cos^2\xi=1$ or $c^2=1$.
\end{proof}

%%%%%%%%%%%%%%%%%%%%%%%%%%%%%%%%%%%%
% stabilité dans le carré ouvert
%%%%%%%%%%%%%%%%%%%%%%%%%%%%%%%%%%%%
\subsection{Stability on \texorpdfstring{$\squarettinterior$}{the open square}}
%%%%%%%%%%%%%%%%%%%%%%%%%%%%%%%%%%%%

In this section, we consider $(\tun,\tdeux)\in\squarettinterior$, that is $\vert\tun\vert<1$ and $\vert\tdeux\vert<1$. We prove in Prop.~\ref{thm:stabsquare} that (see Fig.~\ref{fig:stability_square}):
\begin{itemize}
    \item for $\vert\tun\vert<1$ and $\vert\tdeux\vert<1$, the scheme is stable for $c\in\lbrace0,1\rbrace$;
    \item for $\vert\tun\vert<1$, $\vert\tdeux\vert<1$, and $0<c<1$, the scheme is stable if, and only if $-1<\tdeux\leq\tdeux^\star(\tun)$, the function $\tdeux^\star$ being defined in the sequel.
\end{itemize}

\begin{figure}[htb]
    \includegraphics[width=.33\textwidth]{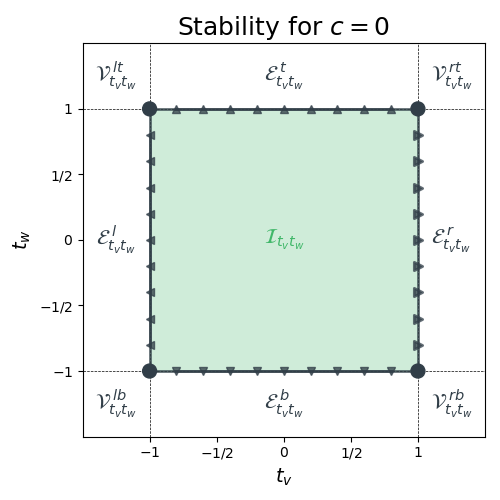}%
    \hfill%
    \includegraphics[width=.33\textwidth]{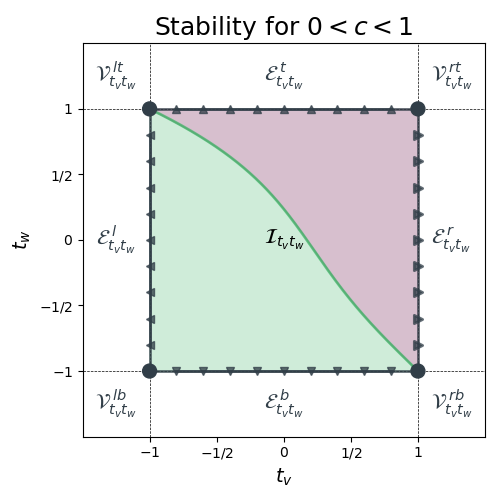}%
    \hfill%
    \includegraphics[width=.33\textwidth]{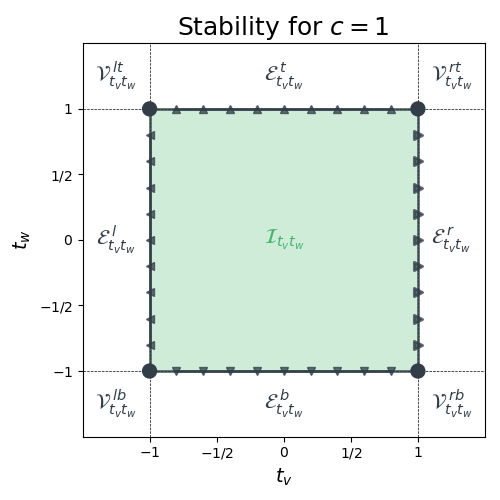}%
    \caption{Stability for $(\tun,\tdeux)\in\squarettinterior$. Green means stable while purple means unstable.}
    \label{fig:stability_square}
\end{figure}

We first give some technical results used to describe the stability area. We introduce a polynomial of degre two of the variables $\tun$ and $\tdeux$ and we investigate its sign. We will see in the sequel that this sign precisely describes the stability area restricted to the open square $(\tun,\tdeux)\in\squarettinterior$. Let's be
\begin{equation*}
    \Upsilon(t,s) = 
    - t^2s^2 + 6ts^2 + 4t^2s - s^2 + t^2 - 4s - 6t + 1 .
\end{equation*}

\begin{figure}
    \includegraphics[width=.45\textwidth]{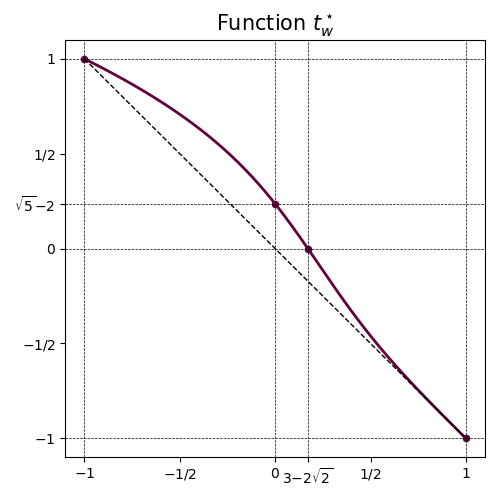}
    \hfill
    \includegraphics[width=.45\textwidth]{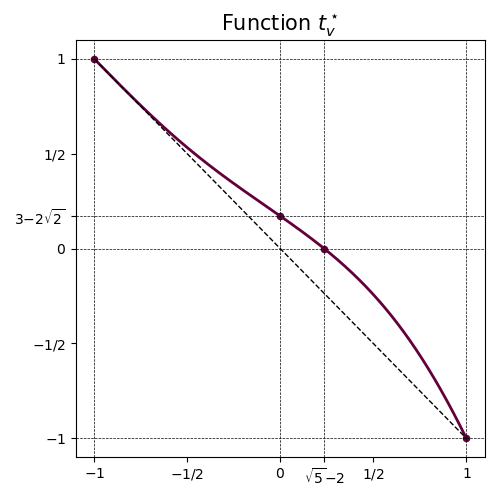}
    \caption{Representation of the functions $\tdeux^\star$ and $\tun^\star$}
    \label{fig:repr_twtv_star}
\end{figure}

\subsubsection{Parametrization according to \texorpdfstring{$\tun$}{first relaxation parameter}}

In this first paragraph, we consider $\Upsilon$ as an element of $\mathbb{R}[t][s]$.
We introduce two polynomial functions $\kappa, \eta:[-1,1]\to\mathbb{R}$ by
\begin{equation*}
    \kappa(t) = 1-6t+t^2,
    \qquad
    \eta(t) = 2(1-t^2),
\end{equation*}
in order to have
\begin{equation*}
    \Upsilon(t, s) = \kappa(t)(1-s^2) - 2\eta(t)s.
\end{equation*}
We immediately have the following table of signs
\begin{center}
    \begin{tikzpicture}
        \tkzTabInit[espcl=4, deltacl=.75]{
            $t$ / .5 , $\kappa(t)$ / .75, $\eta(t)$ / .75
        }{
            $-1$, $3-2\sqrt{2}$, $1$
        }
        \tkzTabLine{ t , + , z , - , t}
        \tkzTabLine{ z , + , t , + , z}
    \end{tikzpicture}
\end{center}

We then define two functions that are the roots of $\Upsilon$ parametrized by the first variable $t$.
The first one $\tdeux^\star:[-1,1]\to\mathbb{R}$ 
and the second one $\tdeux^\diamond:[-1,1]\backslash\lbrace 3-2\sqrt{2}\rbrace\to\mathbb{R}$ read
\begin{equation}
    \label{eq:deftwstar}
    \tdeux^\star(t) = \frac{\kappa(t)}{\sqrt{\eta(t)^2{+}\kappa(t)^2}+\eta(t)},
    \quad
    %\tdeux^\diamond(t) = -\frac{\kappa(t)}{\sqrt{\eta(t)^2+\kappa(t)^2}-\eta(t)}.
    \tdeux^\diamond(t) = -\frac{\sqrt{\eta(t)^2{+}\kappa(t)^2}+\eta(t)}{\kappa(t)}.
\end{equation}

We have the following table of variations proved in Prop.~\ref{thm:tdeuxstardecroissante}.

\begin{center}
    \begin{tikzpicture}
        \tkzTabInit[espcl=4, deltacl=.75]{
            $t$ / .5, $\tdeux^\star(t)$ / .75, $\tdeux^\diamond(t)$ / .75
        }{
            $-1$, $3{-}2\sqrt{2}$, $1$
        }%
        \tkzTabVar{ +/$1$, R , -/$-1$ }%
        \tkzTabIma{1}{3}{2}{0}
        \tkzTabVar{ +/$-1$, -D+/$-\infty$/$+\infty$, -/ $1$ }%
    \end{tikzpicture}
\end{center}

\begin{proposition}\label{thm:tdeuxstardecroissante}
    The functions $\tdeux^\star$ and $\tdeux^\diamond$ defined in \eqref{eq:deftwstar} satisfy the three following properties:
    \begin{itemize}
        \item the function $\tdeux^\star$ is decreasing and bijective from $[-1,1]$ into $[-1,1]$; %(See Fig.~\ref{fig:repr_tw_star});
        \item the function $\tdeux^\diamond$ is decreasing and bijective from $[-1,3{-}\sqrt{2}[$ into $]-\infty,-1]$;
        \item the function $\tdeux^\diamond$ is decreasing and bijective from $]3{-}\sqrt{2},1]$ into $[1,+\infty[$.
    \end{itemize}
\end{proposition}

\begin{proof}[Proof of Prop.~\ref{thm:tdeuxstardecroissante}]
    As $\eta(-1)=\eta(1)=0$, $\kappa(-1)>0$, and $\kappa(1)<0$ , we first have 
    \begin{equation*}
        \tdeux^\star(-1)=1,
        \quad
        \tdeux^\star(1)=-1,
        \quad
        \tdeux^\diamond(-1)=-1,
        \quad
        \tdeux^\diamond(1)=1.
    \end{equation*}
    For $t\in(-1,1)$, the function $\tdeux^\star$ reads
    \begin{equation*}
        \tdeux^\star(t) = \varphi\Bigl(
        \frac{\kappa}{\eta}(t) \Bigr)
        \qquad \text{with }
        \varphi(s) = \frac{s}{\sqrt{1+s^2}+1}.
    \end{equation*}
    The function $\varphi$ is defined for $s\in\mathbb{R}$ and verifies
    \begin{equation*}
        \varphi'(s) = \frac{1}{
            \sqrt{1+s^2}(1+\sqrt{1+s^2})
        } > 0 \quad s\in\mathbb{R},
    \end{equation*}
    the function $\varphi$ is then an increasing function.
    % For $t\in(-1,1)\backslash\lbrace3{-}2\sqrt{2}\rbrace$, the function $\tdeux^\diamond$ reads
    % \begin{equation*}
    %     \tdeux^\diamond(t) = \psi\Bigl(
    %     \frac{\kappa}{\eta}(t) \Bigr)
    %     \qquad \text{with }
    %     \psi(s) = \frac{s}{\sqrt{1+s^2}-1}.
    % \end{equation*}
    % The function $\psi$ is defined for $s\in\mathbb{R}\backslash\lbrace0\rbrace$ and verifies
    % \begin{equation*}
    %     \psi'(s) = \frac{1}{
    %         \sqrt{1+s^2}(\sqrt{1+s^2}-1)
    %     } > 0 \quad s\in\mathbb{R}\backslash\lbrace0\rbrace,
    % \end{equation*}
    % the function $\psi$ is then an increasing function and $\psi(s)$ goes to $\mp\infty$ for $s$ going to $0^\pm$.    
    Moreover we have
    \begin{equation*}
        {\left(\frac{\kappa}{\eta}\right)}'(t) 
        = \frac{\kappa'(t)\eta(t) - \kappa(t)\eta'(t)}{\eta^2(t)} 
        = - \frac{(3t-1)^2 + 8}{3(1-t^2)^2} < 0,
    \end{equation*}
    the function $\kappa/\eta$ is then a decreasing function, what proves the first propertie. The second and third properties are immediately obtained after remarking that $\tdeux^\diamond(t)=-1/\tdeux^\star(t)$.
\end{proof}

For a given $t\in[-1,1]$, the roots of the polynomial $\Upsilon(t,s)$, viewed as a quadratic polynomial in $s$ with coefficients that are themselves quadratic polynomials in $t$, are exactly $\tdeux^\star(t)$ and $\tdeux^\diamond(t)$ (for $t=3{-}2\sqrt{2}$, the only existing root is $\tdeux^\star(t)$).

\subsubsection{Parametrization according to \texorpdfstring{$\tdeux$}{second relaxation parameter}}

In this second paragraph, we consider $\Upsilon$ as an element of $\mathbb{R}[s][t]$.
We introduce two polynomial functions $\mu, \nu:[-1,1]\to\mathbb{R}$ by
\begin{equation*}
    \mu(s) = 1+4s-s^2,
    \qquad
    \nu(s) = 3(1-s^2),
\end{equation*}
in order to have
\begin{equation*}
    \Upsilon(t, s) = \mu(s) t^2 - 2\nu(s) t + \mu(-s).
\end{equation*}
We immediately have the following table of signs
\begin{center}
    \begin{tikzpicture}
        \tkzTabInit[espcl=2.5, deltacl=.75]{
            $s$ / .5 , $\mu(s)$ / .75 , $\mu(-s)$ / .75 , $\nu(s)$ / .75
        }{
            $-1$, $2-\sqrt{5}$, $\sqrt{5}-2$, $1$
        }
        \tkzTabLine{ t , - , z , + , t, +, t}
        \tkzTabLine{ t , + , t , + , z, -, t}
        \tkzTabLine{ z , + , t , + , t, +, z}
    \end{tikzpicture}
\end{center}

We then define two functions that are the roots of $\Upsilon$ parametrized by the second variable $s$.
The first one $\tun^\star:[-1,1]\to\mathbb{R}$ 
and the second one $\tun^\diamond:[-1,1]\backslash\lbrace 2-\sqrt{5}\rbrace\to\mathbb{R}$ read
\begin{equation}
    \label{eq:deftvstar}
    \tun^\star(s) = \frac{\mu(-s)}{\sqrt{\nu(s)^2{-}\mu(s)\mu(-s)}+\nu(s)},
    \quad
    \tun^\diamond(t) = \frac{\sqrt{\nu(s)^2{-}\mu(s)\mu(-s)} + \nu(s)}{\mu(s)}.
\end{equation}

We have the following table of variations proved in Prop.~\ref{thm:tunstardecroissante}.

\begin{center}
    \begin{tikzpicture}
        \tkzTabInit[espcl=4, deltacl=.75]{
            $t$ / .5, $\tun^\star(s)$ / .75, $\tun^\diamond(s)$ / .75
        }{
            $-1$, $\sqrt{5}{-}2$, $1$
        }%
        \tkzTabVar{ +/$1$, R , -/$-1$ }%
        \tkzTabIma{1}{3}{2}{0}
        \tkzTabVar{ +/$-1$, -D+/$-\infty$/$+\infty$, -/ $1$ }%
    \end{tikzpicture}
\end{center}

\begin{proposition}\label{thm:tunstardecroissante}
    The functions $\tun^\star$ and $\tun^\diamond$ defined in \eqref{eq:deftvstar} satisfy the three following properties:
    \begin{itemize}
        \item the function $\tun^\star$ is decreasing and bijective from $[-1,1]$ into $[-1,1]$;
        \item the function $\tun^\diamond$ is decreasing and bijective from $[-1,\sqrt{5}{-}2[$ into $]-\infty,-1]$;
        \item the function $\tun^\diamond$ is decreasing and bijective from $]\sqrt{5}{-}2,1]$ into $[1,+\infty[$.
    \end{itemize}
\end{proposition}

\begin{proof}[Proof of Prop.~\ref{thm:tunstardecroissante}]
    As $\nu(-1)=\nu(1)=0$, $\mu(-1)=-4$, and $\mu(1)=4$ , we first have 
    \begin{equation*}
        \tun^\star(-1)=1,
        \quad
        \tun^\star(1)=-1,
        \quad
        \tun^\diamond(-1)=-1,
        \quad
        \tun^\diamond(1)=1.
    \end{equation*}
    For $s\in(-1,1)$, we have
    \begin{equation*}
    \tun^{\star\prime}(s) = - \frac{
        12(1-s+s^2)
    }{
        (\sqrt{8+8s^4}+3(1-s^2))^2
    } - \frac{
        16(1-s^2)(2-s+2s^2)
    }{
        (\sqrt{8+8s^4}+3(1-s^2))^2\sqrt{8+8s^4}
    } < 0,
    \end{equation*}
    what proves the first property. The second and third properties are immediately obtained after remarking that
    $\tun^\diamond(s) = 1/\tun^\star(-s)$.
\end{proof}

\subsubsection{Description of the stability area in \texorpdfstring{$\squarettinterior$}{the open square}}

The following proposition concerns the caracterization of the sign of $\Upsilon(t, s)$.

\begin{proposition}\label{thm:techlemtstar}
    The function $\Upsilon:[-1,1]\times[-1,1]\to\mathbb{R}$ defined by
    \begin{equation*}
        \Upsilon(t,s) = 
        % \kappa(t)(1-s^2)-2\eta(t)s,
        - t^2s^2 + 6ts^2 + 4t^2s - s^2 + t^2 - 4s - 6t + 1 ,
    \end{equation*}
    satisfies the two properties
    \begin{gather*}
        \Upsilon(t, s) \geq 0
        \iff -1 \leq s \leq \tdeux^\star(t)
        \iff -1 \leq t \leq \tun^\star(s),\\
        \Upsilon(t, s) > 0
        \iff -1 \leq s < \tdeux^\star(t)
        \iff -1 \leq t < \tun^\star(s).
    \end{gather*}
\end{proposition}

\begin{proof}[Proof of Prop.~\ref{thm:techlemtstar}]
A direct investigation of the quadratic polynomial $\Upsilon_{\!t}:s\to\Upsilon(t, s)$ for $t\in[-1,1]$ yields to three properties: 
\begin{itemize}
    \item $\Upsilon_{\!t}$ is strictly concave for $t\in[-1,3{-}2\sqrt{2}[$, linear decreasing for $t=3{-}2\sqrt{2}$, and strictly convex for $t\in]3{-}2\sqrt{2},1]$;
    \item if $t\neq 3{-}2\sqrt{2}$, the two roots of $\Upsilon_{\!t}$ are $\tdeux^\star(t)$ and $\tdeux^\diamond(t)$;
    \item if $t= 3{-}2\sqrt{2}$, the only root of $\Upsilon_{\!t}$ is $\tdeux^\star(t)=0$.
\end{itemize}
Combining these properties, we obtain first that, for $-1\leq t<3{-}2\sqrt{2}$, $\Upsilon_{\!t}(s)\geq0$ if, and only if, $\tdeux^\diamond(t)\leq s\leq\tdeux^\star(t)$, second that, for $t=3{-}2\sqrt{2}$, $\Upsilon_{\!t}(s)\geq0$ if, and only if, $s\leq\tdeux^\star(t)$, and third that, for $3{-}2\sqrt{2}<t\leq1$, $\Upsilon_{\!t}(s)\geq0$ if, and only if, $s\leq\tdeux^\star(t)$ or $s\geq\tdeux^\diamond(t)$. The table of variations of $\tdeux^\star$ and $\tdeux^\diamond$ finally ends the proof for the inequalities between $s$ and $\tdeux^\star(t)$.
Composing these inequalities with the decreasing function $\tun^\star$ yields to the second ones as $\tun^\star\circ\tdeux^\star$ is the identity function of $[-1,1]$.
\end{proof}

\begin{proposition}[stability on $\squarettinterior$]
    \label{thm:stabsquare}
    Assuming $\tun\in(-1,1)$ and $\tdeux\in(-1,1)$, we have
    \begin{itemize}
        \item for $c\in\lbrace0,1\rbrace$, the lattice Boltzmann scheme is stable;
        \item for $0<c<1$, the lattice Boltzmann scheme is stable if, and only if the centered relaxation parameters satisfy $-1<\tdeux\leq\tdeux^\star(\tun)$ (or the equivalent relation $-1<\tun\leq\tun^\star(\tdeux)$).
    \end{itemize}
\end{proposition}

\begin{proof}[Proof of Prop.~\ref{thm:stabsquare}]
    In this proof, we assume that $\vert\tun\vert<1$ and $\vert\tdeux\vert<1$, Lemma~\ref{thm:P3I} implies that $\polyPt(\cdot;\xi,c,\tun,\tdeux)\in\ensinegalite$ and Lemma~\ref{thm:P2csI} implies that $\polyPdcs(\cdot;\xi,c,\tun,\tdeux)\in\ensinegalite$. Then Theorems~\ref{thm:strikwerda437} and \ref{thm:strikwerda432} reduce to
    \begin{gather*}
        \polyPt(\cdot;\xi,c,\tun,\tdeux) \in\svNp
        \iff
        \polyPdcs(\cdot;\xi,c,\tun,\tdeux) \in\svNp
        \iff
        \polyPucscs(\cdot;\xi,c,\tun,\tdeux) \in\svNp,
        \\
        \polyPt(\cdot;\xi,c,\tun,\tdeux) \in\vNp
        \iff
        \polyPdcs(\cdot;\xi,c,\tun,\tdeux) \in\vNp
        \iff
        \polyPucscs(\cdot;\xi,c,\tun,\tdeux) \in\vNp.
    \end{gather*}
    As $\polyPucscs$ is a polynomial of degree one, being a simple von Neumann polynomial or a von Neumann polynomial is exactly the same property. Then we have proven that
    \begin{multline*}
        \polyPt(\cdot;\xi,c,\tun,\tdeux) \in\svNp
        \iff
        \polyPt(\cdot;\xi,c,\tun,\tdeux) \in\vNp
        \iff
        \polyPucscs(\cdot;\xi,c,\tun,\tdeux) \in\vNp,
        \\
        (\tun,\tdeux)\in\squarettinterior,
        \ (c,\xi)\in\squarecxi.
    \end{multline*}
    We now investigate the property $\polyPucscs(\cdot;\xi,c,\tun,\tdeux)\in\vNp$.
    The expression of the polynomial $\polyPucscs$ is given in \eqref{eq:ann_defP1cscs}. We then evaluate 
    $\polycucscs(\xi,c,\tun,\tdeux) = \vert\alphaucscsu\vert^2-\vert\alphaucscsz\vert^2$ and $\polyPucscs\in\vNp$ if, and only if, $\polycucscs\geq0$.
    This polynomial admits the factorization
    \begin{equation*}
        \polycucscs(\xi,c,\tun,\tdeux) = 
        c^2(1-c^2)(1+\tdeux)(1-\tun^2\tdeux^2)(1-\cos\xi)^2 
        \polycubarcscs(\xi,c,\tun,\tdeux),
    \end{equation*}
    where $\polycubarcscs$ is a polynomial of degree two according to the variable $\cos\xi$. It reads
    \begin{equation*}
        \polycubarcscs(\xi,c,\tun,\tdeux) = \alpha_0 (1-\cos^2\xi) + \tfrac{1}{2}\alpha_-(1-\cos\xi) + \tfrac{1}{2}\alpha_+(1+\cos\xi),
    \end{equation*}
    with
    \begin{align*}
    \alpha_0 &= (c^2\tun^2(1-\tdeux^2)^2 - (\tun+\tdeux)^2(1-\tun\tdeux)^2) \\
    &\qquad\qquad
    ((1+\tdeux)(1-\tun^2)(1-c^2) + 4\tun(1-\tdeux)), \\
    \alpha_+ &= (1-\tdeux)(1-\tun^2)(1-\tun\tdeux)^2 \Upsilon(\tun,\tdeux), \\
    \alpha_- &= (1+\tdeux)(1-\tun^2)(1-\tdeux)^2(
    (1-\tun)(1-\tun\tdeux)+2c^2\tun(1+\tdeux)
    )^2,
    \end{align*}
    First, we obtain that, for $c\in\lbrace0,1\rbrace$, we have $\polycucscs(\xi,c,\tun,\tdeux) = 0$ and that the scheme is stable. Second, for $0<c<1$, we have
    \begin{equation*}
        \polycucscs(\xi,c,\tun,\tdeux)\geq0
        \iff
        \xi=0
        \text{ or }
        \polycubarcscs(\xi,c,\tun,\tdeux) \geq0,
    \end{equation*}
    as $\vert\tun\vert<1$ and $\vert\tdeux\vert<1$.
    Observing that the target condition $-1<\tdeux\leq\tdeux^\star(\tun)$ does not depend on $c$, we investigate the sign of $\polycubarcscs(\xi,c,\tun,\tdeux)$ first in the case $c=0$ even if this case is already covered for the stability. For $c=0$, $\polycubarcscs(\xi,0,\tun,\tdeux)$ reads
    \begin{equation*}
        \polycubarcscs(\xi,0,\tun,\tdeux) = 
        (1-\tun\tdeux)^2 (\gammabarucscsz + \gammabarucscsu\cos\xi)(\gammaucscsz + \gammaucscsu\cos\xi),
    \end{equation*}
    where $\gammabarucscsz$, $\gammabarucscsu$, $\gammaucscsz$, and $\gammaucscsu$ only depend on $\tun$ and $\tdeux$. The term $\gammabarucscsz + \gammabarucscsu\cos\xi$ is positive for $\xi\in[0,\pi]$ by linearity as 
    \begin{equation*}
        \gammabarucscsz - \gammabarucscsu = (1+\tun)(1+\tdeux) > 0,
        \qquad
        \gammabarucscsz + \gammabarucscsu = (1-\tun)(1-\tdeux) > 0.
    \end{equation*}
    Moreover, we have
    \begin{equation*}
        \gammaucscsz - \gammaucscsu = 
        (1-\tun)^3(1-\tdeux)^2 > 0,
        \qquad
        \gammaucscsz + \gammaucscsu =
        (1+\tun)\Upsilon(\tun,\tdeux),
    \end{equation*}
    where the function $\Upsilon$ is the one defined in Prop.~\ref{thm:techlemtstar}. As an intermediate result, we obtain for $(\tun,\tdeux)\in\squarettinterior$ that
    \begin{equation*}
        \forall\xi\in[0,\pi] \ \polycubarcscs(\xi,0,\tun,\tdeux)\geq0
        \iff
        \Upsilon(\tun,\tdeux)\geq0 
        \iff -1 < \tdeux \leq \tdeux^\star(\tun).
    \end{equation*}
    Finally, we have to prove that the result remains the same for all the values of $c\in[0,1]$. 
    
    We begin the proof by assuming that $\tdeux<\tdeux^\star(\tun)$. We use the fact that the function $\polycubarcscs(\xi,c,\tun,\tdeux)$ is a polynomial of order 2 in the variable $\cos\xi$ which coefficients are polynomials of order 2 in the variable $c^2$. For $c=0$, $\polycubarcscs(\xi,0,\tun,\tdeux)$ has two real roots in $(-\infty,-1)\cup(1,+\infty)$. We prove that the discriminant of $\polycubarcscs(\xi,c,\tun,\tdeux)$ is non negative for $0\leq c\leq 1$ and that $\polycubarcscs(0,c,\tun,\tdeux)>0$ and $\polycubarcscs(\pi,c,\tun,\tdeux)>0$ for $0\leq c\leq 1$. So it implies that, being given a value $0\leq c\leq 1$, the two roots of $\polycubarcscs(\xi,c,\tun,\tdeux)$ are real and remain in $(-\infty,-1)\cup(1,+\infty)$ by using the continuity. First, asuming $-1<\tun<1$ and $-1<\tdeux<\tdeux^\star(\tun)$,
    \begin{align*}
    \polycubarcscs(0,c,\tun,\tdeux) &=
    (1-\tdeux) (1-\tun^2)(1-\tun\tdeux)^2 \Upsilon(\tun,\tdeux) > 0,\\
    \polycubarcscs(\pi,c,\tun,\tdeux) &=
    (1-\tun^2)(1+\tdeux)(1-\tdeux)^2 \\
    & \qquad
    %(2c^2\tun(1+\tdeux)+(1-\tun)(1-\tun\tdeux))^2
    (
        (1-\tun)(1-\tun\tdeux) (1-c^2)
        + (1+\tun)(1+\tun\tdeux) c^2
    )^2 > 0.
    \end{align*}
    Second, the discriminant $\Delta$ of $\polycubarcscs(\xi,c,\tun,\tdeux)$ reads $\Delta = \Delta_1^2 \Delta_2^{}$ with
    \begin{align*}
        \Delta_1^{} &= 2(\tun^3\tdeux + 2\tun^2\tdeux^2 - \tun^2\tdeux + \tun\tdeux - 2\tun - \tdeux),\\
        \Delta_2^{} &= \delta_{20} + \delta_{22}c^2 + \delta_{24}c^4.
    \end{align*}
    We then prove that $\Delta_2^{}\geq0$ in order to conclude that $\Delta\geq0$. We investigate the sign of $\Delta_2^{}$ as a polynomial of degre 2 in the variable $c^2$. 
    % The discriminant $\delta$ of this polynomial reads
    % \begin{equation*}
    %     \delta = (1-\tun)^2 (1+\tun)^4 (1-\tdeux^2)^3 (1-\tun\tdeux)^4 \Upsilon(\tun,\tdeux) > 0
    % \end{equation*}
    % with the assumptions $-1<\tun<1$ and $-1<\tdeux<\tdeux^\star(\tun)$. 
    We observe that the value of $\Delta_2^{}(c^2)$ is bounded by the chord at $c^2=0$ and $c^2=1$ and the two tangents to the curve at the same points. That is
    \begin{equation*}
        \Delta_2^{}(c^2) \geq
        \min\bigl(
            \delta_{20},
            \delta_{20}+\delta_{22}+\delta_{24},
            \delta_{20}+\tfrac{1}{2}\delta_{22}
        \bigr).
    \end{equation*}
    % point d'intersection des tangentes
    % tangente en c=0 : d0 + xi d2
    % tangente en c=1 : d0+d2+d4 + (d2+2d4) (xi-1)
    %                   = d0-d4 + xi (d2+2d4)
    % d0 + xi d2 = d0-d4 + xi (d2+2d4)
    % <=> xi = 1/2
    % valeur au point d'intersection
    % d0 + d2/2
    We have
    \begin{align*}
        \delta_{20} &=
        4(\tun + \tdeux)^2 (1-\tun\tdeux)^4 \geq 0, \\
        \delta_{20}+\delta_{22}+\delta_{24} &=
        (1-\tun^2)^2(1-\tdeux)^2(1+\tun\tdeux^2)^2 \geq 0,\\
        2\delta_{20}+\delta_{22} &=
        (1-\tdeux)(1-\tun^2)(1-\tun\tdeux)^2 \Lambda(\tun,\tdeux),
    \end{align*}
    where $\Lambda$ reads as a polynomial of degree 2 in the variable $\tun$ 
    \begin{equation*}
        \Lambda(\tun,\tdeux) =
        (\tdeux^3-\tdeux^2+3\tdeux+1) \tun^2
        - 2(1+\tdeux)(\tdeux^2-4\tdeux+1) \tun
        + (\tdeux^3+3\tdeux^2-\tdeux+1).
    \end{equation*}
    We prove that the polynomial $\Lambda$ is non negative under the assumptions $-1<\tun<1$ and $-1<\tdeux<\tdeux^\star(\tun)$.
    We compute the discriminant of $\Lambda$
    \begin{equation*}
        \Delta_{\Lambda} = -16\tdeux(1-\tdeux)^2(2+3\tdeux+2\tdeux^2).
    \end{equation*}
    Then, $\Delta_{\Lambda}$ has the opposite sign to $\tdeux$. 
    First case, $0\leq\tdeux<1$: $\Delta_\Lambda\leq0$ implies that $\Lambda(\tun,\tdeux)$ is non negative as $\Lambda(0,\tdeux)=1-\tdeux + \tdeux^2(3+\tdeux)>0$. 
    Second case, $-1<\tdeux<0$: as $\Delta_\Lambda>0$, the polynomial $\Lambda$ has two distinct real roots. As
    \begin{align*}
        \Lambda(-1, \tdeux) &= 4(1+\tdeux)(1-\tdeux)^2 > 0,\\
        \Lambda(1, \tdeux) &= 8 \tdeux(1+\tdeux) < 0,
    \end{align*}
    there is exactly one sign change of the polynomial $\Lambda$ on the interval $[-1,1]$. We write $\tun^{\square}(\tdeux)$ for the zero of $\Lambda$ that lies in $(-1,1)$ whenever $-1 < \tdeux < 0$. Thus, it suffices to show that 
    \begin{equation*}
    \Lambda(\tun^\star(\tdeux),\tdeux) \geq 0,
    \end{equation*}
    which in turn implies that $\Lambda$ is strictly positive on the interval we are concerned with, namely 
    $\tun\in(-1,\tun^\star(\tdeux))$. 
    A direct calculation yields
    {\footnotesize
    \begin{equation*}
        \Lambda(\tun^\star(\tdeux), \tdeux) = \frac{
        4(1+\tdeux)^3
        }{
            \bigl(
                \sqrt{8{+}8\tdeux^4}{+}3(1{-}\tdeux^2)
            \bigr)^2
        } \Bigl(
            3\tdeux^4{+}4\tdeux^3{-}2\tdeux^2{-}4\tdeux{+}3 
            -
            \sqrt{8{+}8\tdeux^4} (1{-}\tdeux{-}\tdeux^2)
        \Bigr).
    \end{equation*}}
    We remark that
    \begin{equation*}
        3\tdeux^4+4\tdeux^3-2\tdeux^2-4\tdeux+3 = 1+\tdeux^4 + 2(1-\tdeux-\tdeux^2)^2 > 0.
    \end{equation*}
    Then, we conclude that $\Lambda(\tun^\star(\tdeux), \tdeux)\geq0$ as
    {\footnotesize
    \begin{equation*}
        (3\tdeux^4+4\tdeux^3-2\tdeux^2-4\tdeux+3)^2 - 8(1+\tdeux^4) (1-\tdeux-\tdeux^2)^2 =
        (1-\tdeux^2)^2 (1+4\tdeux-\tdeux^2)^2 \geq0.
    \end{equation*}}
    At this point, we prove that 
    \begin{equation*}
        \polycubarcscs(\xi,c,\tun,\tdeux)
        \geq 0,
        \quad
        0\leq\xi\leq\pi, \ 
        0\leq c\leq1,\ 
        -1<\tun<1, \ 
        -1<\tdeux<\tdeux^\star(\tun).
    \end{equation*}
    For $\tdeux=\tdeux^\star(\tun)$, the continuity of the function $\polycubarcscs$ according to all its variables immediately implies that the inquality remains true for $-1<\tdeux\leq\tdeux^\star(\tun)$.
    We end the proof by assuming that $\tdeux>\tdeux^\star(\tun)$. Then $\Upsilon(\tun,\tdeux)<0$ implies
    \begin{equation*}
    \polycubarcscs(0,c,\tun,\tdeux) = 
    (1-\tdeux) (1-\tun^2)(1-\tun\tdeux)^2 \Upsilon(\tun,\tdeux) < 0.
    \end{equation*}
    As an final result, we obtain for $(\tun,\tdeux)\in\squarettinterior$ and $0<c<1$ that
    \begin{equation*}
        \forall\xi\in[0,\pi] \ \polycubarcscs(\xi,c,\tun,\tdeux)\geq0
        \iff
        \Upsilon(\tun,\tdeux)\geq0 
        \iff -1 < \tdeux \leq \tdeux^\star(\tun),
    \end{equation*}
    relation which ends the proof.
\end{proof}
%%%%%%%%%%%%%%%%%%%%%%%%%%%%%%%%%%%%
% stabilité sur les arêtes
%%%%%%%%%%%%%%%%%%%%%%%%%%%%%%%%%%%%
\subsection{Stability on \texorpdfstring{$\edgestt$}{the edges}}
%%%%%%%%%%%%%%%%%%%%%%%%%%%%%%%%%%%%

In this section, we consider the four edges $\edgestt^b$, $\edgestt^t$, $\edgestt^l$, and $\edgestt^r$. We prove that (see Fig.~\ref{fig:stability_edges}):
\begin{itemize}
    \item for $\vert\tun\vert<1$ and $\tdeux=-1$, the scheme is stable for $c\in[0,1]$ (see Prop.~\ref{thm:stabedgeb});
    \item for $\tun=-1$ and $\vert\tdeux\vert<1$, the scheme is stable for $c\in[0,1]$ (see Prop.~\ref{thm:stabedgel});
    \item for $\vert\tun\vert<1$ and $\tdeux=1$, the scheme is stable for $c=0$ and $c=1$ and is unstable for $c\in(0,1)$ (see Prop.~\ref{thm:stabedget});
    \item for $\tun=1$ and $\vert\tdeux\vert<1$, the scheme is stable for $c=0$ and $c=1$ and is unstable for $c\in(0,1)$ (see Prop.~\ref{thm:stabedger}).
\end{itemize}

\begin{figure}[htb]
    \includegraphics[width=.33\textwidth]{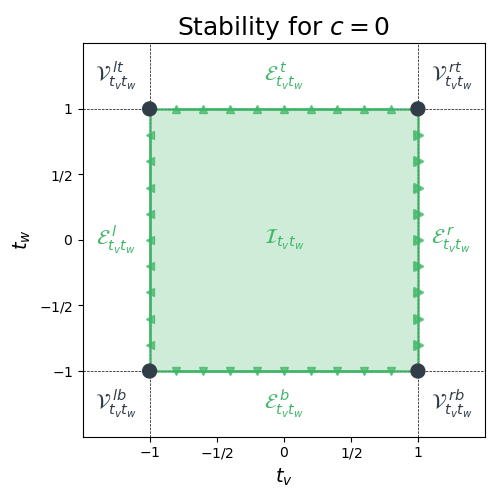}%
    \hfill%
    \includegraphics[width=.33\textwidth]{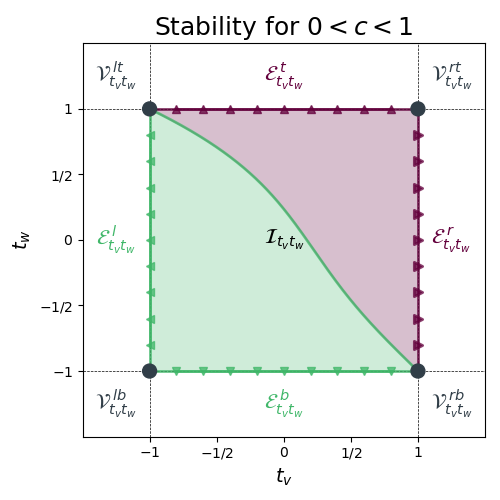}%
    \hfill%
    \includegraphics[width=.33\textwidth]{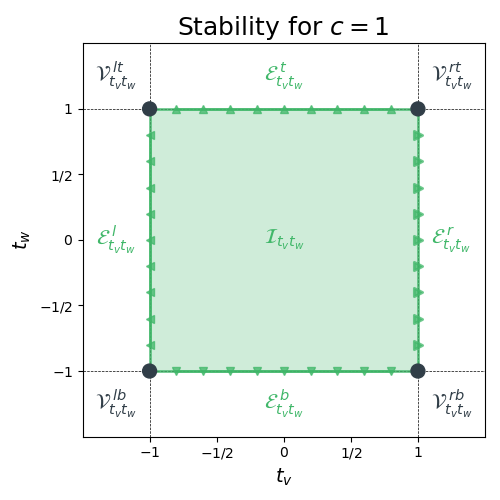}%
    \caption{Stability for $(\tun,\tdeux)\in\edgestt$. Green means stable while purple means unstable.}
    \label{fig:stability_edges}
\end{figure}

\begin{proposition}[stability on $\edgestt^b$]
    \label{thm:stabedgeb}
    Considering $\tun\in(-1,1)$ and $\tdeux=-1$, the lattice Boltzmann scheme is stable for every value $c\in[0,1]$.
\end{proposition}

\begin{proof}[Proof of Prop.~\ref{thm:stabedgeb}]
    In this proof, we assume that $\vert\tun\vert<1$ and $\tdeux=-1$. Lemma~\ref{thm:P3I} then implies that
    $\polyPt(\cdot;\xi,c,\tun,-1)\in\ensinegalite$ for $\xi\in[0,\pi]$, $c\in[0,1]$.
    Then, applying Th.~\ref{thm:strikwerda432}, $\polyPt(\cdot;\xi,c,\tun,-1)\in\svNp$ if, and only if $\polyPdcs(\cdot;\xi,c,\tun,-1)\in\svNp$.
    Lemmas~\ref{thm:P2csI} and \ref{thm:P2csE} read then
    \begin{align*}
        \polyPdcs(\cdot;\xi,c,\tun,-1)\in\ensinegalite &\iff (c,\xi)\in\squarecxiinterior\cup\edgescxi^l,\\
        \polyPdcs(\cdot;\xi,c,\tun,-1)\in\ensegalite &\iff (c,\xi) \in \edgescxi^b\cup\edgescxi^r\cup\edgescxi^t\cup\verticescxi.
    \end{align*}

    First, we consider $(c,\xi)\in\squarecxiinterior\cup\edgescxi^l$, we have
    \begin{equation*}
        \polyPucscs(X;\xi,c,\tun,-1) = 
        (X-1) (1-\tun^2)^2 (1-\cos^2\xi) (1-c^2) \in\svNp.
    \end{equation*}
    Second, we consider $(c,\xi)\in\edgescxi^{b}\cup\edgescxi^{r}\cup\edgescxi^{t}\cup\verticescxi$, we have
    \begin{equation*}
        \polyPucsde(X;\xi,c,\tun,-1) = (2X-1-\cos\xi+\i c\sin\xi)(1-\tun^2) \in \svNp.
    \end{equation*}
    We conclude that $\polyPt(\cdot;\xi,c,\tun,-1)\in\svNp$ for every $(c,\xi)\in\squarecxi$ and then that the scheme is stable for $c\in[0,1]$.
\end{proof}

\begin{proposition}[stability on $\edgestt^l$]
    \label{thm:stabedgel}
    Considering $\tun=-1$ and $\tdeux\in(-1,1)$, the lattice Boltzmann scheme is stable for every value $c\in[0,1]$.
\end{proposition}

\begin{proof}[Proof of Prop.~\ref{thm:stabedgel}]
    In this proof, we assume that $\tun=-1$ and $\vert\tdeux\vert<1$. Lemma~\ref{thm:P3I} then implies that
    $\polyPt(\cdot;\xi,c,-1,\tdeux)\in\ensinegalite$ for $\xi\in[0,\pi]$, $c\in[0,1]$.
    Then, applying Th.~\ref{thm:strikwerda432}, $\polyPt(\cdot;\xi,c,-1,\tdeux)\in\svNp$ if, and only if $\polyPdcs(\cdot;\xi,c,-1,\tdeux)\in\svNp$.
    Lemmas~\ref{thm:P2csI} and \ref{thm:P2csE} read then
    \begin{align*}
        \polyPdcs(\cdot;\xi,c,-1,\tdeux)\in\ensinegalite &\iff (c,\xi)\in\squarecxiinterior\cup\edgescxi^l\cup\edgescxi^t,\\
        \polyPdcs(\cdot;\xi,c,-1,\tdeux)\in\ensegalite &\iff (c,\xi) \in \edgescxi^b\cup\edgescxi^r\cup\verticescxi.
    \end{align*}
    First, we consider $(c,\xi)\in\squarecxiinterior\cup\edgescxi^l\cup\edgescxi^t$, we have
    \begin{equation*}
        \polyPucscs(X;c,\xi,-1,\tdeux) = 
        (1-\cos\xi)(1-c^4)(1-\tdeux^2)^2(
            \alpha_1^{\expcs\expcs} - X \beta_1^{\expcs\expcs}
        )
    \end{equation*}
    with $\mu=(1-c^2)/(1+c^2)\in(0,1]$,
    $\alpha_1^{\expcs\expcs} = \mu+\cos\xi$,
    and
    $\beta_1^{\expcs\expcs} = 1+\mu\cos\xi\geq0$.
    Then $\beta_1^{\expcs\expcs}\geq|\alpha_1^{\expcs\expcs}|$ as $(1-\mu)(1-\cos\xi)\geq0$ and $(1+\mu)(1+\cos\xi)\geq0$, which implies $\polyPucscs(X;c,\xi,-1,\tdeux)\in\svNp$.
    Second, we consider $(c,\xi) \in \edgescxi^b\cup\edgescxi^r\cup\verticescxi$, we have
    \begin{equation*}
        \polyPucsde(X;c,\xi,-1,\tdeux) = (1-\tdeux^2)(2X-1+c^2-(1+c^2)\cos\xi).
    \end{equation*}
    As $-2\leq -2c^2\leq 1-c^2+(1+c^2)\cos\xi\leq 2$, we obtain that $\polyPucsde(X;c,\xi,-1,\tdeux)\in\svNp$.
    We conclude that $\polyPt(\cdot;\xi,c,-1,\tdeux)\in\svNp$ for every $(c,\xi)\in\squarecxi$ and then that the scheme is stable for $c\in[0,1]$.
\end{proof}

\begin{proposition}[stability on $\edgestt^t$]
    \label{thm:stabedget}
    Considering $\tun\in(-1,1)$ and $\tdeux=1$, the lattice Boltzmann scheme is stable for $c=0$ and $c=1$ but unstable for every other values $c\in(0,1)$ (see Fig.~\ref{fig:edges_scxi_t}).
\end{proposition}

\begin{figure}[htb]
    \includegraphics[width=.5\textwidth]{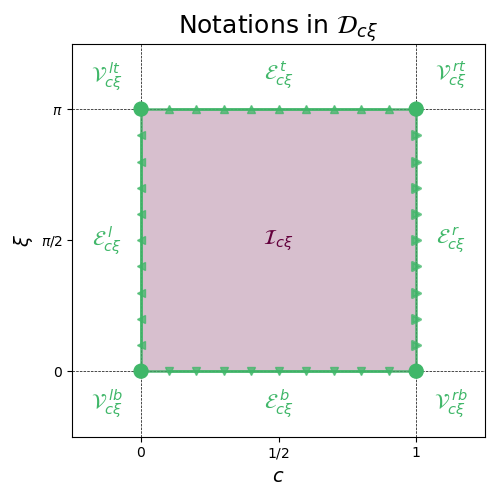}
    \caption{Stability for $(\tun,\tdeux)\in\edgestt^{t}$.}
    \label{fig:edges_scxi_t}
\end{figure}

\begin{proof}[Proof of Prop.~\ref{thm:stabedget}]
    In this proof, we assume that $\vert\tun\vert<1$ and $\tdeux=1$. Lemma~\ref{thm:P3I} then implies that
    $\polyPt(\cdot;\xi,c,\tun,1)\in\ensinegalite$ for $\xi\in[0,\pi]$, $c\in[0,1]$.
    Then, applying Th.~\ref{thm:strikwerda437}, $\polyPt(\cdot;\xi,c,\tun,1)\in\vNp$ if, and only if $\polyPdcs(\cdot;\xi,c,\tun,1)\in\vNp$.
    Lemmas~\ref{thm:P2csI} and \ref{thm:P2csE} read then
    \begin{align*}
        \polyPdcs(\cdot;\xi,c,\tun,1)\in\ensinegalite &\iff (c,\xi)\in\squarecxiinterior\cup\edgescxi^l,\\
        \polyPdcs(\cdot;\xi,c,\tun,1)\in\ensegalite &\iff (c,\xi) \in \edgescxi^b\cup\edgescxi^r\cup\edgescxi^t\cup\verticescxi.
    \end{align*}

    First, we consider $(c,\xi)\in\squarecxiinterior\cup\edgescxi^l$, we have
    \begin{equation*}
        \polyPucscs(X;\xi,c,\tun,1) = 
        (1-\tun^2)^2(1-c^2)(1-\cos\xi)\bigl(
            X(1{+}\cos\xi) - (1{+}\cos\xi-2\i c\sin\xi)
        \bigr)
    \end{equation*}
    As $\sin\xi\neq0$ and $\cos\xi\neq-1$ for $(c,\xi)\in\squarecxiinterior$, the root $1-2\i c\sin\xi/(1+\cos\xi)$ is located outside the closed unit circle and $\polyPucscs(\cdot;\xi,c,\tun,1)\notin\vNp$. 
    However, $\polyPucscs(\cdot;\xi,0,\tun,1)\in\svNp$.
    Second, we consider $(c,\xi)\in\edgescxi^{b}\cup\edgescxi^{r}\cup\edgescxi^{t}\cup\verticescxi$, we have
    \begin{align*}
        \polyPucsde(X;0,c,\tun,1) 
        &= 2X(1-\tun^2) \in\svNp,&c\in[0,1],\\
        \polyPucsde(X;\pi,c,\tun,1) 
        &= 2(1-\tun^2)\bigl(X+2c^2-1\bigr) \in\svNp,&c\in[0,1],\\
        \polyPucsde(X;\xi,1,\tun,1) 
        &= (1-\tun^2)(2X+1-\cos\xi+\i\sin\xi) \in\svNp,
        &\xi\in[0,\pi].
    \end{align*}
    We conclude that $\polyPt(\cdot;\xi,c,\tun,1)\in\svNp$ for every $(c,\xi)\in\edgescxi\cup\verticescxi$, and that  $\polyPt(\cdot;\xi,c,\tun,1)\notin\vNp$ for $(c,\xi)\in\squarecxiinterior$. Finally, the scheme is stable for $c=0$ and $c=1$ and is unstable for $c\in(0,1)$.
\end{proof}

\begin{proposition}[stability on $\edgestt^r$]
    \label{thm:stabedger}
    Considering $\tun=1$ and $\tdeux\in(-1,1)$, the lattice Boltzmann scheme is stable for $c=0$ and $c=1$ but unstable for every other values $c\in(0,1)$ (see Fig.~\ref{fig:edges_scxi_r}).
\end{proposition}

\begin{figure}[htb]
    \includegraphics[width=.5\textwidth]{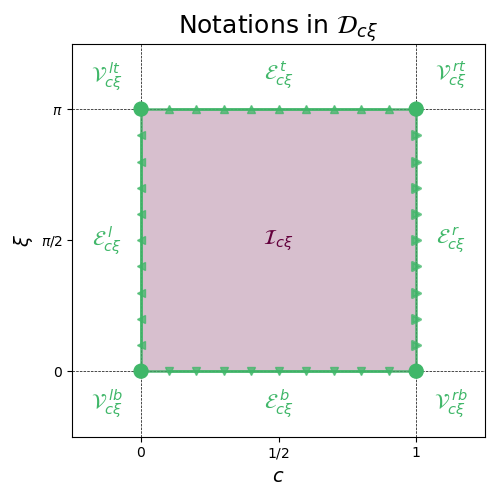}
    \caption{Stability for $(\tun,\tdeux)\in\edgestt^{r}$.}
    \label{fig:edges_scxi_r}
\end{figure}

\begin{proof}[Proof of Prop.~\ref{thm:stabedger}]
    We assume in this proof that $\tun=1$ and $\vert\tdeux\vert<1$. Lemma~\ref{thm:P3I} then implies that
    $\polyPt(\cdot;\xi,c,1,\tdeux)\in\ensinegalite$ for $\xi\in[0,\pi]$, $c\in[0,1]$.
    Then, applying Th.~\ref{thm:strikwerda432}, $\polyPt(\cdot;\xi,c,1,\tdeux)\in\svNp$ if, and only if $\polyPdcs(\cdot;\xi,c,1,\tdeux)\in\svNp$.
    Lemmas~\ref{thm:P2csI} and \ref{thm:P2csE} read then
    \begin{align*}
        \polyPdcs(\cdot;\xi,c,1,\tdeux)\in\ensinegalite &\iff (c,\xi)\in\squarecxiinterior\cup\edgescxi^l\cup\edgescxi^t,\\
        \polyPdcs(\cdot;\xi,c,1,\tdeux)\in\ensegalite &\iff (c,\xi) \in \edgescxi^b\cup\edgescxi^r\cup\verticescxi.
    \end{align*}
    First, we consider $(c,\xi)\in\squarecxiinterior\cup\edgescxi^l\cup\edgescxi^t$, we have
    {\small\begin{align*}
        \polyPucscs(X;c,\xi,1,\tdeux) &= 
        (1-\cos\xi)(1-c^2)(1-\tdeux^2)^2 \\
        &\quad \bigl(
            (1+\cos\xi+c^2(1-\cos\xi))(X-1)+2\i c\sin\xi
        \bigr) \notin\vNp,
        &&(c,\xi)\in\squarecxiinterior,\\
        \polyPucscs(X;0,\xi,1,\tdeux) &= 
        (1-\cos\xi)(1-\tdeux^2)^2
        (1+\cos\xi)(X-1) \in\svNp, &&\xi\in(0,\pi),\\
        \polyPucscs(X;c,\pi,1,\tdeux) &= 
        4c^2(1-c^2)(1-\tdeux^2)^2(X-1) \in\svNp, &&c\in(0,1).
    \end{align*}}
    Second, we consider $(c,\xi) \in \edgescxi^b\cup\edgescxi^r\cup\verticescxi$, we have
    \begin{equation*}
        \polyPucsde(X;c,\xi,1,\tdeux) = (1-\tdeux^2)(2X-(1-c^2)(1-\cos\xi)+2\i c\sin\xi).
    \end{equation*}
    Then
    \begin{align*}
        \polyPucsde(X;c,0,1,\tdeux) 
        &= 2(1-\tdeux^2)X \in\svNp,
        & c\in[0, 1],\\
        \polyPucsde(X;1,\xi,1,\tdeux) 
        &= 2(1-\tdeux^2)(X+\i\sin\xi) \in\svNp,
        &\xi\in[0,\pi],\\
        \polyPucsde(X;0,\pi,1,\tdeux) 
        &= 2(1-\tdeux^2)(X-1) \in\svNp.
    \end{align*}
    We conclude that $\polyPt(\cdot;\xi,c,1,\tdeux)\in\svNp$ for every $(c,\xi)\in\edgescxi\cup\verticescxi$ and that $\polyPt(\cdot;\xi,c,1,\tdeux)\notin\vNp$ for $(c,\xi)\in\squarecxiinterior$. Finally, the scheme is stable for $c=0$ and $c=1$ and is unstable for $c\in(0,1)$.
\end{proof}

%%%%%%%%%%%%%%%%%%%%%%%%%%%%%%%%%%%%
% stabilité aux sommets
%%%%%%%%%%%%%%%%%%%%%%%%%%%%%%%%%%%%
\subsection{Stability in \texorpdfstring{$\verticestt$}{the vertices}}
%%%%%%%%%%%%%%%%%%%%%%%%%%%%%%%%%%%%

In this section, we consider the four vertices 
$\verticestt^{lb}$, $\verticestt^{rb}$, $\verticestt^{lt}$, and $\verticestt^{rt}$.
We prove that (see Fig.~\ref{fig:stability_vertices}):
\begin{itemize}
    \item for $\tun=-1$ and $\tdeux=-1$, the scheme is stable for $c\in[0,1]$ (see Prop.~\ref{thm:stabvertexlb});
    \item for $\tun=1$ and $\tdeux=-1$, the scheme is stable for $c=0$ and unstable for $c\in(0,1]$ (see Prop.~\ref{thm:stabvertexrb});
    \item for $\tun=-1$ and $\tdeux=1$, the scheme is stable for $c\in(0,1)$ and unstable for $c\in\lbrace0,1\rbrace$ (see Prop.~\ref{thm:stabvertexlt});
    \item for $\tun=1$ and $\tdeux=1$, the scheme is unstable for $c\in[0,1]$ (see Prop.~\ref{thm:stabvertexrt}).
\end{itemize}

\begin{figure}[htb]
    \includegraphics[width=.33\textwidth]{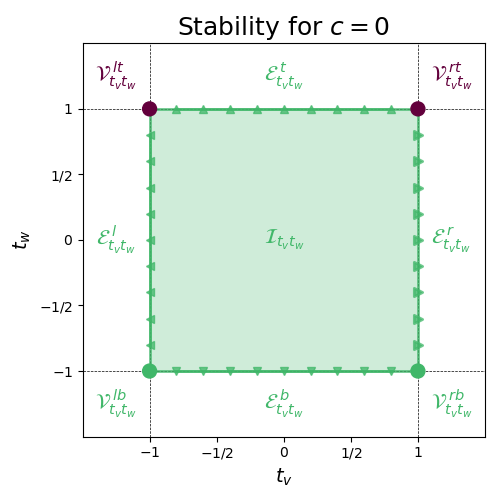}%
    \hfill%
    \includegraphics[width=.33\textwidth]{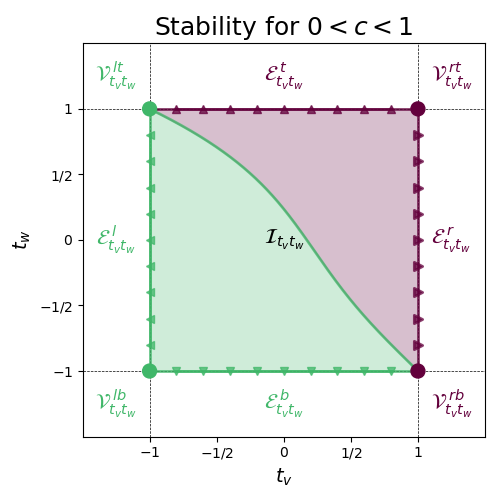}%
    \hfill%
    \includegraphics[width=.33\textwidth]{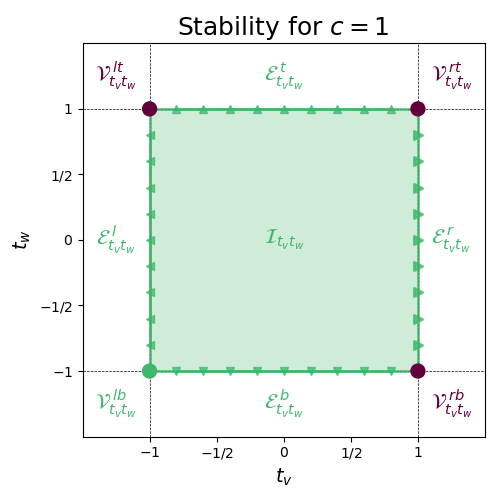}%
    \caption{Stability for $(\tun,\tdeux)\in\verticestt$. Green means stable  while purple means unstable.}
    \label{fig:stability_vertices}
\end{figure}

\begin{proposition}[stability on $\verticestt^{lb}$]
    \label{thm:stabvertexlb}
Considering $\tun=-1$ and $\tdeux=-1$,
    the lattice Boltzmann scheme is stable for every value $c\in[0,1]$.
\end{proposition}

\begin{proof}[Proof of Prop.~\ref{thm:stabvertexlb}]
    Assuming $\tun=-1$ and $\tdeux=-1$, the amplification matrix $\matG$ reads 
    \begin{equation*}
        \matG = \begin{pNiceArray}{ccc}
        1 & 0 & 0 \\
        0 & e^{{-}\i\xi} & 0 \\
        0 & 0 & e^{\i\xi}
        \end{pNiceArray}.
    \end{equation*}
    The scheme is then clearly stable
    % The minimal polynomial $\polymini$ of $\matG$ is then clearly a simple von Neumann polynomial
    as the matrix $\matG$ is diagonal for every $\xi\in[0,\pi]$ with eigenvalues on the unit circle.
    % \begin{equation*}
    %     \polymini(X;\xi,c,-1,-1)= \left\lbrace
    %     \begin{aligned}
    %         &(1-X)(e^{{-}\i\xi}-X)(e^{\i\xi}-X)
    %         &&\text{if } \xi\notin\lbrace 0,\pi\rbrace,\\
    %         &1-X&&\text{if } \xi=0,\\
    %         &(1-X)(1+X)&&\text{if }\xi=\pi.
    %     \end{aligned}\right.
    % \end{equation*}
\end{proof}

\begin{proposition}[stability on $\verticestt^{rb}$]
    \label{thm:stabvertexrb}
    Considering $\tun=1$ and $\tdeux=-1$, 
    the lattice Boltzmann scheme is stable for $c=0$ and unstable for $c\in(0,1]$ (see Fig.~\ref{fig:vertices_scxi_rb}).
\end{proposition}

\begin{figure}[htb]
    \includegraphics[width=.5\textwidth]{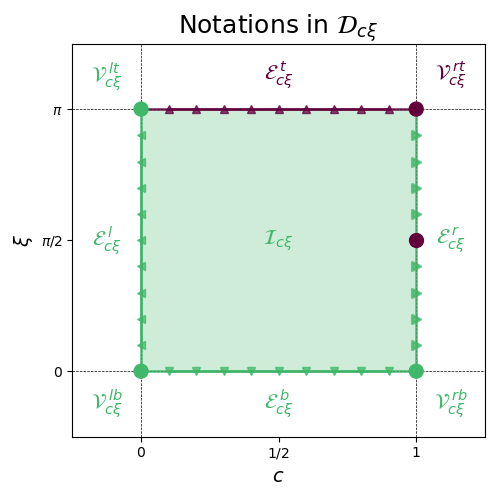}
    \caption{Stability for $(\tun,\tdeux)\in\verticestt^{rb}$.}
    \label{fig:vertices_scxi_rb}
\end{figure}

\begin{proof}[Proof of Prop.~\ref{thm:stabvertexrb}]
    Assuming $\tun=1$ and $\tdeux=-1$, 
    the characteristic polynomial given in \eqref{eq:defP3} yields to
    % \begin{equation*}
    %     \polyPt(X;\xi,c,1,-1) = X^3 
    %     - X^2(1 - 2\i c\sin\xi) - X(1+2\i c\sin\xi)
    %     + 1 
    % \end{equation*}
    \begin{align*}
        \polyPt(X;\xi,c,1,-1) &= X^3 
        - X^2(1 - 2\i c\sin\xi) - X(1+2\i c\sin\xi)
        + 1 \\
        &= (X-1)(X^2 + 2\i c\sin\xi X -1)
    \end{align*}
    for which the roots are $1$, $\sqrt{1-c^2\sin^2\xi} + \i c\sin\xi$, and $-\sqrt{1-c^2\sin^2\xi} + \i c\sin\xi$. Then $\polyPt\in\vNp$ for $(c,\xi)\in\squarecxi$ and
    \begin{equation*}
        \polyPt\in\svNp \iff
        (c,\xi)\in \squarecxiinterior \cup \edgescxi^{r} \backslash\lbrace (1,\pi/2)\rbrace.
    \end{equation*}
    % On doit traiter a part
    % c = 0 => 1 est racine double
    % xi = 0 => 1 est racine double
    % xi = pi => 1 est racine double
    % c = 1 et xi = pi/2 => i est racine double
    For $(c,\xi)\in\edgescxi^{l}\cup\verticescxi^{lt}\cup\verticescxi^{lb}$, that is $c=0$ and $0\leq\xi\leq\pi$, the minimal polynomial of the amplification matrix
    \begin{equation*}
        % sun = 2, sdeux = 0, c = 0
        \matG = \begin{pNiceArray}{ccc}
            1 & 0 & 0 \\
            0 & 0 & e^{{-}\i\xi} \\
            0 & e^{\i\xi} & 0
        \end{pNiceArray}
    \end{equation*}
    reads $\polymini = X^2-1\in\svNp$. The scheme is then stable for $c=0$.
    For $(c,\xi)\in\edgescxi^{b}\cup\verticescxi^{lb}\cup\verticescxi^{rb}$, that is $0\leq c\leq 1$ and $\xi=0$, the minimal polynomial of the amplification matrix
    \begin{equation*}
        % sun = 2, sdeux = 0, xi = 0
        \matG = \begin{pNiceArray}{ccc}
            1 & 0 & 0 \\
            c & c & 1+c \\
            -c & 1-c & -c
        \end{pNiceArray}
    \end{equation*}
    reads also $\polymini = X^2-1\in\svNp$.
    For $(c,\xi)\in\edgescxi^{t}\cup\verticescxi^{rt}$, that is $0<c\leq 1$ and $\xi=\pi$, the minimal polynomial of the amplification matrix
    \begin{equation*}
        % sun = 2, sdeux = 0, xi = pi
        \matG = \begin{pNiceArray}{ccc}
            1 & 0 & 0 \\
            -c & -c & -1-c \\
            c & -1+c & c
        \end{pNiceArray}
    \end{equation*}
    reads $\polymini = (X-1)(X^2-1)\notin\svNp$. The last case is then $c=1$ and $\xi=\pi/2$ for which the minimal polynomial of the matrix
    \begin{equation*}
        % sun = 2, sdeux = 0, c = 1, xi = pi/2
        \matG = \begin{pNiceArray}{ccc}
            1 & 0 & 0 \\
            -\i & -\i & -2\i \\
            -\i & 0 & -\i
        \end{pNiceArray}
        %%%        /  1    0    0  \
        %%% G^2 = | -i-3  -1   -4  |
        %%%        \-i-1   0   -1  /
        %%% G^2-(1+i)G+iI = -i+3 +i(1+i) = 
    \end{equation*}
    reads $\polymini = (X-1)(X+\i)^2 \notin \svNp$ as $\matG^2$ is not equal to $(1-\i)\matG+\i \matId$. The scheme is then unstable for $c>0$.
\end{proof}

\begin{proposition}[stability on $\verticestt^{lt}$]
    \label{thm:stabvertexlt}
    Considering $\tun=-1$ and $\tdeux=1$, 
    the lattice Boltzmann scheme is stable for every value $c\in(0,1)$ and unstable for $c\in\lbrace0,1\rbrace$
    (see Fig.~\ref{fig:vertices_scxi_lt}).
\end{proposition}

\begin{figure}[htb]
    \includegraphics[width=.5\textwidth]{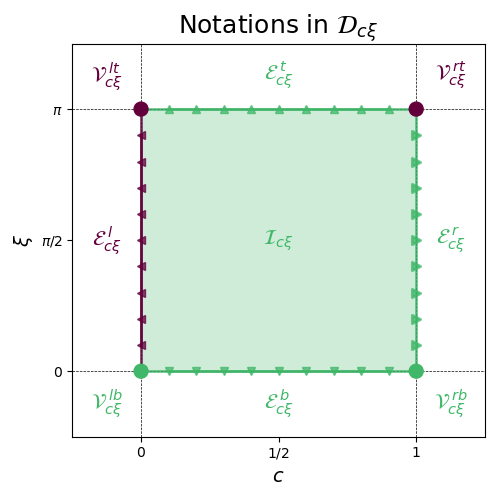}
    \caption{Stability for $(\tun,\tdeux)\in\verticestt^{lt}$.}
    \label{fig:vertices_scxi_lt}
\end{figure}

\begin{proof}[Proof of Prop.~\ref{thm:stabvertexlt}]
    Assuming $\tun=-1$ and $\tdeux=1$, the characteristic polynomial given in \eqref{eq:defP3} yields to
    \begin{align*}
        \polyPt(X;\xi,c,-1,1) &= X^3 
        + X^2(-1+2c^2(1-\cos\xi))
        + X(-1+2c^2(1-\cos\xi))
        + 1 \\
        &= (1+X)(1-2X(1-c^2(1-\cos\xi))+X^2)
    \end{align*}
    for which the roots are $-1$, $\alpha+\i\sqrt{1-\alpha^2}$ and $\alpha-\i\sqrt{1-\alpha^2}$ with $\alpha=1-c^2(1-\cos\xi)$. Then $\polyPt\in\vNp$ for $(c,\xi)\in\squarecxi$ and
    \begin{equation*}
        \polyPt\in\svNp \iff \alpha^2\neq1 \iff
        (c,\xi) \in \squarecxiinterior \cup \edgescxi^{t} \cup \edgescxi^r.
    \end{equation*}
    We first investigate the stability for $c=0$.
    For $(c,\xi)\in\edgescxi^{l}\cup\verticescxi^{lt}$, that is $c=0$ and $0<\xi\leq\pi$, the minimal polynomial of the amplification matrix
    \begin{equation*}
        \matG = \begin{pNiceArray}{ccc}
            1 & 2 & 2 \\
            0 & 0 & -e^{{-}\i\xi} \\
            0 & -e^{\i\xi} & 0
        \end{pNiceArray}
    \end{equation*}
    reads $\polymini=(X-1)^2(X+1)$ as $\matG^2$ is not equal to $\matId$:
    \begin{equation*}
        \matG^2 = \begin{pNiceArray}{ccc}
            1 & 2(1-e^{\i\xi}) & 2(1-e^{{-}\i\xi}) \\
            0 & 1 & 0 \\
            0 & 0 & 1
        \end{pNiceArray}.
    \end{equation*}
    This shows that the scheme is unstable for $c=0$. We remark that the matrix $\matG$ reduces to identity for $\xi=0$ but it is evidently not enough to impose the stability of the scheme.

    Second, we investigate the stability for $c=1$. The necessary condition involving the characteristic polynomial is fulfilled for $0<\xi<\pi$ ($(c,\xi)\in\edgescxi^{r}$). We remark first that for $(c,\xi)\in\verticescxi^{rb}$, that is $c=1$ and $\xi=0$, the minimal polynomial of the amplification matrix
    \begin{equation*}
        \matG = \begin{pNiceArray}{ccc}
            -1 & 0 & 0 \\
            1 & 1 & 0 \\
            1 & 0 & 1
        \end{pNiceArray}
    \end{equation*}
    reads $\polymini=X^2-1\in\svNp$ and second that for $(c,\xi)\in\verticescxi^{rt}$, that is $c=1$ and $\xi=\pi$, the minimal polynomial of the amplification matrix
    \begin{equation*}
        \matG = \begin{pNiceArray}{ccc}
            -1 & 0 & 0 \\
            -1 & -1 & 0 \\
            -1 & 0 & -1
        \end{pNiceArray}
    \end{equation*}
    reads $\polymini=(X+1)^2\notin\svNp$, hence the scheme is unstable for $c=1$.

    Finally, we investigate the stability for $0<c<1$. For $(c,\xi)\in\edgescxi^{b}$, that is $0<c<1$ and $\xi=0$, the minimal polynomial of the amplification matrix
    \begin{equation*}
        \matG = \begin{pNiceArray}{ccc}
            1-2c^2 & 2(1-c^2) & 2(1-c^2) \\
            c^2 & c^2 & c^2-1 \\
            c^2 & c^2-1 & c^2
        \end{pNiceArray}
    \end{equation*}
    reads $\polymini=X^2-1\in\svNp$ as $\matG^2=\matId$.
    This means that the necessary condition for stability is fulfilled. To show that the scheme is indeed stable, we argue by a direct computation of $\matG^n$ for $\xi\in[0,\pi]$.
    \begin{equation*}
        \matG = 
        \begin{pNiceArray}{ccc}
            c^2(\cos\xi-\i \sin\xi)& c^2(\cos\xi-\i \sin\xi)& -(1-c^2)(\cos\xi-\i \sin\xi)\\
            c^2(\cos\xi+\i \sin\xi)& -(1-c^2)(\cos\xi+\i \sin\xi)& c^2(\cos\xi+\i \sin\xi)
        \end{pNiceArray}.%,
    \end{equation*}
    This matrix is diagonalisable on $\mathbb{C}$ for every $\xi \in [0, \pi]$ with eigenvalues $\lambda_0=-1$, $\lambda_\pm=\alpha\pm\i\sqrt{1-\alpha^2}$, with $\alpha=1-c^2(1-\cos\xi)\in(-1,1]$.
    Using half and double angle formulas for sine and cosine, we have
    \begin{equation*}
        \alpha=1-2c^2\sin(\xi/2)^2, 
        \qquad
        \lambda_\pm = 1-2c^2\sin(\xi/2)^2 \pm2\i c\sin(\xi/2)\sqrt{1-c^2\sin(\xi/2)^2}.
    \end{equation*}
    Eigenvectors corresponding to these eigenvalues are given by
    \begin{equation*}
        V_0 = \begin{pmatrix}
        -1-e^{-\i\xi}\\%-(1+\cos\xi) + i\sin\xi\\
        e^{-\i\xi}\\%\cos\xi -i\sin\xi\\
        1
        \end{pmatrix}, 
        \quad
        V_\pm = \begin{pmatrix}
            % 2(1-c^2)(\lambda_\pm-e^{-\i\xi})\\
            % (1-c^2)(1-\lambda_\pm)e^{-\i\xi}\\
            % -(1-\lambda_\pm) - c^2e^{-\i\xi} + c^2\lambda_\pm e^{\i\xi}
            2(1-c^2)\\
            c e^{-\i\xi/2}(c\sin(\xi/2)\mp\sqrt{1-c^2\sin(\xi/2)^2})\\
            c e^{-\i\xi/2}(c\sin(\xi/2)\pm\sqrt{1-c^2\sin(\xi/2)^2})
        \end{pmatrix}.
    \end{equation*}
    Defining $\matP$ the change of basis and $\matL$ the diagonal matrix of the eigenvalues
    \begin{equation*}
        \matP = \begin{pNiceArray}{c|c|c}
            \phantom{V}&\phantom{V}&\phantom{V}\\
            V_0 & V_- & V_+\\
            \phantom{V}&\phantom{V}&\phantom{V}
        \end{pNiceArray}
        \qquad
        \matL = \begin{pNiceArray}{ccc}
            \lambda_0 & 0 & 0 \\
            0 & \lambda_- & 0 \\
            0 & 0 & \lambda_+
        \end{pNiceArray},
    \end{equation*}
    we have
    \begin{equation*}
        \matG^n = \frac{1}{\det\matP}\matP\matL^n \operatorname{Com}(\matP)^T.
    \end{equation*}
    The important point is that
    \begin{equation*}
        \det\matP = -8c (1-c^2\sin(\xi/2)^2)^{3/2} e^{-\i\xi/2}
        \implies
        \vert\det\matP\vert
        \geq 8c(1-c^2)^{3/2} > 0,
        %-4\i\sqrt{1-\alpha^2}e^{-\i\xi}(1-e^{-\i\xi})(1-c^2)(2-c^2(1-\cos^2\xi))
    \end{equation*}
    as $0<c<1$.
    This proves that the iteration matrix is uniformly power bounded, and concludes this part of the proof.
\end{proof}

\begin{proposition}[stability on $\verticestt^{rt}$]
    \label{thm:stabvertexrt}
    Considering $\tun=1$ and $\tdeux=1$, 
    the lattice Boltzmann scheme is unstable for $c\in[0,1]$ (see Fig.~\ref{fig:vertices_scxi_rt}).
\end{proposition}

\begin{figure}[htb]
    \includegraphics[width=.5\textwidth]{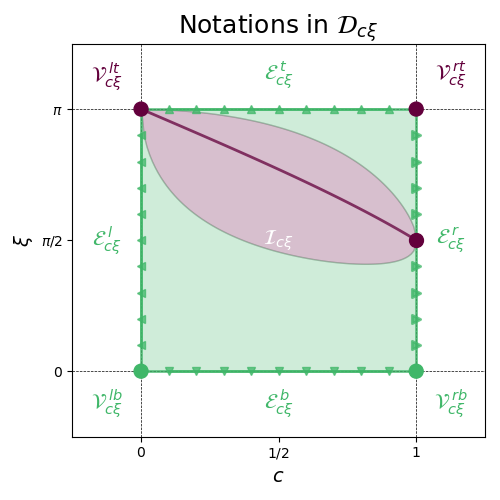}
    \caption{Stability for $(\tun,\tdeux)\in\verticestt^{rt}$.}
    \label{fig:vertices_scxi_rt}
\end{figure}

\begin{proof}[Proof of Prop.~\ref{thm:stabvertexrt}]
    Assuming $\tun=1$ and $\tdeux=1$, the characteristic polynomial given in \eqref{eq:defP3} yields to
    \begin{equation*}
        \polyPt(X;\xi,c,1,1) = X^3 + \alpha X^2 - \bar\alpha X - 1,
    \end{equation*}
    with
    \begin{equation*}
        \alpha = -1+2c^2(1-\cos\xi)+2\cos\xi + 2\i c\sin\xi.
    \end{equation*}
    % As $\cohnschur(\polyPt) = 0$, 
    As $\polyPdcs = 0$ ($\polyPt\in\ensegalite$), 
    we have 
    $\polyPt\in\vNp\iff\polyPdde\in\vNp$ thanks to Th.~\ref{thm:strikwerda437}. Computing $\polyPdde=3X^2+2\alpha X-\bar\alpha$, we obtain that
    \begin{equation*}
        \cohnschur(\polyPdde) = \polyPudecs = (9-|\alpha|^2)X + 2(3\alpha+\bar\alpha^2).
    \end{equation*}
    We then reduce
    \begin{align*}
        \polyPt\in\vNp &\iff
        \polyPdde\in\vNp \iff
        \left\lbrace\begin{aligned}
        &|\alpha|^2<9 \text{ and } 2|3\alpha+\bar\alpha^2| \leq 9-|\alpha|^2\\
        \text{ or }\\
        &|\alpha|^2=9 \text{ and } 3\alpha+\bar\alpha^2=0,
        \end{aligned}\right.\\
        &\iff |\alpha|^2\leq9 \text{ and } 2|3\alpha+\bar\alpha^2| \leq 9-|\alpha|^2.
    \end{align*}
    Thanks to lemma~\ref{thm:techlemtvtwun}, the inequality $\vert\alpha\vert^2\leq9$ is always true.
    We obtain
    \begin{equation*}
        (9-|\alpha|^2)^2 - 4|3\alpha+\bar\alpha^2|^2 = 48 (1-\cos\xi) \beta(c^2,\cos\xi)
    \end{equation*}
    with
    \begin{align*}
        \beta(c^2, C) &=
        (1-c^2)(c^6+5c^4-8c^2+1)
        + C (1-c^2)(-3c^6-5c^4+6c^2+3)\\
        &\quad + C^2(-3c^8+8c^6-13c^4+6c^2+3)
        + C^3(1+c-c^2)^2(1-c-c^2)^2.
    \end{align*}
    Fixing $c\in[0,1]$, we consider $\beta(c^2,C)$ as a polynomial of degree 3 in terms of $C$.
    
    First, fixing $c=0$, we have
    $\beta(0, C)=1+3C+3C^2+C^3=(1+C)^3\geq0$, then $\polyPt\in\vNp$.
    To conclude the case $\edgescxi^l\cup\verticescxi^{lb}\cup\verticescxi^{lt}$, we write the amplification matrix $\matG$:
    \begin{equation*}
        \matG = \begin{pNiceArray}{ccc}
            1 & 2 & 2 \\
            0 & -e^{-\i\xi} & 0 \\
            0 & 0 & -e^{\i\xi}
        \end{pNiceArray}.
    \end{equation*}
    For $\xi\in[0,\pi]$, the minimal polynomial of the amplification matrix reads
    \begin{align*}
        \polymini(X; 0, 0, 1, 1) &= (X-1)(X+1), \\
        \polymini(X; \xi, 0, 1, 1) &
        =(X-1)(X+e^{-\i\xi})(X+e^{\i\xi}), & 0<\xi<\pi,\\
        \polymini(X; \pi, 0, 1, 1)
        &=(X-1)^2.
    \end{align*}
    It is then a simple von Neumann polynomial only for $\xi\in[0,\pi)$. The scheme is then clearly unstable for $c=0$.

    Second, fixing $c=1$, we have
    $\beta(1, C)=C^2(1+C)\geq0$, then $\polyPt\in\vNp$. To conclude the case $\edgescxi^r\cup\verticescxi^{rb}\cup\verticescxi^{rt}$, we write the amplification matrix $\matG$:
    \begin{equation*}
        \matG = \begin{pNiceArray}{ccc}
            -1 & 0 & 0 \\
            2e^{-\i\xi} & e^{-\i\xi} & 2e^{-\i\xi} \\
            0 & 0 & -e^{\i\xi}
        \end{pNiceArray}.
    \end{equation*}
    The minimal polynomial of the amplification matrix reads
    \begin{equation*}
        \polymini(X; \xi, 1, 1, 1)=(X+1)(X-e^{-\i\xi})(X+e^{\i\xi})
    \end{equation*}
    for $\xi\in(0,\pi]$ and $\mu(X;0,1,1,1)=(X-1)(X+1)$. It is then a simple von Neumann polynomial only for $\xi\notin\lbrace\pi/2,\pi\rbrace$. The scheme is then clearly unstable for $c=1$.

    Third, fixing $\xi=0$, the equality $\cos\xi=1$ implies immediately that $\polyPt\in\vNp$. To conclude the case $\edgescxi^b$, we write the amplification matrix $\matG$:
    \begin{equation*}
        \matG = \begin{pNiceArray}{ccc}
            1-2c^2 & 2(1-c^2) & 2(1-c^2) \\
            c(1+c) & -1+c+c^2 & c(1+c) \\
            -c(1-c) & -c(1-c) & -1-c+c^2
        \end{pNiceArray}.
    \end{equation*}
    The characteristic polynomial reads, as $\alpha=1$,
    $\polyPt(X;0,c,1,1)=(X+1)^2(X-1)$. As $\matG^2$ is the identity matrix for $c\in[0,1]$, the minimal polynomial of the amplification matrix is a simple von Neumann polynomial.

    Fourth, fixing $\xi=\pi$, we have 
    $\beta(c^2,-1)= 8c^6(1-c^2)\geq0$, then $\polyPt\in\vNp$. To conclude the case $\edgescxi^t$, we write the amplification matrix $\matG$:
    %       + c^6 + 5c^4 -8c^2 +1
    % - c^8 -5c^6 + 8c^4 - c^2
    %       +3c^6 + 5c^4 -6c^2 -3
    % -3c^8 -5c^6 + 6c^4 +3c^2
    % -3c^8 +8c^6 -13c^4 +6c^2 +3
    % - c^8 +6c^6 -11c^4 +6c^2 -1
    %--------------------------
    % -8c^8 +8c^6
    % 
    %( (1-c^2)^2 - c^2 )^2
    %( 1 -3c^2 + c^4)^2
    %1 + 9c^4 + c^8 -6c^2 + 2c^4 -6c^6
    % c^8 -6c^6 +11c^4 -6c^2 + 1
    \begin{equation*}
        \matG = \begin{pNiceArray}{ccc}
            1-2c^2 & 2(1-c^2) & 2(1-c^2) \\
            -c(1+c) & 1-c-c^2 & -c(1+c) \\
            c(1-c) & c(1-c) & 1+c-c^2
        \end{pNiceArray}.
    \end{equation*}
    The characteristic polynomial reads, as $\alpha=-4(1-c^2)$,
    $\polyPt(X;\pi,c,1,1)=(X-1)(X^2-2(1-2c^2)X+1)$.
    As the polynomial $X^2-2(1-2c^2)X+1$ has two complex conjugate roots of modulus one with nonzero imaginary part for $c\in(0,1)$, the polynomial $\polyPt$ is simple von Neumann for $0<c<1$.

    Finally, we treat the case $(c,\xi)\in\squarecxiinterior$. For a fixed value of $c$, the function $\beta(c^2,C)$ is a polynomial of degree 3 in the $C$ variable with a non negative coefficient of higher degree. 
    We could write the roots and observe that exactly two of them are localized in $(-1,1)$, the third one being in $(-\infty,-1)$. The area where $\beta$ is negative is given in Fig.~\ref{fig:vertices_scxi_rt} for illustration purposes. In order to prove that the scheme is unstable for every $c\in(0,1)$, we only prove that there exists a value of $\xi\in(0,\pi)$ defined by $\cos\xi=-1+c^2$ such that 
    $\beta(c^2,-1+c^2)<0$ for every $c\in(0,1)$. We have
    \begin{align*}
        \beta(c^2,-1+c^2) &= -c^4(1-c^2)
        (c^8-11c^6+38c^4-54c^2+27),\\
        &=-c^4(1-c^2) \left(
            c^8 + 11c^4(1-c^2) + 27(1-c^2)^2
        \right) < 0,
    \end{align*}
    what ends the proof.
\end{proof}

\begin{lemma}[technical lemma]\label{thm:techlemtvtwun}
    Let's $(c,\xi)\in\squarecxi$ defined in \eqref{eq:domaincxi}. In the complex plane, the parameter $\alpha$ defined by
    \begin{equation*}
        \alpha = -1+2c^2(1-\cos\xi)+2\cos\xi + 2\i c\sin\xi,
    \end{equation*}
    lies in the rectangular triangle defined by its three vertices $-3$, $1$ and $1+2\i$ (see Fig.~\ref{fig:alphaleqtrois}). 
    As a consequence, $\vert\alpha\vert\leq3$ and 
    \begin{equation*}
        \vert\alpha\vert=3
        \iff
        \alpha=-3
        \iff
        (c,\xi)\in\verticescxi^{lt} \ 
        (c=0\text{ and }\xi=\pi)
    \end{equation*}
\end{lemma}

\begin{figure}[htb]
    \includegraphics[width=.5\textwidth]{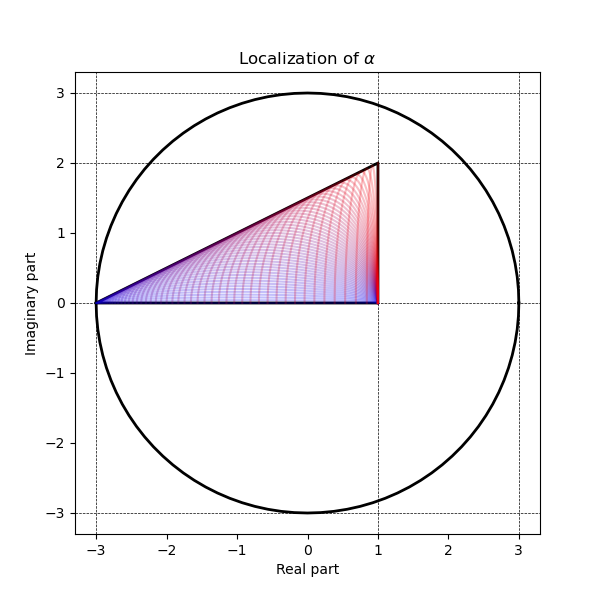}
    \caption{Graphical illustration of lemma~\ref{thm:techlemtvtwun}: the parameter $\alpha$ lies in a rectangular triangle included in the circle of radius $3$.}
    \label{fig:alphaleqtrois}
\end{figure}

\begin{proof}
    The parameter $\alpha$ is in the given rectangular triangle if, and only if 
    \begin{equation*}
        -3\leq\Re(\alpha)\leq1,
        \quad
        0\leq\Im(\alpha)\leq2,
        \quad\text{and}\quad
        2\Im(\alpha)\leq3+\Re(\alpha).
    \end{equation*}
    Rewritting
    $\Re(\alpha) = 1 - 2(1-c^2)(1-\cos\xi)$ 
    implies $-3\leq\Re(\alpha)\leq1$.
    Moreover, $\Im(\alpha) = 2c\sin\xi$ implies $0\leq\Im(\alpha)\leq2$ as $0\leq c\leq 1$ and $0\leq\xi\leq\pi$.
    Finally, we compute
    \begin{equation*}
        3+\Re(\alpha)-2\Im(\alpha) =
        % 4 - 2(1-c^2)(1-\cos\xi) - 4c\sin\xi =
        % 2(1+\cos\xi) -4c\sin\xi c + 2(1-\cos\xi)c^2 =
        2 \bigl(
            \sqrt{1+\cos\xi} - c \sqrt{1-\cos\xi}
        \bigr)^2 \geq 0.
    \end{equation*}
% (1+C) - 2c S + (1-C) c^2
% (1+C) ( 1 - S/(1+C) c )^2
% car S^2/(1+C)^2 = (1-C^2)(1+C)^2 = (1-C)/(1+C)
    To conclude the proof, we observe that there is only one point in the triangle that intersects the circle of radius $3$: the point $-3$. Then $\vert\alpha\vert=3$ is equivalent to $\alpha=-3$. Solving the system $\Re(\alpha)=-3$ and $\Im(\alpha)=0$ yields $c=0$ and $\xi=\pi$.
\end{proof}

%%%%%%%%%%%%%%%%%%%%%%%%%%%%%%%%%%%%%%%%%%%%%%%%%%

%%%%%%%%%%%%%%%%%%%%%%%%%%%%%%%%%%%%%%%%%%%%%%%%%%
% SECTION CONVERGENCE
%%%%%%%%%%%%%%%%%%%%%%%%%%%%%%%%%%%%%%%%%%%%%%%%%%
\section{Convergence and discrete waves}\label{sec:convergence}

In order to avoid dealing with borderline cases, we here consider $(\sun, \sdeux)\in(0, 2)^2$.
The fact of having $\sun\neq0$ and $\sdeux\neq 0$ ensures second-order consistency regardless of the specific values of these parameters.
We also consider parameters where the finite difference scheme whose amplification polynomial is the characteristic polynomial of $\matG$ is $L^2$-stable.
In this way, we can easily re-use known result on this one, see Remark \ref{rem:correspondingFD}.
This gives the following result.

\begin{proposition}[Theorem 10.6.1 in \cite{strikwerda_finite_2004}]\label{prop:consistencyFourier}
    Let $|c|\leq 1$, $0<\sun<2$, and $0<\sdeux\leq\sdeux^\star(\sun)$.
    Then, there exists a unique eigenvalue $z_{\varphi}(\xi)$ defined for $|\xi|\leq \xi_0$ for some $\xi_0>0$ such that $z_{\varphi}(\xi) = 1 - i c \xi + {o}(\xi)$.
    Moreover, there exists a constant $C>0$ such that 
    \begin{equation}\label{eq:consistencyFourier}
        \left | \frac{e^{-ic\xi\spacestep} -z_{\varphi}(\xi\spacestep)}{\timestep}\right |\leq C\spacestep^2 (1+|\xi| )^3.
    \end{equation}
\end{proposition}
\begin{proof}
    All except the power of $1+|\xi|$ is Theorem 10.6.1 in \cite{strikwerda_finite_2004}.
    This power is obtained by Taylor expanding the characteristic equation.
    %  and solve order-by-order, which yields 
    % $z_{\varphi}(\xi) = 1 - i c \xi - \tfrac{1}{2}c^2\xi^2 + \mathcal{O}(\)$.
    % \begin{equation*}        
    %     z_{\varphi}(\xi) = 1 - i c \xi - \tfrac{1}{2}c^2\xi^2 + i c \Bigl ((1-c^2)\bigl ( ( \tfrac{1}{\sun}{-}\tfrac{1}{2})( \tfrac{1}{\sdeux}{-}\tfrac{1}{2}) {-}\tfrac{1}{12}\bigr ) {+} \tfrac{c^2}{6}\Bigr )\xi^3 + \mathcal{O}(\xi^4).
    % \end{equation*}
\end{proof}
Equation \eqref{eq:consistencyFourier} is a way of restating that the scheme is second-order accurate.
Moreover, the power in $\xi$ on the right-hand side of \eqref{eq:consistencyFourier} dictates the minimal smoothness of the initial datum to observe full second-order convergence.
The following convergence theorem is adapted from finite difference, and based on the fact that the scheme behaves essentially as a one-step scheme with symbol $z_{\varphi}(\xi)$, as the other eigenvalues are mastered by stability. 
\begin{theorem}[Convergence]
    Let $|c|\leq 1$, $0<\sun<2$, and $0<\sdeux\leq\sdeux^\star(\sun)$.
    Let $\mz^{\circ}\in H^{\sigma}(\mathbb{R})$ in \eqref{eq:initialcondition} with $\sigma\geq 3$ and $\mz_j^{0}=\mz^{\circ}(j\spacestep)$---which is possible thanks to the smoothness assumption---and initial non-conserved moments at equilibrium, namely
    \begin{equation*}
        \mun_j^0 = \lambda c\, \mz_j^{0},
        \qquad 
        \mdeux_j^0 = \tfrac{1}{2} \lambda^2 c^2\, \mz_j^{0}.
    \end{equation*}
    Then, for all $T>0$, there exists $C_T>0$ such that 
    \begin{equation*}
        \Bigl(
            \spacestep \sum_{j\in\mathbb{Z}} {\bigl\vert\mz(n\timestep, j\spacestep) - \mz_j^n\bigr\vert}^2
        \Bigr)^{1/2}
        \leq C_T \spacestep^2 {\lVert \mz^{\circ} \rVert}_{H^3},
        \qquad 0\leq n\timestep\leq T,
    \end{equation*}
    where $\mz$ is the solution of \eqref{eq:pde} and \eqref{eq:initialcondition}.
\end{theorem}
\begin{proof}
    This is \cite[Theorem 10.6.2]{strikwerda_finite_2004} using Prop. \ref{prop:consistencyFourier} and the fact that initializations at equilibrium yield first-order accurate (see \cite{bellotti2024initialisation}) initialization schemes, given in \eqref{eq:FD} from Remark \ref{rem:correspondingFD}.
    % These initialization schemes read (in Fourier-transformed variables)
    % \begin{equation*}
    %     \hat{\mz}^k(\xi) = \matG^k (\xi)
    %     \begin{pmatrix}
    %         1\\
    %         \lambda c\\
    %         \tfrac{1}{2}\lambda^2 c^2
    %     \end{pmatrix}
    %     \hat{\mz}^0(\xi), \qquad k = 1, 2.
    % \end{equation*}
\end{proof}

In term of expected order of convergence for non-smooth initial data, if $\mz^{\circ}\in H^{\sigma -}(\mathbb{R})$ for the Sobolev regularity, and $\mz^{\circ}\in B_{2, \infty}^{\sigma}(\mathbb{R})$ for the Besov regularity, from \cite{strikwerda_finite_2004} for the Sobolev regularity (up to ``invisible'' logarithmic terms) and from \cite{brenner2006besov}, we expect the empirical order of convergence with respect to the $L^2$ norm in space be $\mathcal{O}(\spacestep^{\min(2, {2}/{3}\sigma)})$.

We finish on wave propagation in the scheme by considering that---if the three roots of the characteristic polynomial that we now call $z_{\varphi}(\xi)$, $z_{\sigma, 1}(\xi)$, and $z_{\sigma, 2}(\xi)$  are distinct for a given $\xi$---any moment (or distribution function)\footnote{Here, we illustrate things for the conserved moment $\mz$.} can be written 
\begin{equation*}
    \hat{\mz}^n(\xi) = z_{\varphi}(\xi)^n C_{\varphi}(\xi) + z_{\sigma, 1}(\xi)^n C_{\sigma, 1}(\xi) +  z_{\sigma, 2}(\xi)^n C_{\sigma, 2}(\xi),
\end{equation*}
where $C_{\varphi}(\xi)$, $C_{\sigma, 1}(\xi)$, and $C_{\sigma, 2}(\xi)$ are determined by the initialization, with the last two being generally small at least close to $\xi = 0$.
This formula expresses that one can interpret the numerical solution as a superposition of the action of three pseudo-schemes, and is therefore interesting to see how each one propagates and damps information, especially for spatially smooth solutions, i.e. in a low-frequency limit.

One can solve the cubic equation with the help of formal computations and perform Taylor expansions for $\xi\to 0$.
However, this is cumbersome with the free parameters $c$, $\sun$, and $\sdeux$.
A more elegant way---up to first order---would boil down to employ the first-order perturbation theory of simple eigenvalues, see \cite[Theorem 1]{greenbaum2020first}.
For second-order result, we could also turn to the procedure by \cite[Chapter II]{kato1995perturbation}.
What we actually employ is injecting a power series expansion in $\xi$ of the eigenvalues into the expansion of the characteristic equation and solve order-by-order, which yields
\begin{align}
    z_{\varphi}(\xi) &= 1 - i c \xi - \tfrac{1}{2}c^2\xi^2 + i c \Bigl ((1-c^2)\bigl ( ( \tfrac{1}{\sun}{-}\tfrac{1}{2})( \tfrac{1}{\sdeux}{-}\tfrac{1}{2}) {-}\tfrac{1}{12}\bigr ) {+} \tfrac{c^2}{6}\Bigr )\xi^3 + \mathcal{O}(\xi^4), \label{eq:taylorPhysicalMode}\\
    z_{\sigma, 1}(\xi) &= (1-\sun)\Bigl (1 + i c \xi + \tfrac{2c^2(\sdeux{-}1) {-} \sun {-} \sdeux {+} 2}{2(\sun {-} \sdeux)}\xi^2 + \mathcal{O}(\xi^3) \Bigr ),\quad \text{and} \notag\\
    z_{\sigma, 2}(\xi) &= (1-\sdeux)\Bigl (1 + \tfrac{(1{-}c^2)({-}\sun {-} \sdeux {+} 2)}{2(\sun {-} \sdeux)}\xi^2+\mathcal{O}(\xi^3) \Bigr ). \notag
\end{align}
The first equation shows that the method is second-order accurate and can become third-order accurate when $ ( \frac{1}{\sun}-\frac{1}{2} ) ( \frac{1}{\sdeux}-\frac{1}{2} ) =\frac{1}{12}$, which is however forbidden by stability reasons.
Moreover, we see that the regular part of the solution carried by the first parasitic mode is geometrically damped if $0<\sun<2$ and propagates at dimensionless velocity $-c$.
The second parasitic mode is damped if $0<\sdeux<2$ and does not propagate.
Notice that taking $\sdeux = 2-\sun$, we obtain $z_{\sigma, 2}\equiv 1-\sdeux = \sun - 1$: this mode can only carries constants.
%%%%%%%%%%%%%%%%%%%%%%%%%%%%%%%%%%%%%%%%%%%%%%%%%%

%%%%%%%%%%%%%%%%%%%%%%%%%%%%%%%%%%%%%%%%%%%%%%%%%%
% SECTION NUMERICAL EXPERIMENTS
%%%%%%%%%%%%%%%%%%%%%%%%%%%%%%%%%%%%%%%%%%%%%%%%%%
\section{Numerical experiments}\label{sec:numericalExperiments}

\subsection{Discrete waves}\label{sec:discreteWaves}

\begin{figure}[htb]
    \includegraphics[width=1\textwidth]{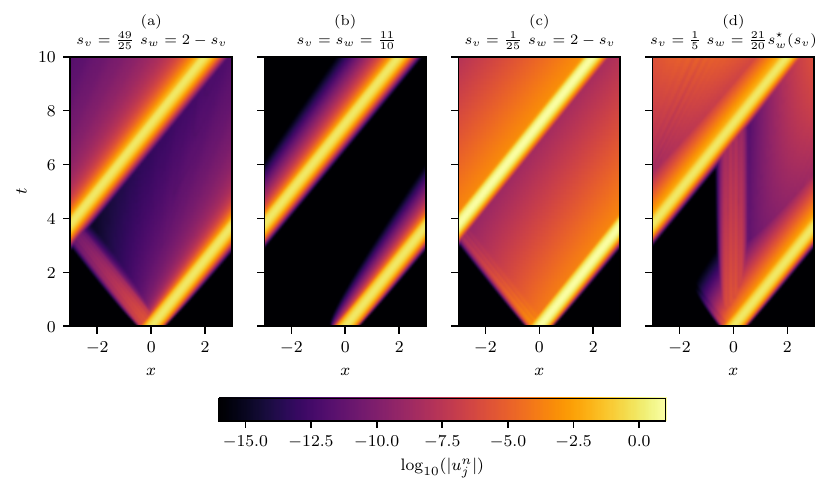}
    \caption{Conserved moment of the scheme as function of space and time for different relaxation parameters.
    In (a), propagation of the mode brought by  $z_{\sigma, 1}$ at speed $-c$ is clearly visible (right-to-left beam), as its damping is weak.
    In (b), the damping of $z_{\sigma, 1}$ and $z_{\sigma, 2}$ is significant, thus only the physical mode $z_{\varphi}$ is observed (left-to-right beam), featuring dispersion.
    In (c), we see that the solution is very inaccurate as $\sun$ is very close to zero.
    Finally, in (d), which is an unstable configuration, we are able to observe the zero-velocity mode brought by $z_{\sigma, 2}$ (vertical beam).}
    \label{fig:exp-waves}
\end{figure}

We now conduct simulations on the domain $[-3, 3]$ using 500 grid-points, endowed with periodic boundary conditions.
We take $\lambda = 1$ and $c=4/5$.
The initial datum is smooth and given by $\mz^{\circ}(x) = \cos(\pi x)^{10}\mathds{1}_{(0, 1/2)}(|x|)$, with non-conserved moments at equilibrium.
We consider different relaxation parameters and the results are in Fig. \ref{fig:exp-waves}.
For the case (a), we see the physical wave propagating at speed $c$ and undergoing dispersion (the milky shadow in the direction opposite to propagation).
A parasitic wave is brought by $z_{\sigma, 1}(\xi)$ propagating at speed $-c$, and experiencing temporal damping, which is slow enough to make it observable using a logarithmic color-scale.
For case (b), we do not observe anything except the physical wave, since the relaxation parameters are pretty close to one.
This entails a quick damping of the parasitic modes and the scheme behaves essentially as a Lax-Wendroff one.
For (c), we see the effect of taking $\sun$ very close to zero.
Eventually, configuration (d) is unstable. Here, the spurious mode brought by $z_{\sigma, 1}(\xi)$ damps away quickly, whereas the one brought by $z_{\sigma, 2}(\xi)$, which does not propagate, is visible for a long time, as $\sdeux$ is quite close to two.

\subsection{Convergence and order}

\begin{figure}[htb]
    \includegraphics[width=.5\textwidth]{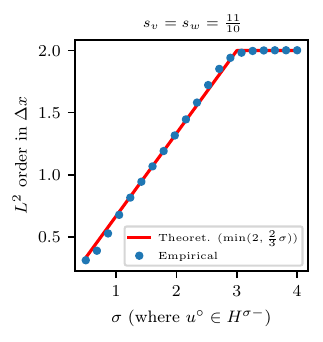}
    \caption{Empirical order of convergence for the $L^2$ error at the final time, as function of the smoothness of the initial datum.}
    \label{fig:convergence}
\end{figure}

Under the same setting as Sec. \ref{sec:discreteWaves} and a final time equal to one, we consider initial data $\mz^{\circ}(x) = \cos(\pi x)^{\sigma-1/2}\mathds{1}_{(0, 1/2)}(|x|)$, so that $\mz^{\circ}\in H^{\sigma -}(\mathbb{R})$ for the Sobolev regularity, and $\mz^{\circ}\in B_{2, \infty}^{\sigma}(\mathbb{R})$ for the Besov regularity. 
We estimate this order by two simulations, one with 2000 grid-points and the other with 4000, yielding the results in Fig. \ref{fig:convergence}.
These are in very good agreement with the theoretical behavior, so that the method achieves full second-order accuracy starting from $\sigma \geq 3$.
Although we present results only for $\sun=\sdeux=\tfrac{11}{10}$, analogous trends of convergence have been obtained for $\sun = \tfrac{49}{25}$ or $\sun = \tfrac{1}{25}$, and $\sdeux = 2-\sun$.

\subsection{Empirical convergence for the Burgers' equation}

\begin{figure}[htb]
    \includegraphics[width=0.66\textwidth]{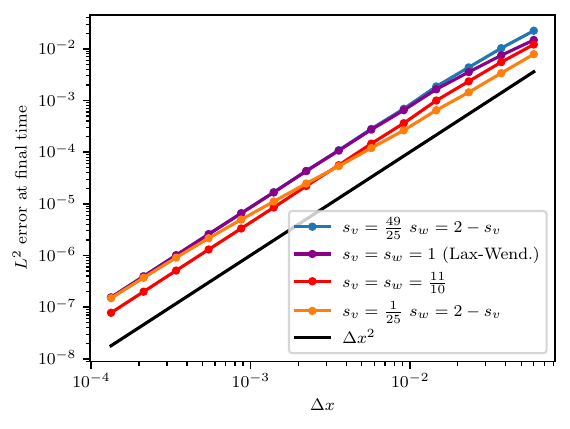}
    \caption{$L^2$ errors at the final time for Burgers' equation.}
    \label{fig:convergence-burgers}
\end{figure}

If we abandon the setting of a linear problem as \eqref{eq:pde}, by taking non-linear equilibria, we can also obtain a formally second-order scheme for the inviscid Burgers' equation
\begin{equation*}
    \partial_t
    \mz
    +
    \partial_{x}
    (\tfrac{1}{2}\mz^2)
    =0.
\end{equation*}
Notice that we do not say anything about its stability, and the scheme can be unstable even if the linearization is stable.
By \cite[Theorem 3.7]{bellotti2023truncation}, the modified equation of the scheme reads 
\begin{equation*}
    \partial_t \mz +\partial_x \muneq(\mz) - \Delta t \Bigl (\frac{1}{2} - \frac{1}{\sun}\Bigr )\partial_x \Bigl ( 2\partial_x\mdeuxeq (\mz)-\frac{\text{d}\muneq(\mz)}{\text{d}\mz} \partial_x\muneq(\mz)\Bigr )
    = \mathcal{O}(\Delta t^2).
\end{equation*}
Formally, we obtain 
\begin{equation*}
    2\partial_x\mdeuxeq (\mz)-\frac{\text{d}\muneq(\mz)}{\text{d}\mz} \partial_x\muneq(\mz) = \Bigl (2\frac{\text{d}\mdeuxeq (\mz)}{\text{d}\mz} - \Bigl ( \frac{\text{d}\muneq(\mz)}{\text{d}\mz}\Bigr )^2\Bigr )\partial_x\mz
\end{equation*}
so that formal second-order consistency with a conservation law with flux $\muneq(\mz)$ can be achieved selecting 
\begin{equation*}
    \mdeuxeq (\mz) = \frac{1}{2}\int^{\mz} \Bigl ( \frac{\text{d}\muneq(\eta)}{\text{d}\eta}\Bigr )^2\text{d}\eta.
\end{equation*}
For the Burgers' equation, this boils down to $\muneq(\mz) = \mz^2/2$ and $\mdeuxeq (\mz) = \mz^3/6$
This scheme is already known in the literature: \cite{coulette2020vectorial} calls it ``Scheme 3'' or ``D1Q3 Lax-Wendroff''.

We now conduct simulations on the domain $[-3, 3]$  endowed with periodic boundary conditions.
We take $\lambda = 1$.
The initial datum is $\mz^{\circ}(x) = \cos(\pi x/4)^{10}\mathds{1}_{(0, 2)}(|x|)$, with non-conserved moments at equilibrium.
The final time is $1/2$, which is slightly before the blowup time.
Results are provided in Fig. \ref{fig:convergence-burgers} and confirm empirical second-order convergence for all considered values of relaxation parameters.
Moreover, we see that the choice $\sun=\sdeux = \tfrac{11}{10}$ yields asymptotically smaller errors compared to other choices, in particular against the Lax-Wendroff scheme, obtained when the relaxation is a projection on the equilibrium.
%%%%%%%%%%%%%%%%%%%%%%%%%%%%%%%%%%%%%%%%%%%%%%%%%%

%%%%%%%%%%%%%%%%%%%%%%%%%%%%%%%%%%%%%%%%%%%%%%%%%%
% SECTION CONCLUSION
%%%%%%%%%%%%%%%%%%%%%%%%%%%%%%%%%%%%%%%%%%%%%%%%%%
\section{Conclusion}
%%%%%%%%%%%%%%%%%%%%%%%%%%%%%%%%%%%%%%%%%%%%%%%%%%

We have conducted a comprehensive stability analysis of a three-velocity lattice Boltzmann scheme for the linear advection equation. By choosing the equilibria to ensure second-order consistency, we have shown that the two relaxation parameters do not affect the order of accuracy of the scheme, but determine its stability properties. The Lax-Wendroff method is recovered as a particular case, corresponding to unit relaxation parameters, thereby establishing a direct connection between the lattice Boltzmann formulation and a classical finite difference scheme.

Our main contribution is a complete characterization, in terms of the Courant number and the two relaxation parameters, of the conditions under which the amplification matrix is uniformly power-bounded with respect to the Fourier frequency.
Therefore, for this stability investigation, we decide not to restrain the relaxation parameters domain to those, for example, ensuring second order accuracy.
We derive explicit necessary and sufficient conditions by combining von Neumann analysis with a detailed investigation of the eigenvalue structure of the amplification matrix. In particular, our analysis goes beyond the location of the eigenvalues in the closed unit disk by accounting for their multiplicities and the associated constraints required for power-boundedness. The use of the Cohn-Schur algorithm makes it possible to obtain an explicit characterization of the admissible parameter set, providing a rigorous and systematic account of the stability properties of the scheme.

By levereging on existent results for finite difference schemes, we have further established convergence of the conserved moment towards the solution of the advection equation for smooth initial data, under equilibrium initialization of the non-conserved moments. The analysis also provides insight into the propagation of discrete waves and the relationship between the lattice Boltzmann dynamics and the underlying macroscopic equation. Numerical experiments support the theoretical findings, illustrate the expected propagation patterns, and confirm second-order convergence for smooth initial data. Additional experiments involving non-smooth data and an adaptation to the inviscid Burgers equation suggest the potential relevance of the proposed framework beyond the linear setting considered in the analysis.

Several directions for future research naturally follow from this work. A first objective is to extend the present analysis to lattice Boltzmann schemes with more velocities and additional relaxation parameters. Further extensions to higher spatial dimensions and more complex target equations would provide a broader framework for investigating the stability properties of lattice Boltzmann methods and their applicability to a wider range of transport phenomena.

\appendix
%%%%%%%%%%%%%%%%%%%%%%%%%%%%%%%%%%%%%%%%%%%%%%%%%%
% APPENDIX: expression of the polynomials
%%%%%%%%%%%%%%%%%%%%%%%%%%%%%%%%%%%%%%%%%%%%%%%%%%
%%%%%%%%%%%%%%%%%%%%%%%%%%%%%%%%%%%%%%%%%%%%%%%%%%
% Annexe: expression of the polynomials
%%%%%%%%%%%%%%%%%%%%%%%%%%%%%%%%%%%%%%%%%%%%%%%%%%
\section{Expression of the polynomials}
\label{sec:exprpoly}

The characteristic polynomial of the amplification matrix $\matG$ reads
\begin{equation}\label{eq:ann_defP3}
\polyPt(X;\xi,c,\tun,\tdeux) = \alphatz + \alphatu  X + \alphatd  X^2 + \alphatt  X^3 ,
\end{equation}
with
\begin{align*}
    \alphatz &= -\tun \tdeux,\\
    \Re(\alphatu ) &=
    \tun \tdeux -(1-\cos\xi)\tun (1+\tdeux) c^2 - \cos\xi(\tun+\tdeux),
    \\
    \Im(\alphatu ) &= c \tdeux (1+\tun) \sin\xi,\\
    \Re(\alphatd ) &= 
    -1+c^2(1+\tdeux)(1 - \cos\xi) + \cos\xi(\tun+\tdeux),
    \\
    \Im(\alphatd ) &= c(1+\tun) \sin\xi,\\
    \alphatt &= 1.
\end{align*}

Applying the Cohn-Schur algorithm on $\polyPt$ yields:
\begin{equation}\label{eq:ann_defP2cs}
    \cohnschur(\polyPt) = \polyPdcs(X;\xi,c,\tun,\tdeux) = \alphadcsz + \alphadcsu X + \alphadcsd X^2,
\end{equation}
with
\begin{align*}
    \Re(\alphadcsz) &=
    -c^2 \tun(1-\tdeux^2)(1-\cos(\xi)) - \cos(\xi)(\tun+\tdeux)(1-\tun\tdeux),\\
    \Im(\alphadcsz) &= c\tdeux(1-\tun^2)\sin(\xi),\\
    \Re(\alphadcsu) &=
    - (1-\tun^2\tdeux^2) + c^2(1+\tdeux)(1-\cos\xi)(1-\tun^2\tdeux)  \\
    & \qquad\qquad\qquad+ \bigl(\tun(1-\tdeux^2) + \tdeux(1-\tun^2)\bigr)\cos\xi,\\
    \Im(\alphadcsu) &= c(1-\tun\tdeux^2)(1+\tun)\sin(\xi),\\
    \alphadcsd &= 1-\tun^2\tdeux^2.
\end{align*}

The first derivative of $\polyPt$ reads:
\begin{equation}\label{eq:ann_defP2de}
    \mathrm{d}\polyPt/\mathrm{d}X=\polyPdde(X;\xi,c,\tun,\tdeux) = \alphaddez + \alphaddeu X + \alphadded X^2 ,
\end{equation}
with
\begin{align*}
    \Re(\alphaddez ) &=
    \tun \tdeux -(1-\cos\xi)\tun (1+\tdeux) c^2 - \cos\xi(\tun+\tdeux),
    \\
    \Im(\alphaddez ) &= c \tdeux (1+\tun) \sin\xi,\\
    \Re(\alphaddeu ) &= 
    -2+2c^2(1+\tdeux)(1 - \cos\xi) + 2\cos\xi(\tun+\tdeux),
    \\
    \Im(\alphaddeu ) &= 2c(1+\tun) \sin\xi,\\
    \alphadded &= 3.
\end{align*}

Applying the Cohn-Schur algorithm on $\polyPdcs$ yields:
\begin{equation}\label{eq:ann_defP1cscs}
    \cohnschur(\polyPdcs) = \polyPucscs(X;\xi,c,\tun,\tdeux) = \alphaucscsz + \alphaucscsu X,
\end{equation}
with
% \begin{align*}
%     \Re(\alphaucscsz) &=
%     - (1-\tun\tdeux)^2 \bigl(
%         (1+\tun\tdeux)^2-\cos^2(\xi)(\tun+\tdeux)^2
%     \bigr)
%     + c^2(1-\cos\xi)\big(\\
%     &\qquad\qquad
%         (1-\tun)\bigl(
%             (1-\tun\tdeux)(1+\tun^2\tdeux^3)
%             + \tun\tdeux(\tdeux-\tun)(1+\tdeux)
%         \bigr)\\
%     &\qquad\qquad
%         + \cos\xi(1+\tun)\bigl(
%             (\tun+\tdeux^2)(1-\tun\tdeux)^2 + \tun\tdeux(1-\tdeux)^2(1+\tun)
%         \bigr)\\
%     &\qquad
%     \big) + c^4(1-\cos\xi)^2(1-\tdeux)\tun(1+\tdeux)^2(1-\tun^2\tdeux)
%     ,\\
%     \Im(\alphaucscsz) &= 
%     c\sin\xi (1+\tun)(1+\tdeux) \big( \\
%     &\qquad\qquad
%         (1-\tun\tdeux)^2(1+\tun\tdeux-\cos\xi(\tun+\tdeux)) \\
%     &\qquad\qquad%\qquad\qquad 
%         - c^2(1-\cos\xi)(
%             \tun(1-\tdeux)(1-\tun\tdeux^2)
%             + \tdeux(1-\tun)(1-\tun^2\tdeux)
%         )\\
%     &\qquad
%         \big),\\
%     \alphaucscsu &= (1-\tun\tdeux)^2\bigl(
%         (1+\tun\tdeux)^2-\cos^2(\xi)(\tun+\tdeux)^2
%     \bigr) \\
%     &\qquad
%     - c^2(1-\cos^2(\xi))\tdeux^2(1-\tun^2)^2
%     - c^4(1-\cos\xi)^2\tun^2(1-\tdeux^2)^2 \\
%     &\qquad
%     - 2 c^2\cos\xi(1-\cos\xi)\tun(1-\tdeux^2)(1-\tun\tdeux)(\tun+\tdeux).
% \end{align*}
{\small
\begin{align*}
    \Re(\alphaucscsz) &=
    - (1-\tun\tdeux)^2 \bigl(
        (1+\tun\tdeux)^2-\cos^2(\xi)(\tun+\tdeux)^2
    \bigr) \\
    %\\
    &\qquad+ c^2(1-\cos\xi)%%\qquad
        \Big[(1-\tun)\bigl(
            (1-\tun\tdeux)(1+\tun^2\tdeux^3)
            + \tun\tdeux(\tdeux-\tun)(1+\tdeux)
        \bigr)\\
    &\qquad\qquad\qquad\qquad
        + \cos\xi(1+\tun)\bigl(
            (\tun+\tdeux^2)(1-\tun\tdeux)^2 + \tun\tdeux(1-\tdeux)^2(1+\tun)
        \bigr)\Big]\\
    &\qquad
    + c^4(1-\cos\xi)^2(1-\tdeux)\tun(1+\tdeux)^2(1-\tun^2\tdeux)
    ,\\
    \Im(\alphaucscsz) &= 
    c\sin\xi (1+\tun)(1+\tdeux) \Big[
        (1-\tun\tdeux)^2(1+\tun\tdeux-\cos\xi(\tun+\tdeux)) \\
    &\qquad\qquad\qquad\qquad 
        - c^2(1-\cos\xi)(
            \tun(1-\tdeux)(1-\tun\tdeux^2)
            + \tdeux(1-\tun)(1-\tun^2\tdeux)
        ) \Big],\\
    \alphaucscsu &= (1-\tun\tdeux)^2\bigl(
        (1+\tun\tdeux)^2-\cos^2(\xi)(\tun+\tdeux)^2
    \bigr) \\
    &\qquad
    - c^2(1-\cos^2(\xi))\tdeux^2(1-\tun^2)^2
    - c^4(1-\cos\xi)^2\tun^2(1-\tdeux^2)^2 \\
    &\qquad
    - 2 c^2\cos\xi(1-\cos\xi)\tun(1-\tdeux^2)(1-\tun\tdeux)(\tun+\tdeux).
\end{align*}}

The first derivative of $\polyPdcs$ reads:
\begin{equation}\label{eq:ann_defP1csde}
    \mathrm{d}\polyPdcs/\mathrm{d}X = \polyPucsde(X;\xi,c,\tun,\tdeux) = \alphaucsdez + \alphaucsdeu X,
\end{equation}
with
\begin{align*}
    \Re(\alphaucsdez) &=
    - (1-\tun^2\tdeux^2) + c^2(1+\tdeux)(1-\cos\xi)(1-\tun^2\tdeux)  \\
    & \qquad\qquad\qquad+ \bigl(\tun(1-\tdeux^2) + \tdeux(1-\tun^2)\bigr)\cos\xi,\\
    \Im(\alphaucsdez) &= c(1-\tun\tdeux^2)(1+\tun)\sin(\xi),\\
    \alphaucsdeu &= 2(1-\tun^2\tdeux^2).
\end{align*}

Applying the Cohn-Schur algorithm on $\polyPdde$ yields:
\begin{equation}\label{eq:ann_defP1decs}
    \cohnschur(\polyPdde) = \polyPudecs(X;\xi,c,\tun,\tdeux) = \alphaudecsz + \alphaudecsu X,
\end{equation}
with
\begin{align*}
    \Re(\alphaudecsz) &=
    -2(1-\cos\xi(\tun+\tdeux))(3-\tun\tdeux + \cos\xi(\tun+\tdeux)) + 2(1-\cos\xi)c^2 \big(\\
    &\qquad\qquad
        3(1+\tdeux) + \cos\xi(\tun+\tdeux^2)(1+\tun) 
        - \tun(1+\tdeux)^2 - \tdeux(1+\tun)^2 \\
    &\qquad
    \big) + 2c^4\tun(1-\cos\xi)^2(1+\tdeux)^2,\\
    \Im(\alphaudecsz) &=
    2c\sin\xi(1+\tun)\bigl(
        3+\tdeux(1+\tun)-
        % (\cos\xi+c^2(1-\cos\xi))
        (c^2+\cos\xi(1-c^2))
        (\tun+\tdeux)(1+\tdeux)
    \bigr),\\
    \alphaudecsu &= 
    9-(\cos\xi(\tun+\tdeux)-\tun\tdeux)^2 
    + (1-\cos\xi)c^2 \big(\\
    &\qquad\qquad
        2\tun(1+\tdeux)(\tdeux\tun-\cos\xi(\tun+\tdeux)) 
        - (1+\cos\xi)\tdeux^2(1+\tun)^2\\
    &\qquad
    ) - c^4(1-\cos\xi)^2(1+\tdeux)^2\tun^2 .
\end{align*}

The first derivative of $\polyPdde$ reads:
\begin{equation}\label{eq:ann_defP1dede}
    \mathrm{d}\polyPdde/\mathrm{d}X = \polyPudede(X;\xi,c,\tun,\tdeux) = \alphaudedez + \alphaudedeu X,
\end{equation}
with
\begin{align*}
    \Re(\alphaudedez ) &= 
    -2+2c^2(1+\tdeux)(1 - \cos\xi) + 2\cos\xi(\tun+\tdeux),
    \\
    \Im(\alphaudedez ) &= 2c(1+\tun) \sin\xi,\\
    \alphaudedeu &= 6.
\end{align*}

%%%%%%%%%%%%%%%%%%%%%%%%%%%%%%%%%%%%%%%%%%%%%%%%%%

%%%%%%%%%%%%%%%%%%%%%%%%%%%%%%%%%%%%%%%%%%%%%%%%%%
% Bibliography
%%%%%%%%%%%%%%%%%%%%%%%%%%%%%%%%%%%%%%%%%%%%%%%%%%
\bibliographystyle{plain}
\bibliography{BCDG}
%%%%%%%%%%%%%%%%%%%%%%%%%%%%%%%%%%%%%%%%%%%%%%%%%%

\end{document}